\documentclass[12pt]{amsart}

\usepackage[pdftex,plainpages=false,hypertexnames=false,pdfpagelabels,breaklinks]{hyperref}
\hypersetup{
  colorlinks, linkcolor={black},
  citecolor={medium-blue}, urlcolor={medium-blue}
}

\usepackage{amsmath}
\usepackage{amssymb}
\usepackage{xcolor}
\usepackage{tikz}
\usetikzlibrary{cd}
\usetikzlibrary{knots}
\usetikzlibrary{calc}

\tikzset{string/.style={ultra thick}}
\tikzset{smallstring/.style={thick,scale=0.75,every node/.style={transform shape}}}

\usepackage{graphicx}
\usepackage{fullpage}
\usepackage[T1]{fontenc}
\usepackage{microtype}
\usepackage{ifthen}
\usepackage{libertine}
\usepackage[libertine]{newtxmath}
\usepackage[all]{xy}

\usepackage{xcolor}
\definecolor{dark-red}{rgb}{0.7,0.25,0.25}
\definecolor{dark-blue}{rgb}{0.15,0.15,0.55}
\definecolor{medium-blue}{rgb}{0,0,0.65}
\definecolor{DarkGreen}{RGB}{0,150,0}

\newcommand{\arxiv}[1]{\href{https://arxiv.org/abs/#1}{\small  arXiv:#1}}
\newcommand{\doi}[1]{\href{https://dx.doi.org/#1}{{\small DOI:#1}}}
\newcommand{\euclid}[1]{\href{https://projecteuclid.org/getRecord?id=#1}{{\small  #1}}}
\newcommand{\mathscinet}[1]{\href{https://www.ams.org/mathscinet-getitem?mr=#1}{\small  #1}}
\newcommand{\googlebooks}[1]{(preview at \href{https://books.google.com/books?id=#1}{google books})}

\newcommand{\numdam}[1]{}

\theoremstyle{plain}
\newtheorem{prop}{Proposition}[section]

\newtheorem{thm}[prop]{Theorem}
\newtheorem{lem}[prop]{Lemma}
\newtheorem{cor}[prop]{Corollary}
\newenvironment{rem}{\\ \noindent\textsl{Remark.}}{}  
\numberwithin{equation}{section}

\theoremstyle{remark}
\newtheorem{example}[prop]{Example}

\newtheorem{remark}[prop]{Remark}

\theoremstyle{definition}
\newtheorem{defn}[prop]{Definition}         

\newcommand{\iso}{\cong}
\DeclareMathOperator{\ev}{ev}

\DeclareMathOperator{\Rep}{Rep}

\DeclareMathOperator{\Tr}{Tr}

\newcommand{\id}{\mathbf{1}}

\newcommand{\Set}{{\mathsf {Set}}}		

\definecolor{kwcolor}{rgb}{.15, .5, .8}

\def\semicolon{;}
\def\applytolist#1{
    \expandafter\def\csname multi#1\endcsname##1{
        \def\multiack{##1}\ifx\multiack\semicolon
            \def\next{\relax}
        \else
            \csname #1\endcsname{##1}
            \def\next{\csname multi#1\endcsname}
        \fi
        \next}
    \csname multi#1\endcsname}

\def\calc#1{\expandafter\def\csname c#1\endcsname{{\mathcal #1}}}
\applytolist{calc}QWERTYUIOPLKJHGFDSAZXCVBNM;
\def\bbc#1{\expandafter\def\csname bb#1\endcsname{{\mathbb #1}}}
\applytolist{bbc}QWERTYUIOPLKJHGFDSAZXCVBNM;
\def\bfc#1{\expandafter\def\csname bf#1\endcsname{{\mathbf #1}}}
\applytolist{bfc}QWERTYUIOPLKJHGFDSAZXCVBNM;

\tikzset{anchorbase/.style={baseline={([yshift=-0.5ex]current bounding box.center)}}}
\usetikzlibrary{calc}
\usetikzlibrary{decorations.markings}
\usetikzlibrary{decorations.pathreplacing}
\usetikzlibrary{arrows,shapes,positioning}
\tikzstyle directed=[postaction={decorate,decoration={markings,
    mark=at position #1 with {\arrow{>}}}}]
\tikzstyle rdirected=[postaction={decorate,decoration={markings,
    mark=at position #1 with {\arrow{<}}}}]

\newcommand{\fsphere}[3]{
\draw[white, line width=1mm] (#1-#3,#2) to [out=270,in=180] (#1,#2-#3*2/3) to [out=0,in=270] (#1+#3,#2);
\draw (#1-#3,#2) to [out=270,in=180] (#1,#2-#3*2/3) to [out=0,in=270] (#1+#3,#2);
\draw[thick] (#1,#2) circle (#3);
}

\newcommand{\bsphere}[3]{
\draw[dashed] (#1-#3,#2) to [out=90,in=180] (#1,#2+#3*2/3) to [out=0,in=90] (#1+#3,#2);
}

\newcommand{\fgsphere}[3]{
\draw[white, line width=1mm] (#1-#3,#2) to [out=270,in=180] (#1,#2-#3*2/3) to [out=0,in=270] (#1+#3,#2);
\draw[opacity=.3] (#1-#3,#2) to [out=270,in=180] (#1,#2-#3*2/3) to [out=0,in=270] (#1+#3,#2);
\draw[thick, opacity=.3] (#1,#2) circle (#3);
}

\newcommand{\bgsphere}[3]{
\draw[dashed,opacity=.3] (#1-#3,#2) to [out=90,in=180] (#1,#2+#3*2/3) to [out=0,in=90] (#1+#3,#2);
}

\newcommand{\sphere}[3]{
\bsphere{#1}{#2}{#3}
\fsphere{#1}{#2}{#3}
}
\newcommand{\trefoil}[4]{
\begin{scope}[shift={(#1,#2-5.9)},scale=#3]
\draw[#4] (-0.3,5.15) to [out=150,in=270] (-1,6) to   [out=90,in=180] (0,6.8) to [out=0,in=110] (0.85,6.3)
(-0.7,6) to [out=20,in=160] (1,6) to   [out=340,in=60] (1.4,5) to [out=240,in=330] (0.3,4.85)
(0.85,5.7) to [out=240,in=30] (0,5) to   [out=210,in=300] (-1.4,5) to [out=120,in=200] (-1.15,5.85);
\end{scope}
}

\newcommand{\hopflink}[4]{
\begin{scope}[shift={(#1,#2)},scale=#3]
\begin{scope}[yscale=.66]
\draw[#4]
(.5,0) to [out=90,in=0] (-.5,1) to [out=180,in=90] (-1.5,0);
\draw[white, line width=1mm]
(-.5,0) to [out=90,in=180] (.5,1) to [out=0,in=90] (1.5,0);
\draw[#4]
(-.5,0) to [out=90,in=180] (.5,1) to [out=0,in=90] (1.5,0) to [out=270,in=0] (.5,-1) to [out=180,in=270] (-.5,0);
\draw[white, line width=1mm]
(.5,0) to [out=270,in=0] (-.5,-1);
\draw[#4]
(.5,0) to [out=270,in=0] (-.5,-1) to [out=180,in=270] (-1.5,0);
\end{scope}
\end{scope}
}

\newcommand{\unknot}[4]{
\begin{scope}[shift={(#1-1,#2)},scale=#3]
\begin{scope}[yscale=.66]
\draw[#4]
(-.5,0) to [out=90,in=180] (.5,1) to [out=0,in=90] (1.5,0) to [out=270,in=0] (.5,-1) to [out=180,in=270] (-.5,0);
\end{scope}
\end{scope}
}

\newcommand{\torushole}[4]{
\draw[#4] (#1-#3/2,#2+#3/6) to [out=300,in=180] (#1,#2-#3/6) to [out=0,in=240] (#1+#3/2,#2+#3/6);
\draw[#4] (#1-#3/3,#2) to [out=45,in=180] (#1,#2+#3/6) to [out=0,in=135] (#1+#3/3,#2);
}

\newcommand{\torus}[4]{
\draw[#4] (#1,#2) ellipse (#3 and #3*2/3);
\torushole{#1}{#2}{#3}{#4}
}

\newcommand{\lasagnafigure}{
\begin{tikzpicture}[anchorbase,scale=.2]
\bsphere{0}{0}{15}
\sphere{1}{-5}{3}
\sphere{10.5}{-.5}{3}
\sphere{-8}{.6}{3}
\trefoil{-9}{1.5}{1}{red,very thick}
\hopflink{10.5}{0}{1}{red,very thick}
\unknot{0.5}{-5}{.75}{red,very thick}
\unknot{3}{-4.5}{.75}{red,very thick}
\draw[orange, thick]
(-10.5,.7) to [out=80,in=290] (-7.3,11)
(-2.9,11) to [out=280, in=180] (-1,9) to [out=0,in=260] (.7,11)
(5.2,10.7) to [out=250, in=180] (6,8) to [out=0,in=260] (7.3,10)
(10.3,10) to [out=250, in=90] (9,6) to [out=270,in=110] (12,0)
(9,0) to [out=115, in=0] (0,5) to [out=180,in=70] (-7.5,.7)
(3.1,-4.7) to [out=75, in=0] (-.5,3) to [out=180,in=110] (-.8,-5.2)
(1.5,-4.4) to [out=75, in=0] (-.1,.5) to [out=180,in=110] (.7,-5.2)
;
\torushole{-4}{6}{2}{thick, orange}
\torushole{2}{7}{2}{thick, orange}
\torushole{-.3}{1.5}{2}{thick, orange}
\torus{-5}{-4}{2.3}{thick, orange}
\trefoil{3}{9}{1.5}{red,very thick}
\hopflink{-5}{11}{1.5}{red,very thick}
\unknot{9}{10}{1.5}{red,very thick}
\node at (-12,-2.5) {\tiny${\color{red}L_1}{\subset} S_1$};
\node at (-2,-8.5) {\tiny${\color{red}L_2}{\subset} S_2$};
\node at (7.5,-3.5) {\tiny${\color{red}L_3}{\subset} S_3$};
\node at (-2,13) {\tiny${\color{red}L}{\subset} S$};
\node at (-1.5,6) {\tiny${\color{orange}\Sigma}$};
\node at (1.5,1) {\tiny${\color{orange}\Sigma}$};
\node at (-3.4,-4) {\tiny${\color{orange}\Sigma}$};
\fsphere{0}{0}{15}
\end{tikzpicture}
}

\newcommand{\lasagnafillingfigure}[1]{
    \begin{tikzpicture}[anchorbase,scale=#1]
        \begin{scope}[shift={(-18,-1)},xscale=3,yscale=12]
            \draw[dashed] (0,1) to [out=0,in=90] (1,0) to [out=270,in=0] (0,-1)
            (0,1) to [out=180,in=90] (-1,0) to [out=270,in=180] (0,-1);
        \end{scope}
        \begin{scope}[shift={(-18,-1)},xscale=9,yscale=12]
            \draw[thick] (0,-1) to [out=15,in=180] (2,-1.3) to [out=0,in=270] (3.8,0)
            to [out=90,in=0] (2,1.3) to [out=180,in=340] (0,1);
        \end{scope}
        \bsphere{-4}{-4}{10}
        \fsphere{-4}{-4}{10}
        \sphere{10.5}{-.5}{3}
        \draw[white, line width=1mm]
        (-8,4.5) to [out=80,in=290] (-7.3,11)
        (-2.9,11) to [out=280, in=180] (-1,9) to [out=0,in=260] (.7,11)
        (5.2,10.7) to [out=250, in=180] (6,8) to [out=0,in=260] (7.3,10)
        (10.3,10) to [out=250, in=90] (9,6) to [out=270,in=110] (12,0)
        (9,0) to [out=115, in=0] (0,5) to [out=180,in=20] (-3.3,4.5)
        ;
        \draw[orange, thick]
        (-8,4.5) to [out=80,in=290] (-7.3,11)
        (-2.9,11) to [out=280, in=180] (-1,9) to [out=0,in=260] (.7,11)
        (5.2,10.7) to [out=250, in=180] (6,8) to [out=0,in=260] (7.3,10)
        (10.3,10) to [out=250, in=90] (9,6) to [out=270,in=110] (12,0)
        (9,0) to [out=115, in=0] (0,5) to [out=180,in=20] (-3.3,4.5)
        ;
        \torushole{2}{7}{2}{thick, orange}
        \hopflink{10.5}{0}{1}{red,very thick}
        \begin{scope}[shift={(-6.9,4.5)},xscale=2.3,yscale=.4]
            \draw[red, very thick]
            (-.5,0) to [out=90,in=180] (.5,1) to [out=0,in=90] (1.5,0) to [out=270,in=0] (.5,-1) to [out=180,in=270] (-.5,0);
        \end{scope}
        \trefoil{3}{9}{1.5}{red,very thick}
        \hopflink{-5}{11}{1.5}{red,very thick}
        \unknot{9}{10}{1.5}{red,very thick}
        \node at (8,-8) {\small$F_1$};
        \node at (-8,3.4) {\tiny$v_1$};
        \node at (10.5,-1.5) {\tiny$v_k$};
        \end{tikzpicture}
\quad \sim \quad
\begin{tikzpicture}[anchorbase,scale=#1]
    \begin{scope}[shift={(-18,-1)},xscale=3,yscale=12]
        \draw[dashed] (0,1) to [out=0,in=90] (1,0) to [out=270,in=0] (0,-1)
        (0,1) to [out=180,in=90] (-1,0) to [out=270,in=180] (0,-1);
    \end{scope}
    \begin{scope}[shift={(-18,-1)},xscale=9,yscale=12]
        \draw[thick] (0,-1) to [out=15,in=180] (2,-1.3) to [out=0,in=270] (3.8,0)
        to [out=90,in=0] (2,1.3) to [out=180,in=340] (0,1);
    \end{scope}
    \bgsphere{-4}{-4}{10}
    \fgsphere{-4}{-4}{10}
    \sphere{1}{-5}{3}
    \sphere{10.5}{-.5}{3}
    \sphere{-8}{.6}{3}
    \draw[white, line width=1mm]
    (-9.5,.7) to [out=80,in=290] (-7.3,11)
    (-2.9,11) to [out=280, in=180] (-1,9) to [out=0,in=260] (.7,11)
    (5.2,10.7) to [out=250, in=180] (6,8) to [out=0,in=260] (7.3,10)
    (10.3,10) to [out=250, in=90] (9,6) to [out=270,in=110] (12,0)
    (9,0) to [out=115, in=0] (0,5) to [out=180,in=70] (-6.5,.7)
    (3.1,-4.7) to [out=75, in=0] (-.5,3) to [out=180,in=110] (-.8,-5.2)
    (1.5,-4.4) to [out=75, in=0] (-.1,.5) to [out=180,in=110] (.7,-5.2)
    ;
    \draw[orange, thick]
    (-9.5,.7) to [out=80,in=290] (-7.3,11)
    (-2.9,11) to [out=280, in=180] (-1,9) to [out=0,in=260] (.7,11)
    (5.2,10.7) to [out=250, in=180] (6,8) to [out=0,in=260] (7.3,10)
    (10.3,10) to [out=250, in=90] (9,6) to [out=270,in=110] (12,0)
    (9,0) to [out=115, in=0] (0,5) to [out=180,in=70] (-6.5,.7)
    (3.1,-4.7) to [out=75, in=0] (-.5,3) to [out=180,in=110] (-.8,-5.2)
    (1.5,-4.4) to [out=75, in=0] (-.1,.5) to [out=180,in=110] (.7,-5.2)
    ;
    \torushole{2}{7}{2}{thick, orange}
    \torushole{-.3}{1.5}{2}{thick, orange}
    \torus{-5}{-4}{2.3}{thick, orange}
    \trefoil{-8}{1.5}{1}{red,very thick}
    \hopflink{10.5}{0}{1}{red,very thick}
    \unknot{0.5}{-5}{.75}{red,very thick}
    \unknot{3}{-4.5}{.75}{red,very thick}
    \begin{scope}[shift={(-6.9,4.5)},xscale=2.3,yscale=.4]
        \draw[red, very thick, opacity=.5]
        (-.5,0) to [out=90,in=180] (.5,1) to [out=0,in=90] (1.5,0) to [out=270,in=0] (.5,-1) to [out=180,in=270] (-.5,0);
    \end{scope}
    \trefoil{3}{9}{1.5}{red,very thick}
    \hopflink{-5}{11}{1.5}{red,very thick}
    \unknot{9}{10}{1.5}{red,very thick}
    \node at (8,-8) {\small$F_2$};
    \node at (-8,-8) {\small$F_3$};
    \node at (-8,-.5) {\tiny$v_i$};
    \node at (1.6,-6) {\tiny$v_j$};
    \node at (10.5,-1.5) {\tiny$v_k$};
    \end{tikzpicture}
    }

\newcommand{\parlines}[3]{
\def\sx{#1}
\def\sy{#2}
\def\leng{#3}
\draw[thick] (\sx+0.25,\sy) to (\sx+0.25,\sy+\leng) (\sx+1.25,\sy) to (\sx+1.25,\sy+\leng);
\draw[dotted] (\sx+0.75,\sy) to (\sx+0.75,\sy+\leng);
}

\newcommand{\braidfg}{
\begin{tikzpicture}[anchorbase,scale=.2]
\fatbraid{0}{0}
\begin{scope}[shift={(0,-1.5)}]
\draw (0,0) rectangle (1.5,1.5) ;
\draw (2,0) rectangle (3.5,1.5) ;
\node at (2.75,0.75) {\tiny $g$};
\node at (.75,0.75) {\tiny $f$};
\parlines{0}{-.25}{.25}
\parlines{2}{-.25}{.25}
\end{scope}
\end{tikzpicture}
\to
\begin{tikzpicture}[anchorbase,scale=.2]
\fatbraid{0}{0}
\begin{scope}[shift={(0,4)}]
\draw (0,0) rectangle (1.5,1.5) ;
\draw (2,0) rectangle (3.5,1.5) ;
\node at (2.75,0.75) {\tiny $f$};
\node at (.75,0.75) {\tiny $g$};
\parlines{0}{1.5}{.25}
\parlines{2}{1.5}{.25}
\end{scope}
\end{tikzpicture}
}

\newcommand{\capp}[3]{
\def\sx{#1}
\def\vert{#2-4}
\def\hor{#3}
\draw[thick] (\sx+1.25,4+\vert) to [out=90,in=180] (\sx+1.75,4.25+\vert) to (\sx+1.75+\hor,4.25+\vert)  to [out=0,in=90] (\sx+2.25+\hor,4+\vert);
\draw[dotted]  (\sx+0.75,4+\vert) to [out=90,in=180] (\sx+1.75,4.5+\vert) to (\sx+1.75+\hor,4.5+\vert)  to [out=0,in=90] (\sx+2.75+\hor,4+\vert) ;
\draw[thick] (\sx+0.25,4+\vert) to [out=90,in=180] (\sx+1.75,4.75+\vert) to (\sx+1.75+\hor,4.75+\vert)  to [out=0,in=90] (\sx+3.25+\hor,4+\vert);
}

\newcommand{\Rtilde}{
\begin{tikzpicture}[anchorbase,scale=.2]
\fatbraid{0}{0}
\fatbraid{2}{4}
\parlines{0}{4}{4}
\parlines{4}{0}{4}
\end{tikzpicture}
\to
\begin{tikzpicture}[anchorbase,scale=.2]
\def\sx{0}
\def\sy{0}
\draw[thick]
(\sx+2.25,\sy) to (\sx+2.25,\sy+1)to [out=90,in=270] (\sx+.25,\sy+7) to (\sx+.25,\sy+8)
(\sx+3.25,\sy) to (\sx+3.25,\sy+1) to [out=90,in=270] (\sx+1.25,\sy+7) to (\sx+1.25,\sy+8) ;
\draw[dotted] (\sx+2.75,\sy) to (\sx+2.75,\sy+1) to [out=90,in=270] (\sx+.75,\sy+7) to (\sx+.75,\sy+8);
\begin{scope}[shift={(2,0)}]
\draw[thick]
(\sx+2.25,\sy) to (\sx+2.25,\sy+1)to [out=90,in=270] (\sx+.25,\sy+7) to (\sx+.25,\sy+8)
(\sx+3.25,\sy) to (\sx+3.25,\sy+1) to [out=90,in=270] (\sx+1.25,\sy+7) to (\sx+1.25,\sy+8) ;
\draw[dotted] (\sx+2.75,\sy) to (\sx+2.75,\sy+1) to [out=90,in=270] (\sx+.75,\sy+7) to (\sx+.75,\sy+8);
\end{scope}
\draw[white, line width=.1cm]
(\sx+0.25,\sy) to (\sx+0.25,\sy+3.3)to [out=90,in=270] (\sx+4.25,\sy+5.3) to (\sx+4.25,\sy+8)
(\sx+1.25,\sy) to (\sx+1.25,\sy+2.7) to [out=90,in=270] (\sx+5.25,\sy+4.7) to (\sx+5.25,\sy+8) ;
\draw[dotted] (\sx+0.75,\sy) to (\sx+0.75,\sy+3) to [out=90,in=270] (\sx+4.75,\sy+5) to (\sx+4.75,\sy+8);
\draw[thick]
(\sx+0.25,\sy) to (\sx+0.25,\sy+3.3)to [out=90,in=270] (\sx+4.25,\sy+5.3) to (\sx+4.25,\sy+8)
(\sx+1.25,\sy) to (\sx+1.25,\sy+2.7) to [out=90,in=270] (\sx+5.25,\sy+4.7) to (\sx+5.25,\sy+8) ;
\draw[dotted] (\sx+0.75,\sy) to (\sx+0.75,\sy+3) to [out=90,in=270] (\sx+4.75,\sy+5) to (\sx+4.75,\sy+8);
\end{tikzpicture}
}

\newcommand{\fatbraid}[2]{
\def\sx{#1}
\def\sy{#2}
\draw[thick]
(\sx+2.25,\sy) to (\sx+2.25,\sy+.3)to [out=90,in=270] (\sx+.25,\sy+3.3) to (\sx+.25,\sy+4)
(\sx+3.25,\sy) to (\sx+3.25,\sy+.7) to [out=90,in=270] (\sx+1.25,\sy+3.7) to (\sx+1.25,\sy+4) ;
\draw[dotted] (\sx+2.75,\sy) to (\sx+2.75,\sy+.5) to [out=90,in=270] (\sx+.75,\sy+3.5) to (\sx+.75,\sy+4);
\draw[white, line width=.1cm]
(\sx+0.25,\sy) to (\sx+0.25,\sy+1.8)to [out=90,in=270] (\sx+2.25,\sy+2.8) to (\sx+2.25,\sy+4)
(\sx+1.25,\sy) to (\sx+1.25,\sy+1.2) to [out=90,in=270] (\sx+3.25,\sy+2.2) to (\sx+3.25,\sy+4) ;
\draw[dotted] (\sx+0.75,\sy) to (\sx+0.75,\sy+1.5) to [out=90,in=270] (\sx+2.75,\sy+2.5) to (\sx+2.75,\sy+4);
\draw[thick]
(\sx+0.25,\sy) to (\sx+0.25,\sy+1.8)to [out=90,in=270] (\sx+2.25,\sy+2.8) to (\sx+2.25,\sy+4)
(\sx+1.25,\sy) to (\sx+1.25,\sy+1.2) to [out=90,in=270] (\sx+3.25,\sy+2.2) to (\sx+3.25,\sy+4) ;
\draw[dotted] (\sx+0.75,\sy) to (\sx+0.75,\sy+1.5) to [out=90,in=270] (\sx+2.75,\sy+2.5) to (\sx+2.75,\sy+4);
}

\newcommand{\Trrrr}[8]{
\def\varA{#3}
\def\varB{#4}
\def\boxf{#6}
\def\boxx{#7}
\def\topp{#8}
\def\vert{\boxf+\boxx+\topp}
\def\hor{#5}
\begin{scope}[shift={(\varA,\varB)}]
\draw[thick] (.25,1.5+\boxf) to (.25,2.5+\boxf+\boxx) (1.25,1.5+\boxf) to (1.25,2.5+\boxf+\boxx);
\draw[dotted] (.75,1.5+\boxf) to (.75,2.5+\boxf+\boxx);
\draw[thick] (1.25,4+\boxf+\boxx) to (1.25,4+\vert) to [out=90,in=180] (1.75,4.25+\vert) to (1.75+\hor,4.25+\vert)  to [out=0,in=90] (2.25+\hor,4+\vert)  to (2.25+\hor,0) to [out=270,in=0] (1.75+\hor,-.25) to (1.75,-.25) to [out=180,in=270] (1.25,0) to (1.25,\boxf);
\draw[dotted] (0.75,4+\boxf+\boxx) to (0.75,4+\vert) to [out=90,in=180] (1.75,4.5+\vert) to (1.75+\hor,4.5+\vert)  to [out=0,in=90] (2.75+\hor,4+\vert)  to (2.75+\hor,0) to [out=270,in=0] (1.75+\hor,-.5) to (1.75,-.5) to [out=180,in=270] (.75,0) to (.75,\boxf);
\draw[thick]  (0.25,4+\boxf+\boxx) to (0.25,4+\vert) to [out=90,in=180] (1.75,4.75+\vert) to (1.75+\hor,4.75+\vert)  to [out=0,in=90] (3.25+\hor,4+\vert)  to (3.25+\hor,0) to [out=270,in=0] (1.75+\hor,-.75) to (1.75,-.75) to [out=180,in=270] (.25,0) to (.25,\boxf);
\draw (0,0+\boxf) rectangle (1.5,1.5+\boxf) ;
\draw (0,2.5+\boxf+\boxx) rectangle (1.5,4+\boxf+\boxx) ;
\node at (.75,0.75+\boxf) {\tiny $#1$};
\node at (.75,3.25+\boxf+\boxx) {\tiny \scalebox{1}[-1]{$#2$}};
\end{scope}
}

\newcommand{\Trf}[1]{
\begin{tikzpicture}[anchorbase,scale=.25]
\def\vert{-2.5}
\def\hor{0}
\draw[thick] (1.25,4+\vert) to [out=90,in=180] (1.75,4.25+\vert) to (1.75+\hor,4.25+\vert)  to [out=0,in=90] (2.25+\hor,4+\vert)  to (2.25+\hor,0) to [out=270,in=0] (1.75+\hor,-.25) to (1.75,-.25) to [out=180,in=270] (1.25,0);
\draw[dotted] (0.75,4+\vert) to [out=90,in=180] (1.75,4.5+\vert) to (1.75+\hor,4.5+\vert)  to [out=0,in=90] (2.75+\hor,4+\vert)  to (2.75+\hor,0) to [out=270,in=0] (1.75+\hor,-.5) to (1.75,-.5) to [out=180,in=270] (.75,0);
\draw[thick] (0.25,4+\vert) to [out=90,in=180] (1.75,4.75+\vert) to (1.75+\hor,4.75+\vert)  to [out=0,in=90] (3.25+\hor,4+\vert)  to (3.25+\hor,0) to [out=270,in=0] (1.75+\hor,-.75) to (1.75,-.75) to [out=180,in=270] (.25,0);
\draw (0,0) rectangle (1.5,1.5) ;
\node at (.75,0.75) {\tiny $#1$};
\end{tikzpicture}
}

\newcommand{\Trrr}[6]{
\def\varA{#3}
\def\varB{#4}
\def\vert{#6}
\def\hor{#5}
\begin{scope}[shift={(\varA,\varB)}]
\draw[thick] (.25,1.5) to (.25,2.5+\vert) (1.25,1.5) to (1.25,2.5+\vert);
\draw[dotted] (.75,1.5) to (.75,2.5+\vert);
\draw[thick] (1.25,4+\vert) to [out=90,in=180] (1.75,4.25+\vert) to (1.75+\hor,4.25+\vert)  to [out=0,in=90] (2.25+\hor,4+\vert)  to (2.25+\hor,0) to [out=270,in=0] (1.75+\hor,-.25) to (1.75,-.25) to [out=180,in=270] (1.25,0);
\draw[dotted] (0.75,4+\vert) to [out=90,in=180] (1.75,4.5+\vert) to (1.75+\hor,4.5+\vert)  to [out=0,in=90] (2.75+\hor,4+\vert)  to (2.75+\hor,0) to [out=270,in=0] (1.75+\hor,-.5) to (1.75,-.5) to [out=180,in=270] (.75,0);
\draw[thick] (0.25,4+\vert) to [out=90,in=180] (1.75,4.75+\vert) to (1.75+\hor,4.75+\vert)  to [out=0,in=90] (3.25+\hor,4+\vert)  to (3.25+\hor,0) to [out=270,in=0] (1.75+\hor,-.75) to (1.75,-.75) to [out=180,in=270] (.25,0);
\draw (0,0) rectangle (1.5,1.5) ;
\draw (0,2.5+\vert) rectangle (1.5,4+\vert) ;
\node at (.75,0.75) {\tiny $#1$};
\node at (.75,3.25+\vert) {\tiny \scalebox{1}[-1]{$#2$}};
\end{scope}
}

\newcommand{\Trr}[4]{
\Trrr{#1}{#2}{#3}{#4}{0}{0}
}

\newcommand{\hcomp}{
\begin{tikzpicture}[anchorbase,scale=.25]
\def\varA{-4}
\def\varB{0}
\draw[thick] (\varA+.25,\varB+1.5) to (\varA+.25,\varB+2.5) (\varA+1.25,\varB+1.5) to (\varA+1.25,\varB+2.5);
\draw[dotted] (\varA+.75,\varB+1.5) to (\varA+.75,\varB+2.5);
\draw[thick] (\varA+1.25,\varB+4) to [out=90,in=180] (\varA+1.75,\varB+4.25)  to [out=0,in=90] (\varA+2.25,\varB+4)  to (\varA+2.25,\varB+2.5) to [out=270,in=180] (\varA+2.25+.5,\varB+2.1) to (\varA+2.25+2.5,\varB+2.1) to [out=0,in=270] (\varA+2.25+3,\varB+2.5)
(\varA+2.25+3,\varB+1.5) to [out=90,in=0] (\varA+2.25+2.5,\varB+1.9) to (\varA+2.25+.5,\varB+1.9) to  [out=180,in=90]
(\varA+2.25,\varB+1.5) to  (\varA+2.25,\varB+0) to [out=270,in=\varB+0] (\varA+1.75,\varB-.25) to [out=180,in=270] (\varA+1.25,\varB+0);
\draw[dotted] (\varA+0.75,\varB+4) to [out=90,in=180] (\varA+1.75,\varB+4.5)  to [out=0,in=90] (\varA+2.75,\varB+4)  to (\varA+2.75,\varB+2.5) to [out=270,in=180] (\varA+2.75+1,\varB+2.25) to [out=0,in=270] (\varA+2.75+2,\varB+2.5)
(\varA+2.75+2,\varB+1.5) to [out=90,in=0] (\varA+2.75+1,\varB+1.75) to [out=180,in=90]
(\varA+2.75,\varB+1.5) to (\varA+2.75,\varB+0) to [out=270,in=0] (\varA+1.75,\varB-.5) to [out=180,in=270] (\varA+.75,\varB+0);
\draw[thick] (\varA+0.25,\varB+4) to [out=90,in=180] (\varA+1.75,\varB+4.75)  to [out=0,in=90] (\varA+3.25,\varB+4)  to (\varA+3.25,\varB+2.5) to [out=270,in=180] (\varA+3.25+.5,\varB+2.25) to [out=0,in=270] (\varA+3.25+1,\varB+2.5)
(\varA+3.25+1,\varB+1.5) to [out=90,in=0] (\varA+3.25+.5,\varB+1.75) to [out=180,in=90]
(\varA+3.25,\varB+1.5) to (\varA+3.25,\varB+0) to [out=270,in=0] (\varA+1.75,\varB-.75) to [out=180,in=270] (\varA+.25,\varB+0);
\draw (\varA+0,\varB+0) rectangle (\varA+1.5,\varB+1.5) ;
\draw (\varA+0,\varB+2.5) rectangle (\varA+1.5,\varB+4) ;
\node at (\varA+.75,\varB+0.75) {\tiny $k$};
\node at (\varA+.75,\varB+3.25) {\tiny \scalebox{1}[-1]{$l$}};
\draw[thick] (1.25,4) to [out=90,in=180] (1.75,4.25)  to [out=0,in=90] (2.25,4)  to (2.25,0) to [out=270,in=0] (1.75,-.25) to [out=180,in=270] (1.25,0);
\draw[dotted] (0.75,4) to [out=90,in=180] (1.75,4.5)  to [out=0,in=90] (2.75,4)  to (2.75,0) to [out=270,in=0] (1.75,-.5) to [out=180,in=270] (.75,0);
\draw[thick] (0.25,4) to [out=90,in=180] (1.75,4.75)  to [out=0,in=90] (3.25,4)  to (3.25,0) to [out=270,in=0] (1.75,-.75) to [out=180,in=270] (.25,0);
\draw (0,0) rectangle (1.5,1.5) ;
\draw (0,2.5) rectangle (1.5,4) ;
\node at (.75,0.75) {\tiny $f$};
\node at (.75,3.25) {\tiny \scalebox{1}[-1]{$g$}};
\end{tikzpicture}
}

\newcommand{\hcomptwo}{
\begin{tikzpicture}[anchorbase,scale=.25]
\draw[thick] (.25,3) to (.25,4) (1.25,3) to (1.25,4);
\draw[thick] (1.25,7) to [out=90,in=180] (1.75,7.25)  to [out=0,in=90] (2.25,7)  to (2.25,0) to [out=270,in=0] (1.75,-.25) to [out=180,in=270] (1.25,0);
\draw[dotted] (0.75,7) to [out=90,in=180] (1.75,7.5)  to [out=0,in=90] (2.75,7)  to (2.75,0) to [out=270,in=0] (1.75,-.5) to [out=180,in=270] (.75,0);
\draw[thick] (0.25,7) to [out=90,in=180] (1.75,7.75)  to [out=0,in=90] (3.25,7)  to (3.25,0) to [out=270,in=0] (1.75,-.75) to [out=180,in=270] (.25,0);
\draw (0,0) rectangle (1.5,1.5) ;
\draw (0,1.5) rectangle (1.5,3) ;
\draw (0,4) rectangle (1.5,5.5) ;
\draw (0,5.5) rectangle (1.5,7) ;
\node at (.75,0.75) {\tiny $\overline{f}$};
\node at (.75,2.25) {\tiny $\overline{k}$};
\node at (.75,4.75) {\tiny \scalebox{1}[-1]{$\overline{l}$}};
\node at (.75,6.25) {\tiny \scalebox{1}[-1]{$\overline{g}$}};
\end{tikzpicture}
}

\newcommand{\Trfg}[2]{
\begin{tikzpicture}[anchorbase,scale=.25]
\Trr{#1}{#2}{0}{0}
\end{tikzpicture}
}

\newcommand{\movieconstant}{
\draw [thick, directed=.55] (3.5,3.5) to (3.5,0);
\draw [thick, directed=.55] (4.5,3.5) to (4.5,0);
\draw [thick] (3.5,0) to [out=270,in=0] (3,-.5) to [out=180,in=270] (2.5,0)to (2.5,0.5) (2.5,3) to (2.5,3.5) to [out=90,in=180] (3,4) to [out=0,in=90] (3.5,3.5);
\draw [thick,dotted] (4,3.5) to (4,0) to [out=270,in=0] (3,-1) to [out=180,in=270] (2,0) to (2,0.5) (2,3) to (2,3.5) to [out=90,in=180] (3,4.5) to [out=0,in=90] (4,3.5);
\draw [thick] (4.5,0) to [out=270,in=0] (3,-1.5) to [out=180,in=270] (1.5,0) to (1.5,0.5) (1.5,3) to (1.5,3.5) to [out=90,in=180] (3,5) to [out=0,in=90] (4.5,3.5);
\draw (.75,.5) rectangle (2.75,3);
\node at (1.75,1.25) {$\beta$};
\draw [white]  (2,-2.25) to (2,-2.24);
\draw [white]  (.5,6.5) to (.5,6.4);
}
\newcommand{\sixfoam}{
\includegraphics[width=\linewidth]{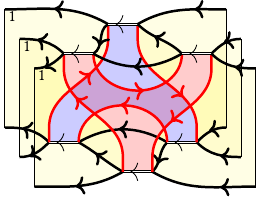}
}

\newcommand{\neckcutexample}{
\begin{tikzpicture}[anchorbase,scale=.5]
	\fill[yellow, opacity=0.1] (1,1) to [out=270, in=180] (2,0.5) to [out=0, in=270] (3,1) to (3,4.5) to [out=270, in=0] (2,4) to [out=180, in=270] (1,4.5) to (1,1);
	\fill[yellow, opacity=0.1] (1,1) to [out=90, in=180] (2,1.5) to [out=0, in=90] (3,1) to (3,4.5) to [out=90, in=0] (2,5) to [out=180, in=90] (1,4.5) to (1,1);
	\draw [very thick] (1,4.5) to [out=270, in=180] (2,4) to [out=0, in=270] (3,4.5);
	\draw [very thick] (1,4.5) to [out=90, in=180] (2,5) to [out=0, in=90] (3,4.5);
	\draw [very thick] (1,1) to [out=270, in=180] (2,0.5) to [out=0, in=270] (3,1);
	\draw [thick, dotted] (1,1) to [out=90, in=180] (2,1.5) to [out=0, in=90] (3,1);
	\draw [thick] (1,1) to (1,4.5);
	\draw [thick] (3,1) to (3,4.5);
	\node at (2,4.5) {\tiny $1$};
\end{tikzpicture}
\quad = \sum_{a+b=N-1}\;\;
\begin{tikzpicture}[anchorbase,scale=.5]
	\fill[yellow, opacity=0.1] (1,1) to [out=270, in=180] (2,0.5) to [out=0, in=270] (3,1) to [out=90, in=0] (2,2.5) to [out=180, in=90] (1,1);
	\fill[yellow, opacity=0.1] (1,1) to [out=90, in=180] (2,1.5) to [out=0, in=90] (3,1) to [out=90, in=0] (2,2.5) to [out=180, in=90] (1,1);
	\fill[yellow, opacity=0.1](1,4.5) to [out=270, in=180] (2,3) to [out=0, in=270] (3,4.5) to [out=90, in=0] (2,5) to [out=180, in=90] (1,4.5) ;
	\fill[yellow, opacity=0.1](1,4.5) to [out=270, in=180] (2,3) to [out=0, in=270] (3,4.5) to [out=270, in=0] (2,4) to [out=180, in=270] (1,4.5);
	\draw [very thick] (1,1) to [out=270, in=180] (2,0.5) to [out=0, in=270] (3,1);
	\draw [thick, dotted] (1,1) to [out=90, in=180] (2,1.5) to [out=0, in=90] (3,1);
	\draw [thick] (1,1) to [out=90, in=180] (2,2.5) to [out=0, in=90] (3,1);
	\draw [very thick] (1,4.5) to [out=270, in=180] (2,4) to [out=0, in=270] (3,4.5);
	\draw [very thick] (1,4.5) to [out=90, in=180] (2,5) to [out=0, in=90] (3,4.5);
	\draw [thick] (1,4.5) to [out=270, in=180] (2,3) to [out=0, in=270] (3,4.5);
	\node at (2,1.75) {\tiny $1$};
	\node at (2,4.5) {\tiny $1$};
	\node at (2,1) {$\bullet^a$};
	\node at (2,3.5) {$\bullet^b$};
\end{tikzpicture}
\;\;, \qquad
\begin{tikzpicture}[anchorbase,scale=.4]
	\fill [yellow, opacity=0.2] (2.5,-2) to [out=270, in=180] (4,-3.5) to [out=0, in=270] (5.5,-2) to [out=90, in=00] (4,-0.5) to [out=180, in=90] (2.5,-2);
    \draw [thick] (2.5,-2) to [out=270, in=180] (4,-3.5) to [out=0, in=270] (5.5,-2);
    \draw [thick] (2.5,-2) to [out=90, in=180] (4,-0.5) to [out=0, in=90] (5.5,-2);
    \draw [thick] (2.5,-2) to [out=315, in=180] (4,-2.5) to [out=0, in=225] (5.5,-2);
    \draw [thick, dotted] (2.5,-2) to [out=45, in=180] (4,-1.5) to [out=0, in=135] (5.5,-2);
    \node at (4,-1) {\tiny $1$};
    \node at (4,-2) {$\bullet^n$};
\end{tikzpicture}
\quad = \quad \delta_{n,N-1}
}

\newcommand{\movA}[1]{
\begin{tikzpicture}[anchorbase,scale=#1]
\movieconstant
\draw [thick] (5,0) to (5,-1.5) to [out=270,in=0] (4.5,-2) to (1.5,-2) to [out=180,in=270] (1,-1.5) to (1,0.5) (1,3) to (1,5) to [out=90,in=180] (1.5,5.5) to (4.5,5.5) to [out=0,in=90] (5,5) to (5,0);
\draw[thick, ->] (1,4.49) to (1,4.5);
\end{tikzpicture}
}

\newcommand{\movB}[1]{
\begin{tikzpicture}[anchorbase,scale=#1]
\movieconstant
\draw [thick]  (5,0) to (5,-1.5) to [out=270,in=0] (4.5,-2) to (1.5,-2) to [out=180,in=270] (1,-1.5) to (1,0.5) (1,3) to (1,6) to  [out=90,in=0] (.5,6.5) to [out=180,in=90] (0,6) to [out=270,in=180] (.5,5.5) to (4.5,5.5) to [out=0,in=90] (5,5) to (5,0);
\draw[thick, ->] (1,4.49) to (1,4.5);
\end{tikzpicture}
}

\newcommand{\movC}[1]{
\begin{tikzpicture}[anchorbase,scale=#1]
\movieconstant
\draw [thick] (5,0) to (5,-1.5) to [out=270,in=0] (4.5,-2) to (1.5,-2) to [out=180,in=270] (1,-1.5) to (1,0.5) (1,3) to (1,3.75) to  [out=90,in=0] (.5,4.25) to [out=180,in=90]  (0,3.75) to [out=270,in=180] (.5,3.25) to (3,3.25) to (4.5, 3.25) to [out=0,in=90] (5,2.75) to (5,0);
\draw[thick, ->] (1,-1.01) to (1,-1);
\end{tikzpicture}
}

\newcommand{\movD}[1]{
\begin{tikzpicture}[anchorbase,scale=#1]
\movieconstant
\draw [thick] (5,-.25) to (5,-1.5) to [out=270,in=0] (4.5,-2) to (1.5,-2) to [out=180,in=270] (1,-1.5) to (1,0.5) (1,3) to (1,3.5) to  [out=90,in=0] (.5,4) to [out=180,in=90] (0,3.5) to  (0,.75) to [out=270,in=180] (.5,.25) to (3,.25) to (4.5,.25) to [out=0,in=90]  (5,-.25);
\draw[thick, ->] (0,1.55) to (0,1.54);
\end{tikzpicture}
}

\newcommand{\movE}[1]{
\begin{tikzpicture}[anchorbase,scale=#1]
\movieconstant
\draw [thick] (2,-1.75) to [out=270,in=0]  (1.5,-2.25) to [out=180,in=270] (1,-1.75) to (1,0.5) (1,3) to (1,3.5) to  [out=90,in=0] (.5,4) to [out=180,in=90] (0,3.5) to (0,-.75) to [out=270,in=180] (.5,-1.25) to (1.5,-1.25) to [out=0,in=90]  (2,-1.75);
\draw[thick, ->] (0,1.55) to (0,1.54);
\end{tikzpicture}
}

\newcommand{\movF}[1]{
\begin{tikzpicture}[anchorbase,scale=#1]
\movieconstant
\draw [thick]  (1,3) to (1,3.5) to  [out=90,in=0] (.5,4) to [out=180,in=90] (0,3.5) to (0,0) to [out=270,in=180] (.5,-.5) to [out=0,in=270] (1,0) to (1,.5);
\draw[thick, ->] (0,1.55) to (0,1.54);
\end{tikzpicture}
}

\newcommand{\dotslide}{
\begin{tikzpicture}[anchorbase,scale=.25]
\movieconstant
\draw [thick]  (5,0) to (5,-1.5) to [out=270,in=0] (4.5,-2) to (1.5,-2) to [out=180,in=270] (1,-1.5) to (1,0.5) (1,3) to (1,6) to  [out=90,in=0] (.5,6.5) to [out=180,in=90] (0,6) to [out=270,in=180] (.5,5.5) to (4.5,5.5) to [out=0,in=90] (5,5) to (5,0);
\draw[thick, ->] (1,4.49) to (1,4.5);
\node at (.5,5.45) {\small $*$};
\node at (1.5,5.45) {\small $*$};
\end{tikzpicture}
\to
\begin{tikzpicture}[anchorbase,scale=.25]
\movieconstant
\draw [thick] (2,-1.75) to [out=270,in=0]  (1.5,-2.25) to [out=180,in=270] (1,-1.75) to (1,0.5) (1,3) to (1,3.5) to  [out=90,in=0] (.5,4) to [out=180,in=90] (0,3.5) to (0,-.75) to [out=270,in=180] (.5,-1.25) to (1.5,-1.25) to [out=0,in=90]  (2,-1.75);
\draw[thick, ->] (0,1.55) to (0,1.54);
\node at (.5,-1.3) {\small $*$};
\node at (1.5,-1.3) {\small $*$};
\end{tikzpicture}
}

\newcommand{\RIid}[1]{
\begin{tikzpicture}[anchorbase,xscale=#1,yscale=.3]
\draw [very thick, ->] (0,0) to (0,2.25);
\end{tikzpicture}
}
\newcommand{\RIidcirc}[1]{
\begin{tikzpicture}[anchorbase,xscale=#1,yscale=.3]
\draw [very thick, ->] (0,0) to (0,2.25);
\draw[very thick] (-1,1) circle (5mm);
\draw[very thick,->] (-1.5,0.81) to (-1.5,.8);
\end{tikzpicture}
}

\newcommand{\RIt}[1]{
\begin{tikzpicture}[anchorbase,xscale=#1,yscale=.3]
\draw [very thick, ->] (0,0) to [out=90,in=315] (-.25,.75) to (-.25,1.25) to [out=45,in=270] (0,2.25);
\draw[very thick] (-.25,1.25) to [out=135,in=0] (-.75,1.5) to [out=180,in=90](-1.25,1) to [out=270,in=180] (-.75,0.5) to [out=0,in=225] (-.25,.75);
\draw[very thick,->] (-1.25,0.81) to (-1.25,.8);
\end{tikzpicture}
}

\newcommand{\tinym}{
\scalebox{.7}{-}
}
\newcommand{\tinyy}[1]{
\scalebox{.7}{#1}
}

\newcommand{\tinyz}{
\scalebox{.7}{0}
}

\newcommand{\Bp}[1]{
\draw[blue, line width=.5mm] (#1,-1) to (#1,1);
}
\newcommand{\Rp}[1]{
\draw[red, line width=.5mm] (#1,-1) to (#1,1);
}
\newcommand{\bp}[1]{
\fill[blue] (#1,0) circle (2.5mm);
}
\newcommand{\rp}[1]{
\fill[red] (#1,0) circle (2.5mm);
}

\newcommand{\iB}[1]{
\begin{tikzpicture}[anchorbase,scale=#1]
\draw[blue, line width=.5mm] (0,-1) to (0,1);
\end{tikzpicture}
}
\newcommand{\iR}[1]{
\begin{tikzpicture}[anchorbase,scale=#1]
\draw[red, line width=.5mm] (0,-1) to (0,1);
\end{tikzpicture}
}
\newcommand{\iBR}[1]{
\begin{tikzpicture}[anchorbase,scale=#1]
\draw[blue, line width=.5mm] (0,-1) to (0,1);
\draw[red, line width=.5mm] (1,-1) to (1,1);
\end{tikzpicture}
}
\newcommand{\iRB}[1]{
\begin{tikzpicture}[anchorbase,scale=#1]
\draw[red, line width=.5mm] (0,-1) to (0,1);
\draw[blue, line width=.5mm] (1,-1) to (1,1);
\end{tikzpicture}
}
\newcommand{\sixvert}[2]{
\begin{tikzpicture}[anchorbase,xscale=#1,yscale=#2]
\draw[blue, line width=.5mm] (-1,-1) to (0,0) to (0,1) (0,0) to (1,-1);
\draw[red, line width=.5mm] (-1,1) to (0,0) to (0,-1) (0,0) to (1,1);
\end{tikzpicture}
}
\newcommand{\BRBv}{
\sixvert{0.2}{0.2}
}

\newcommand{\trival}[3]{
\begin{tikzpicture}[anchorbase,xscale=#2,yscale=#3]
\draw[#1, line width=.5mm] (-1,-1) to (0,0) to (0,1) (0,0) to (1,-1);
\end{tikzpicture}
}

\newcommand{\trivd}[4]{
\begin{tikzpicture}[anchorbase,xscale=#3,yscale=#4]
\draw[#1, line width=.5mm] (-1,-1) to (0,0) to (0,1) (0,0) to (1,-1);
\draw[#2, line width=.5mm] (-1,0) to (-1,1);
\fill[#2] (-1,0) circle (2.5mm);
\end{tikzpicture}
}
\newcommand{\trivmd}[4]{
\begin{tikzpicture}[anchorbase,xscale=#3,yscale=#4]
\draw[#1, line width=.5mm] (-1,-1) to (0,.5) to (0,1) (0,.5) to (1,-1);
\draw[#2, line width=.5mm] (0,-.375) to (0,-1);
\fill[#2] (0,-.375) circle (2.5mm);
\end{tikzpicture}
}
\newcommand{\Sdot}[3]{
\begin{tikzpicture}[anchorbase,xscale=#2,yscale=#3]
\draw[#1, line width=.5mm] (0,0) to (0,-1);
\fill[#1] (0,0) circle (2.5mm);
\end{tikzpicture}
}
\newcommand{\SCup}[3]{
\begin{tikzpicture}[anchorbase,xscale=#2,yscale=#3]
\draw[#1, line width=.5mm] (-1,-1) to [out=90,in=180] (0,0) to [out=0,in=90] (1,-1);
\draw[#1, line width=.5mm,opacity=0] (-1,1) to [out=270,in=180] (0,0.25) to [out=0,in=270] (1,1);
\end{tikzpicture}
}
\newcommand{\SCupc}[3]{
\begin{tikzpicture}[anchorbase,xscale=#2,yscale=#3]
\draw[#1, line width=.5mm] (-1,-1) to [out=90,in=180] (0,0) to [out=0,in=90] (1,-1);
\end{tikzpicture}
}

\newcommand{\CCup}[4]{
\begin{tikzpicture}[anchorbase,xscale=#3,yscale=#4]
\draw[#1, line width=.5mm] (-1,-1) to [out=90,in=180] (0,-.25) to [out=0,in=90] (1,-1);
\draw[#2, line width=.5mm] (-1,1) to [out=270,in=180] (0,0.25) to [out=0,in=270] (1,1);
\end{tikzpicture}
}

\newcommand{\BB}{
\begin{tikzpicture}[anchorbase,scale=.2]
\Bp{0}\Bp{1}
\end{tikzpicture}
}
\newcommand{\Bb}{
\begin{tikzpicture}[anchorbase,scale=.2]
\Bp{0}\bp{1}
\end{tikzpicture}
}
\newcommand{\bB}{
\begin{tikzpicture}[anchorbase,scale=.2]
\bp{0}\Bp{1}
\end{tikzpicture}
}
\newcommand{\bb}{
\begin{tikzpicture}[anchorbase,scale=.2]
\bp{0}\bp{1}
\end{tikzpicture}
}

\newcommand{\BRB}{
\begin{tikzpicture}[anchorbase,scale=.2]
\Bp{0}\Rp{1}\Bp{2}
\end{tikzpicture}
}

\newcommand{\bRB}{
\begin{tikzpicture}[anchorbase,scale=.2]
\bp{0}\Rp{1}\Bp{2}
\end{tikzpicture}
}
\newcommand{\BrB}{
\begin{tikzpicture}[anchorbase,scale=.2]
\Bp{0}\rp{1}\Bp{2}
\end{tikzpicture}
}
\newcommand{\BRb}{
\begin{tikzpicture}[anchorbase,scale=.2]
\Bp{0}\Rp{1}\bp{2}
\end{tikzpicture}
}
\newcommand{\Brb}{
\begin{tikzpicture}[anchorbase,scale=.2]
\Bp{0}\rp{1}\bp{2}
\end{tikzpicture}
}
\newcommand{\bRb}{
\begin{tikzpicture}[anchorbase,scale=.2]
\bp{0}\Rp{1}\bp{2}
\end{tikzpicture}
}
\newcommand{\brB}{
\begin{tikzpicture}[anchorbase,scale=.2]
\bp{0}\rp{1}\Bp{2}
\end{tikzpicture}
}
\newcommand{\brb}{
\begin{tikzpicture}[anchorbase,scale=.2]
\bp{0}\rp{1}\bp{2}
\end{tikzpicture}
}
\newcommand{\RBR}{
\begin{tikzpicture}[anchorbase,scale=.2]
\Rp{0}\Bp{1}\Rp{2}
\end{tikzpicture}
}
\newcommand{\RBr}{
\begin{tikzpicture}[anchorbase,scale=.2]
\Rp{0}\Bp{1}\rp{2}
\end{tikzpicture}
}
\newcommand{\RbR}{
\begin{tikzpicture}[anchorbase,scale=.2]
\Rp{0}\bp{1}\Rp{2}
\end{tikzpicture}
}
\newcommand{\rBR}{
\begin{tikzpicture}[anchorbase,scale=.2]
\rp{0}\Bp{1}\Rp{2}
\end{tikzpicture}
}
\newcommand{\Rbr}{
\begin{tikzpicture}[anchorbase,scale=.2]
\Rp{0}\bp{1}\rp{2}
\end{tikzpicture}
}
\newcommand{\rBr}{
\begin{tikzpicture}[anchorbase,scale=.2]
\rp{0}\Bp{1}\rp{2}
\end{tikzpicture}
}
\newcommand{\rbR}{
\begin{tikzpicture}[anchorbase,scale=.2]
\rp{0}\bp{1}\Rp{2}
\end{tikzpicture}
}
\newcommand{\rbr}{
\begin{tikzpicture}[anchorbase,scale=.2]
\rp{0}\bp{1}\rp{2}
\end{tikzpicture}
}

\newcommand{\arroww}[1]{
\draw[thick, ->] (0.5,.1) to (1.5,.4);
\draw[thick, ->] (0.5,
-.1) to (1.5,-.4);
\draw[thick, ->] (2.5,.4) to (3.5,.1);
\draw[thick, ->] (2.5,-.4) to (3.5,-.1);
\draw[thick, ->] (0.5,#1) to (1.5,#1);
\draw[thick, ->] (2.5,#1) to (3.5,#1);}

\newcommand{\arro}[1]{
\draw[thick, ->] (0.5,#1) to (1.5,#1);
\draw[thick, ->] (0.5,#1+.2) to (1.5,#1+.8);
\draw[thick, ->] (0.5,#1-.2) to (1.5,#1-.8);
\draw[thick, ->] (2.5,#1+1) to (3.5,#1+1);
\draw[thick, ->] (2.5,#1+.8) to (3.5,#1+.2);
\draw[thick, ->] (2.5,#1+.2) to (3.5,#1+.8);
\draw[thick, ->] (2.5,#1-.2) to (3.5,#1-.8);
\draw[thick, ->] (2.5,#1-.8) to (3.5,#1-.2);
\draw[thick, ->] (2.5,#1-1) to (3.5,#1-1);
\draw[thick, ->] (4.5,#1) to (5.5,#1);
\draw[thick, ->] (4.5,#1+.8) to (5.5,#1+.2);
\draw[thick, ->] (4.5,#1-.8) to (5.5,#1-.2);
}

\newcommand{\arrowd}[1]{
\draw[thick, ->] (-1.6,-.625) to (-1.6,-1.375);
\draw[thick, <-] (-1.4,-0.625) to (-1.4,-1.375);
\draw[thick, ->] (#1-0.1,-.625) to (#1-.1,-1.375);
\draw[thick, <-] (#1+.1,-.625) to (#1+.1,-1.375);
}

\newcommand{\arrowda}[1]{
\draw[thick, <-] (-1.6,-.625) to (-1.6,-1.375);
\draw[thick, ->] (-1.4,-0.625) to (-1.4,-1.375);
\draw[thick, <-] (#1-0.1,-.625) to (#1-.1,-1.375);
\draw[thick, ->] (#1+.1,-.625) to (#1+.1,-1.375);
}

\newcommand{\arrowwd}{
\draw[thick, ->] (-2.1,-.75) to (-2.1,-1.75);
\draw[thick, <-] (-1.9,-0.75) to (-1.9,-1.75);
\draw[thick, ->] (1.9,-1) to (1.9,-2);
\draw[thick, <-] (2.1,-1) to (2.1,-2);
}

\newcommand{\arrod}{
\draw[thick, ->] (-2,-1.25) to (-2,-2.75);
\draw[thick, ->] (0,-.5) to (0,-3.5);
\draw[thick, ->] (2,-1.5) to (2,-2.5);
\draw[thick, ->] (4,-1.5) to (4,-2.5);
\draw[thick, ->] (6,-.5) to (6,-3.5);
}

\newcommand{\shadingTR}[2]{
\fill[opacity=.3,#1] (1.6,.9+#2) to (1.6,1.2+#2) to [out=90,in=180] (1.8,1.4+#2) to (4.1,1.4+#2) to [out=0,in=150] (4.5,1.3+#2) to (6.3,.35+#2) to [out=330,in=90] (6.4,.2+#2) to (6.4,-.2+#2) to [out=270,in=0] (6.2,-.4+#2) to  (4,-.4+#2) to [out=180,in=330] (3.6,-.3+#2) to (1.7,.65+#2) to [out=150,in=270] (1.6,.9+#2);
}

\newcommand{\shadingBR}[2]{
\fill[opacity=.3,#1] (1.6,-.9+#2) to (1.6,-1.2+#2) to [out=-90,in=180] (1.8,-1.4+#2) to (4.1,-1.4+#2) to [out=0,in=-150] (4.5,-1.3+#2) to (6.3,-.35+#2) to [out=-330,in=-90] (6.4,-.2+#2) to (6.4,.2+#2) to [out=-270,in=0] (6.2,.4+#2) to  (4,.4+#2) to [out=-180,in=-330] (3.6,.3+#2) to (1.7,-.65+#2) to [out=-150,in=-270] (1.6,-.9+#2);
}

\newcommand{\shadingBL}[2]{
\fill[opacity=.3,#1] (4.4,-.9+#2) to (4.4,-1.2+#2) to [out=270,in=0] (4.2,-1.4+#2) to (1.9,-1.4+#2) to [out=180,in=330] (1.5,-1.3+#2) to (-.3,-.35+#2) to [out=150,in=270] (-.4,-.2+#2) to (-.4,.2+#2) to [out=90,in=180] (-.2,.4+#2) to  (2,.4+#2) to [out=0,in=150] (2.4,.3+#2) to (4.3,-.65+#2) to [out=330,in=90] (4.4,-.9+#2);
}

\newcommand{\shadingTL}[2]{
\fill[opacity=.3,#1] (4.4,.9+#2) to (4.4,1.2+#2) to [out=-270,in=0] (4.2,1.4+#2) to (1.9,1.4+#2) to [out=180,in=-330] (1.5,1.3+#2) to (-.3,.35+#2) to [out=-150,in=-270] (-.4,.2+#2) to (-.4,-.2+#2) to [out=-90,in=180] (-.2,-.4+#2) to  (2,-.4+#2) to [out=0,in=-150] (2.4,-.3+#2) to (4.3,.65+#2) to [out=-330,in=-90] (4.4,.9+#2);
}

\newcommand{\LR}{
\draw [white,line width=0.15cm] (0,0) to [out=90,in=270] (2,3);
\draw [very thick, ->] (0,0) to [out=90,in=270] (2,3);
}

\newcommand{\RL}{
\draw [white,line width=0.15cm] (2,0) to [out=90,in=270] (0,3);
\draw [very thick, ->] (2,0) to [out=90,in=270] (0,3);
}

\newcommand{\MiL}{
\draw [white,line width=0.15cm] (1,0) to [out=90,in=270](0,1.5)to [out=90,in=270]  (1,3);
\draw [very thick, ->] (1,0) to [out=90,in=270](0,1.5)to [out=90,in=270]  (1,3);
}

\newcommand{\MiR}{
\draw [white,line width=.15cm] (1,0) to [out=90,in=270](2,1.5)to [out=90,in=270]  (1,3);
\draw [very thick, ->] (1,0) to [out=90,in=270](2,1.5)to [out=90,in=270]  (1,3);
}

\newcommand{\RIeqn}{
\begin{tikzpicture}[anchorbase,scale=.75]
\node at (-1.5,0) {
$\CC{
\begin{tikzpicture}[anchorbase,scale=.3]
\draw [very thick] (0,0) to (0,1) to [out=90,in=0] (-.5,1.5) to [out=180,in=90] (-1,1);
\draw [white,line width=0.15cm] (-1,1) to [out=270,in=180] (-.5,.5) to [out=0,in=270] (0,1)to (0,2);
\draw [very thick, ->]
(-1,1) to [out=270,in=180] (-.5,.5) to [out=0,in=270] (0,1)to (0,2.25);\end{tikzpicture}
}$
};
\node at (-1.5,-2) {
$\CC{\begin{tikzpicture}[scale=.3]
\draw [very thick, ->] (0,0) to (0,2.25);
\end{tikzpicture}
}$
};
\node at (0,0) {\RIt{.3}};
\node at (2,0) {\RIidcirc{.3}};
\node at (0,-2) {\tiny $\emptyset$};.
\node at (2,-2) {\RIid{.3}};
\draw[thick, ->] (0.5,0) to (1.35,0);
\draw[thick, ->] (0.5,-2) to (1.5,-2);
\arrowda{2}
\node at (-2.1,-1) {\tiny $R1_+$};
\node at (2.6,-1) {\tiny\textrm{dcap}};
\node at (1.4,-1) {\tiny\textrm{cup}};
\end{tikzpicture}
\quad
\begin{tikzpicture}[anchorbase,scale=.75]
\node at (-1.5,0) {
$\CC{
\begin{tikzpicture}[anchorbase,scale=.3]
\draw [very thick, ->]
(-1,1) to [out=270,in=180] (-.5,.5) to [out=0,in=270] (0,1)to (0,2.25);
\draw [white,line width=0.15cm] (0,0) to (0,1) to [out=90,in=0] (-.5,1.5) to [out=180,in=90] (-1,1);
\draw [very thick] (0,0) to (0,1) to [out=90,in=0] (-.5,1.5) to [out=180,in=90] (-1,1);
;\end{tikzpicture}
}$
};
\node at (-1.5,-2) {
$\CC{\begin{tikzpicture}[scale=.3]
\draw [very thick, ->] (0,0) to (0,2.25);
\end{tikzpicture}
}$
};
\node at (2,0) {\RIt{.3}};
\node at (0,0) {\RIidcirc{.3}};
\node at (2,-2) {\tiny $\emptyset$};.
\node at (0,-2) {\RIid{.3}};
\draw[thick, ->] (0.5,0) to (1.35,0);
\draw[thick, ->] (0.5,-2) to (1.5,-2);
\arrowda{0}
\node at (-2.1,-1) {\tiny $R1_-$};
\node at (.6,-1) {\tiny\textrm{cap}};
\node at (-.6,-1) {\tiny\textrm{dcup}};
\end{tikzpicture}
\quad
\begin{tikzpicture}[anchorbase,scale=.75]
\node at (-1.5,0) {
$\CC{
\begin{tikzpicture}[anchorbase,xscale=-.3,yscale=.3]
\draw [very thick, ->]
(-1,1) to [out=270,in=180] (-.5,.5) to [out=0,in=270] (0,1)to (0,2.25);
\draw [white,line width=0.15cm] (0,0) to (0,1) to [out=90,in=0] (-.5,1.5) to [out=180,in=90] (-1,1);
\draw [very thick] (0,0) to (0,1) to [out=90,in=0] (-.5,1.5) to [out=180,in=90] (-1,1);
\end{tikzpicture}
}$
};
\node at (-1.5,-2) {
$\CC{\begin{tikzpicture}[scale=.3]
\draw [very thick, ->] (0,0) to (0,2.25);
\end{tikzpicture}
}$
};
\node at (0,0) {\RIt{-.3}};
\node at (2,0) {\RIidcirc{-.3}};
\node at (0,-2) {\tiny $\emptyset$};.
\node at (2,-2) {\RIid{.3}};
\draw[thick, ->] (0.5,0) to (1.35,0);
\draw[thick, ->] (0.5,-2) to (1.5,-2);
\arrowd{2}
\node at (-2.1,-1) {\tiny $R1^{-1}_+$};
\node at (1.4,-1) {\tiny\textrm{dcap}};
\node at (2.6,-1) {\tiny\textrm{cup}};
\end{tikzpicture}
\quad
\begin{tikzpicture}[anchorbase,scale=.75]
\node at (-1.5,0) {
$\CC{
\begin{tikzpicture}[anchorbase,xscale=-.3,yscale=.3]
\draw [very thick] (0,0) to (0,1) to [out=90,in=0] (-.5,1.5) to [out=180,in=90] (-1,1);
\draw [white,line width=0.15cm] (-1,1) to [out=270,in=180] (-.5,.5) to [out=0,in=270] (0,1)to (0,2.25);
\draw [very thick, ->]
(-1,1) to [out=270,in=180] (-.5,.5) to [out=0,in=270] (0,1)to (0,2.25);
\end{tikzpicture}
}$
};
\node at (-1.5,-2) {
$\CC{\begin{tikzpicture}[scale=.3]
\draw [very thick, ->] (0,0) to (0,2.25);
\end{tikzpicture}
}$
};
\node at (2,0) {\RIt{-.3}};
\node at (0,0) {\RIidcirc{-.3}};
\node at (2,-2) {\tiny $\emptyset$};.
\node at (0,-2) {\RIid{.3}};
\draw[thick, ->] (0.5,0) to (1.35,0);
\draw[thick, ->] (0.5,-2) to (1.5,-2);
\arrowd{0}
\node at (-2.1,-1) {\tiny $R1^{-1}_-$};
\node at (-.6,-1) {\tiny\textrm{cap}};
\node at (.6,-1) {\tiny\textrm{dcup}};
\end{tikzpicture}
}

\newcommand{\RIIeqn}{
\begin{tikzpicture}[anchorbase,scale=.75]
\node at (-2,0) {
$\CC{
\begin{tikzpicture}[anchorbase,scale=.3]
\draw [white,line width=0.15cm] (1,0) to [out=90,in=270](0,1.5)to [out=90,in=270]  (1,3);
\draw [very thick, ->] (1,0) to [out=90,in=270](0,1.5)to [out=90,in=270]  (1,3);
\draw [white,line width=.15cm] (0,0) to [out=90,in=270](1,1.5)to [out=90,in=270]  (0,3);
\draw [very thick, ->] (0,0) to [out=90,in=270](1,1.5)to [out=90,in=270]  (0,3);
\end{tikzpicture}
}$
};
\node at (-2,-2.5) {
$\CC{\begin{tikzpicture}[scale=.3]
\draw [very thick, ->] (0,0) to (0,3);
\draw [very thick, ->] (1,0) to (1,3);
\end{tikzpicture}
}$
};
\node at (0,0) {\bB};
\node at (2,.5) {\bb};
\node at (2,-.5) {\BB};
\node at (4,0) {\Bb};
\arroww{-2.5}
\node at (3,-.15) {\tiny $-$};
\arrowwd
\node at (0,-2.5) {\tiny $\emptyset$};.
\node at (2,-2.5) {\bb};
\node at (4,-2.5) {\tiny $\emptyset$};
\node at (1.4,-1.5) {
\scalebox{.6}{$
\begingroup
\setlength\arraycolsep{2pt}\begin{bmatrix}
1&-\SCupc{blue}{.15}{.15}
\end{bmatrix}
\endgroup
$}};
\node at (2.6,-1.5) {
\scalebox{.6}{$
\begingroup
\setlength\arraycolsep{2pt}\begin{bmatrix}
1\\ \SCupc{blue}{.15}{-.15}
\end{bmatrix}
\endgroup
$}};
\end{tikzpicture}
\qquad\qquad
\begin{tikzpicture}[anchorbase,scale=.75]
\node at (-2,0) {
$\CC{
\begin{tikzpicture}[anchorbase,scale=.3]
\draw [white,line width=.15cm] (0,0) to [out=90,in=270](1,1.5)to [out=90,in=270]  (0,3);
\draw [very thick, ->] (0,0) to [out=90,in=270](1,1.5)to [out=90,in=270]  (0,3);
\draw [white,line width=0.15cm] (1,0) to [out=90,in=270](0,1.5)to [out=90,in=270]  (1,3);
\draw [very thick, ->] (1,0) to [out=90,in=270](0,1.5)to [out=90,in=270]  (1,3);
\end{tikzpicture}
}$
};
\node at (-2,-2.5) {
$\CC{\begin{tikzpicture}[scale=.3]
\draw [very thick, ->] (0,0) to (0,3);
\draw [very thick, ->] (1,0) to (1,3);
\end{tikzpicture}
}$
};
\node at (0,0) {\Bb};
\node at (2,.5) {\bb};
\node at (2,-.5) {\BB};
\node at (4,0) {\bB};
\arroww{-2.5}
\node at (1,-.15) {\tiny $-$};
\arrowwd
\node at (0,-2.5) {\tiny $\emptyset$};.
\node at (2,-2.5) {\bb};
\node at (4,-2.5) {\tiny $\emptyset$};
\node at (1.4,-1.5) {
\scalebox{.6}{$
\begingroup
\setlength\arraycolsep{2pt}\begin{bmatrix}
1&\SCupc{blue}{.15}{.15}
\end{bmatrix}
\endgroup
$}};
\node at (2.6,-1.5) {
\scalebox{.6}{$
\begingroup
\setlength\arraycolsep{2pt}\begin{bmatrix}
1\\ -\SCupc{blue}{.15}{-.15}
\end{bmatrix}
\endgroup
$}};
\end{tikzpicture}
}

\newcommand{\RIIIypeqn}{
\begin{tikzpicture}[anchorbase,scale=.85]
\node at (-2,0) {
$\CC{
\begin{tikzpicture}[anchorbase,scale=.35]
\RL
\MiR
\LR
\end{tikzpicture}
}$
};
\node at (-2,-4) {
$\CC{\begin{tikzpicture}[scale=.35]
\RL
\MiL
\LR
\end{tikzpicture}
}$
};
\fill[yellow, opacity=.3] (-.1,-1.8) rectangle (-.5,-2.2);
\fill[yellow, opacity=.3] (.1,-1.85) rectangle (1.25,-2.53);
\fill[green, opacity=.3] (0.1,-1.46) rectangle (1.28,-1.85);
\fill[blue, opacity=.3] (1.8,-1.85) rectangle (1.28,-1.46);
\fill[blue, opacity=.3] (4.7,-1.46) rectangle (5.48,-2.2);
\fill[green, opacity=.3] (4.42,-1.46) rectangle (4.7,-2.2);
\fill[yellow, opacity=.3] (4.42,-2.2) rectangle (4.7,-2.53);
\fill[blue, opacity=.3] (6.1,-1.8) rectangle (6.5,-2.2);
\shadingBR{blue}{0}
\shadingTL{yellow}{0}
\node at (0,0) {\BRB};
\node at (2,1) {\bRB};
\node at (2,0) {\BrB};
\node at (2,-1) {\BRb};
\node at (4,1) {\brB};
\node at (4,0) {\bRb};
\node at (4,-1) {\Brb};
\node at (6,0) {\brb};
\arro{0}
\node at (1,.1) {\tiny $-$};
\node at (2.8,-1.15) {\tiny $-$};
\node at (2.8,-.25) {\tiny $-$};
\node at (2.8,.75) {\tiny $-$};
\arrod
\shadingTR{blue}{-4}
\shadingBL{yellow}{-4}
\node at (0,-4) {\RBR};.
\node at (2,-3) {\rBR};
\node at (2,-4) {\RbR};
\node at (2,-5) {\RBr};
\node at (4,-3) {\rbR};
\node at (4,-4) {\rBr};
\node at (4,-5) {\Rbr};
\node at (6,-4) {\rbr};
\arro{-4}
\node at (1,-3.9) {\tiny $-$};
\node at (2.8,-5.15) {\tiny $-$};
\node at (2.8,-4.25) {\tiny $-$};
\node at (2.8,-3.25) {\tiny $-$};
\node at (-.3,-2) {
\scalebox{.6}{\sixvert{.15}{.15}}
};
\node at (1,-2) {
\scalebox{.6}{$
\begingroup
\setlength\arraycolsep{1pt}\begin{bmatrix}
\tinyz &\tinym\!\trivd{blue}{red}{-.15}{.15}&\iBR{.15} \\
\tinyy{(y-1)}\trivd{red}{blue}{-.15}{-.15}
&\CCup{blue}{red}{.15}{.15}&
\tinyy{-y~}\trivd{red}{blue}{.15}{-.15}
\\
\iRB{.15} & \tinym\trivd{blue}{red}{.15}{.15} & \tinyz
\end{bmatrix}
\endgroup
$
}
};
\node at (5,-2) {
\scalebox{.6}{$
\begingroup
\setlength\arraycolsep{2pt}\begin{bmatrix}
\tinyz &
\tinyy{(1-y)~}\iR{.15}
&\tinyz \\
\iB{.15} & \tinyz & \iB{.15}\\
\tinyz &
\tinyy{y~}\iR{.15}
& \tinyz
\end{bmatrix}
\endgroup
$
}
};
\node at (6.3,-2) {
\scalebox{.5}{$1$}
};
\end{tikzpicture}
\qquad
\begin{tikzpicture}[anchorbase,scale=.85]
\node at (-2,0) {
$\CC{
\begin{tikzpicture}[anchorbase,scale=.35]
\LR
\RL
\MiR
\end{tikzpicture}
}$
};
\node at (-2,-4) {
$\CC{\begin{tikzpicture}[scale=.35]
\LR
\RL
\MiL
\end{tikzpicture}
}$
};
\fill[yellow, opacity=.3] (-.1,-1.8) rectangle (-.5,-2.2);
\fill[blue, opacity=.3] (1.28,-1.46) rectangle (1.8,-1.85);
\fill[green, opacity=.3] (0.14,-1.46) rectangle (1.28,-1.85);
\fill[yellow, opacity=.3] (.14,-2.53) rectangle (1.28,-1.85);
\fill[yellow, opacity=.3] (4.43,-2.2) rectangle (4.75,-2.53);
\fill[green, opacity=.3] (4.43,-2.2) rectangle (4.75,-1.46);
\fill[blue, opacity=.3] (5.5,-2.2) rectangle (4.75,-1.46);
\fill[blue, opacity=.3] (6.1,-1.8) rectangle (6.5,-2.2);
\shadingBR{blue}{0}
\shadingTL{yellow}{0}
\node at (0,0) {\brB};
\node at (2,1) {\BrB};
\node at (2,0) {\bRB};
\node at (2,-1) {\brb};
\node at (4,1) {\BRB};
\node at (4,0) {\Brb};
\node at (4,-1) {\bRb};
\node at (6,0) {\BRb};
\arro{0}
\node at (5,0.1) {\tiny $-$};
\node at (2.8,1.15) {\tiny $-$};
\node at (2.8,-.25) {\tiny $-$};
\node at (2.8,.75) {\tiny $-$};
\arrod
\shadingTR{blue}{-4}
\shadingBL{yellow}{-4}
\node at (0,-4) {\Rbr};
\node at (2,-3) {\rbr};
\node at (2,-4) {\RBr};
\node at (2,-5) {\RbR};
\node at (4,-3) {\rBr};
\node at (4,-4) {\rbR};
\node at (4,-5) {\RBR};
\node at (6,-4) {\rBR};
\arro{-4}
\node at (1,-3.9) {\tiny $-$};
\node at (1,-4.4) {\tiny $-$};
\node at (2.8,-4.9) {\tiny $-$};
\node at (4.8,-3.2) {\tiny $-$};
\node at (-.3,-2) {
\scalebox{.5}{
0
}
};
\node at (1,-2) {
\scalebox{.6}{$
\begingroup
\setlength\arraycolsep{1pt}\begin{bmatrix}
\tinym\SCup{blue}{.15}{.15}
&
\tinyz
&1
\\
\tinym\trivd{blue}{red}{.15}{.15}
&
\iRB{.15}
&
\tinyz
\\
\CCup{blue}{red}{.15}{.15} & 
\tinyy{(y-1)}\trivd{red}{blue}{-.15}{-.15}
 &
\tinyy{-y}\SCup{red}{.15}{-.15}
\end{bmatrix}
\endgroup
$
}
};
\node at (5,-2) {
\scalebox{.6}{$
\begingroup
\setlength\arraycolsep{1pt}\begin{bmatrix}
\trivmd{blue}{red}{.15}{.15} &
\iB{.15}&\tinyz \\
\tinyz & \tinyz &
\tinyy{(1-y)~}\iR{.15}
\\
\sixvert{.15}{.15} &
\tinyz
&
\tinyy{y}\trivmd{red}{blue}{.15}{-.15}
\end{bmatrix}
\endgroup
$
}
};
\node at (6.3,-2) {
\scalebox{.6}{\iBR{.15}}
};
\end{tikzpicture}
}

\newcommand{\RIIIyneqn}{
\begin{tikzpicture}[anchorbase,scale=.85]
\node at (-2,0) {
$\CC{
\begin{tikzpicture}[anchorbase,scale=.35]
\MiR
\RL
\LR
\end{tikzpicture}
}$
};
\node at (-2,-4) {
$\CC{\begin{tikzpicture}[scale=.35]
\MiL
\RL
\LR
\end{tikzpicture}
}$
};
\fill[blue, opacity=.3] (-.1,-1.8) rectangle (-.5,-2.2);
\fill[blue, opacity=.3] (.23,-1.85) rectangle (1.2,-2.53);
\fill[green, opacity=.3] (.23,-1.46) rectangle (1.2,-1.85);
\fill[yellow, opacity=.3] (1.7,-1.85) rectangle (1.2,-1.46);
\fill[blue, opacity=.3] (4.22,-2.2) rectangle (4.6,-2.53);
\fill[green, opacity=.3] (4.22,-1.46) rectangle (4.6,-2.2);
\fill[yellow, opacity=.3] (5.7,-2.2) rectangle (4.6,-1.46);
\fill[yellow, opacity=.3] (6.1,-1.8) rectangle (6.5,-2.2);
\shadingBR{yellow}{0}
\shadingTL{blue}{0}
\node at (0,0) {\BRb};
\node at (2,1) {\bRb};
\node at (2,0) {\Brb};
\node at (2,-1) {\BRB};
\node at (4,1) {\brb};
\node at (4,0) {\bRB};
\node at (4,-1) {\BrB};
\node at (6,0) {\brB};
\node at (1,.1) {\tiny $-$};
\node at (2.8,-1.15) {\tiny $-$};
\node at (2.8,-0.25) {\tiny $-$};
\node at (2.8,0.75) {\tiny $-$};
\arrod
\shadingTR{yellow}{-4}
\shadingBL{blue}{-4}
\node at (0,-4) {\rBR};
\node at (2,-3) {\RBR};
\node at (2,-4) {\rbR};
\node at (2,-5) {\rBr};
\node at (4,-3) {\RbR};
\node at (4,-4) {\RBr};
\node at (4,-5) {\rbr};
\node at (6,-4) {\Rbr};
\arro{0}
\node at (5,-3.9) {\tiny $-$};
\node at (1,-4.4) {\tiny $-$};
\node at (2.8,-3.85) {\tiny $-$};
\node at (4.8,-3.25) {\tiny $-$};
\arro{-4}
\node at (-.3,-2) {
\scalebox{.6}{\iBR{.15}}
};
\node at (1,-2) {
\scalebox{.6}{$
\begingroup
\setlength\arraycolsep{1pt}\begin{bmatrix}
\tinym\trivmd{red}{blue}{.15}{-.15} &
\tinyz
& \sixvert{.15}{.15} \\
\iR{.15}
 & \tinyz & \tinyz\\
 \tinyz &
\tinyy{(1-y)~}\iB{.15}
&
\tinyy{-y} \trivmd{blue}{red}{.15}{.15}
\end{bmatrix}
\endgroup
$
}
};
\node at (5,-2) {
\scalebox{.6}{$
\begingroup
\setlength\arraycolsep{1pt}\begin{bmatrix}
\SCup{red}{.15}{-.15}&
\tinym\trivd{red}{blue}{-.15}{-.15}&\CCup{red}{blue}{.15}{-.15} \\
\tinyz
&\iRB{.15}&
\tinyy{(y-1)}\trivd{blue}{red}{.15}{.15}
\\
1 &  \tinyz &
\tinyy{y} \SCup{blue}{.15}{.15}
\end{bmatrix}
\endgroup
$
}
};
\node at (6.3,-2) {
\scalebox{.5}{$0$}
};
\end{tikzpicture}
\qquad
\begin{tikzpicture}[anchorbase,scale=.85]
\node at (-2,0) {
$\CC{
\begin{tikzpicture}[anchorbase,scale=.35]
\LR
\MiR
\RL
\end{tikzpicture}
}$
};
\node at (-2,-4) {
$\CC{\begin{tikzpicture}[scale=.35]
\LR
\MiL
\RL
\end{tikzpicture}
}$
};
\fill[blue, opacity=.3] (-.1,-1.8) rectangle (-.5,-2.2);
\fill[yellow, opacity=.3] (1.55,-1.46) rectangle (1.2,-1.8);
\fill[green, opacity=.3] (1.2,-1.46) rectangle (.4,-1.8);
\fill[blue, opacity=.3] (1.2,-2.53) rectangle (.4,-1.8);
\fill[yellow, opacity=.3] (4.6,-1.46) rectangle (5.73,-2.2);
\fill[green, opacity=.3] (4.2,-1.46) rectangle (4.6,-2.2);
\fill[blue, opacity=.3] (4.6,-2.53) rectangle (4.2,-2.2);
\fill[yellow, opacity=.3] (6.1,-1.8) rectangle (6.5,-2.2);
\shadingBR{yellow}{0}
\shadingTL{blue}{0}
\node at (0,0) {\brb};
\node at (2,1) {\Brb};
\node at (2,0) {\bRb};
\node at (2,-1) {\brB};
\node at (4,1) {\BRb};
\node at (4,0) {\BrB};
\node at (4,-1) {\bRB};
\node at (6,0) {\BRB};
\arro{0}
\node at (5,0.1) {\tiny $-$};
\node at (2.8,1.15) {\tiny $-$};
\node at (2.8,-.25) {\tiny $-$};
\node at (2.8,.75) {\tiny $-$};
\arrod
\shadingTR{yellow}{-4}
\shadingBL{blue}{-4}
\node at (0,-4) {\rbr};
\node at (2,-3) {\Rbr};
\node at (2,-4) {\rBr};
\node at (2,-5) {\rbR};
\node at (4,-3) {\RBr};
\node at (4,-4) {\RbR};
\node at (4,-5) {\rBR};
\node at (6,-4) {\RBR};
\arro{-4}
\node at (5,-3.9) {\tiny $-$};
\node at (2.8,-2.85) {\tiny $-$};
\node at (2.8,-4.25) {\tiny $-$};
\node at (2.8,-3.25) {\tiny $-$};

\node at (-.3,-2) {
\scalebox{.5}{$1$}
};
\node at (1,-2) {
\scalebox{.6}{$
\begingroup
\setlength\arraycolsep{2pt}\begin{bmatrix}
\tinyz&
\iR{.15}
&\tinyz \\
\tinyy{(1-y)~}\iB{.15}
& \tinyz &
\tinyy{y~}\iB{.15}
\\
\tinyz &
\iR{.15}
& \tinyz
\end{bmatrix}
\endgroup
$
}
};
\node at (5,-2) {
\scalebox{.6}{$
\begingroup
\setlength\arraycolsep{1pt}\begin{bmatrix}
\tinyz&
\tinyy{(y-1)}\trivd{blue}{red}{.15}{.15}
&\iRB{.15} \\
\tinym \trivd{red}{blue}{.15}{-.15}
&\CCup{blue}{red}{.15}{.15}&
 \tinym\trivd{red}{blue}{-.15}{-.15}\\
\iBR{.15} &

\tinyy{-y}\trivd{blue}{red}{-.15}{.15}
& \tinyz
\end{bmatrix}
\endgroup
$
}
};
\node at (6.3,-2) {
\scalebox{.6}{\sixvert{.15}{.15}}
};
\end{tikzpicture}
}

\usepackage[normalem]{ulem}
\usepackage{bbm}
\usepackage{wrapfig}

\newcommand{\CC}[1]{\left\llbracket #1 \right\rrbracket}
\newcommand{\Komh}{\boldsymbol{\mathrm{K}}^b}
\newcommand{\HCh}{\boldsymbol{\mathrm{HCh}}^b}
\newcommand{\Foam}{\boldsymbol{\mathrm{Foam}}_N}
\newcommand{\Web}{\boldsymbol{\mathrm{Web}}_N}
\newcommand{\AbGrp}{\boldsymbol{\mathrm{AbGrp}}}

\newcommand{\TD}{\boldsymbol{\mathrm{TD}}}
\newcommand{\Linko}{\boldsymbol{\mathrm{Link}}^\circ}
\newcommand{\KhR}{\mathrm{KhR}_N}
\newcommand{\Kh}{\KhR}

\newcommand{\Khb}{\KhR} 
\def\deq{\stackrel{\mathrm{def}}{=}}
\def\bd{\partial}
\newcommand{\Fund}{\mathrm{Fund}}
\newcommand{\glN}{\mathfrak{gl}_N}

\newcommand{\GLN}{\mathrm{GL}(N)}
\newcommand{\Gr}{\mathrm{Gr}}
\newcommand{\gre}{\mathrm{gr}_{\mathrm{ext}}}
\newcommand{\gri}{\mathrm{gr}_{\mathrm{int}}}
\newcommand{\swf}{\mathrm{sw}_+}
\newcommand{\swb}{\mathrm{sw}_-}
\newcommand{\sw}{\mathrm{sw}}
\newcommand{\Wi}{W_\mathrm{in}}
\newcommand{\Wo}{W_\mathrm{out}}
\newcommand{\Si}{\Sigma_\mathrm{in}}
\newcommand{\So}{\Sigma_\mathrm{out}}
\newcommand{\Diff}{\mathrm{Diff}}

\newcommand{\bmKhR}{\boldsymbol{\mathrm{KhR}}_N}
\newcommand{\KhRcat}{\mathsf{KhR}_N} 
\newcommand{\skeinzero}{\mathcal{S}^N_{0}} 
\newcommand{\skein}{\mathcal{S}^N} 
\newcommand{\skeinstar}{\mathcal{S}^N_*} 

\newcommand{\bmKhRi}{\boldsymbol{\mathrm{KhR}}_\infty}
\newcommand{\hdual}[1]{#1^{\#}}
\newcommand{\hddual}[1]{#1^{\#\#}}
\newcommand{\vdual}[1]{#1^{*}}
\newcommand{\vddual}[1]{#1^{**}}

\newcommand{\R}{{\mathbb R}}

\newcommand{\C}{{\mathbb C}}

\newcommand{\Z}{{\mathbb Z}}

\newcommand{\I}{{[0,1]}} 

\newcommand{\lasagnaf}{lasagna filling} 
\newcommand{\lasagnafs}{lasagna fillings} 

\begin{document}

\title{Invariants of 4-manifolds from Khovanov--Rozansky link homology}
\author{Scott Morrison, Kevin Walker, and Paul Wedrich}

\address{S.M.: Canberra ACT 2602, Australia, 
\href{https://tqft.net}{tqft.net}}

\address{K.W.: Microsoft Station Q, Santa Barbara, California 93106-6105, USA, 
\href{http://canyon23.net/math/}{canyon23.net/math/}}

\address{P.W.: Mathematical Sciences Institute, The Australian National University, 
Hanna Neumann Building,
Canberra ACT 2601, Australia, 
\href{http://paul.wedrich.at}{paul.wedrich.at}
}

\begin{abstract}
We use Khovanov--Rozansky $\glN$ link homology to define 
invariants of oriented smooth 4-manifolds, as skein modules constructed
from certain 4-categories with well-behaved duals.

The technical heart of this construction is a proof of the sweep-around property,
which makes these link homologies well defined in the 3-sphere.

\end{abstract}
\maketitle
\thispagestyle{empty}
\setcounter{tocdepth}{1}
 \tableofcontents

\section{Introduction}

Following the seminal articles of Jones, Witten, and
Atiyah~\cite{MR0766964,MR990772,MR1001453}, Crane and Frenkel outlined their
vision for an algebraic construction of invariants of smooth 4-dimensional
manifolds~\cite{MR1295461,MR1355904}, inspired by the initial signs of
categorification they saw in Lusztig's theory of canonical bases
\cite{MR1227098}. A major milestone towards this goal was Khovanov's celebrated
categorification of the Jones polynomial~\cite{MR1740682}---now known as
\emph{Khovanov homology}---which has since been rediscovered or reconstructed in
many parts of mathematics and theoretical physics, see e.g. Stroppel
\cite{MR2120117, MR2521250}, Gukov--Schwarz--Vafa~\cite{MR2193547},
Seidel--Smith~\cite{MR2254624} and Abouzaid--Smith \cite{MR3867999},
Cautis--Kamnitzer \cite{MR2411561,MR2430980}, and Witten~\cite{MR2852941}.
Rasmussen's construction of his slice genus bound \cite{MR2729272} demonstrates
that Khovanov homology is sensitive to 4-dimensional smooth structure and shares
similarities with invariants defined using gauge theory---two impressions that
have since been supported by subsequent work, such as the unknot detection
theorem of Kronheimer--Mrowka~\cite{MR2805599}.

The purpose of this article is to construct a family of bigraded abelian groups
$\skeinzero(W; L)$, depending on an oriented smooth 4-manifold $W$ and a framed
oriented link $L$ in its boundary, from on the Khovanov--Rozansky $\glN$ link
homology theories \cite{MR2391017} (which specialize to Khovanov homology at
$N=2$). Our construction has three steps. First we establish the functoriality
of Khovanov--Rozansky link homology theories under link cobordisms in $S^3\times
\I$. In the second step we use these functorial invariants to construct certain
4-categories, which are the algebraic objects that encode the invariant
$\skeinzero(B^4;L)$ for the 4-ball along with the operations induced by gluing
4-balls. In the third step, we integrate his local data over an oriented smooth
4-manifold using standard colimit/skein techniques to produce the
invariant $\skeinzero(W; L)$, which should be thought of as the Hilbert
space of an associated $4{+}\epsilon$-dimensional TQFT.

As the notation suggests, there are also bigraded abelian groups
$\skein_i(W; L)$ for $i>0$, defined using the \emph{blob homology} construction
of Morrison--Walker \cite{MR2978449}, which we will not pursue in this paper. Another idea left
for future work concerns a lift to a fully homotopy-coherent theory valued in
chain complexes rather than abelian groups, which we will comment on below.

The conceptual innovation here is the identification of a property
that ensures that a 4-category has sufficiently well-behaved duality, allowing
us to integrate it over an oriented smooth 4-manifold. This property, which we
call the \emph{sweep-around property}, is relevant in each of the two
axiomatizations of 4-categories with duals we describe below.

Our computational advance is an 
explicit verification of this property for the 4-categories built from Khovanov--Rozansky link homology, 
specifically that 
link cobordisms represented by movies of the form\vspace{-3mm}
\begin{equation}\label{eqn:sweep1}
\begin{tikzpicture}[anchorbase,scale=1]
 	\node at (0,0) {\includegraphics[width=.7\linewidth]{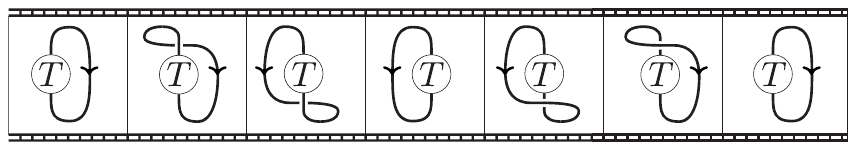}};
\end{tikzpicture}
\end{equation} 
induce identity maps on the level of link homology. For a link homology theory,
this property is equivalent to functoriality under link cobordisms in $S^3\times
\I$, which has important consequences beyond the scope of this paper, such as the
injectivity of maps induced by ribbon concordances, see Kang~\cite{1909.06969}.

\subsection*{Link homology in the 3-sphere}
In the following we give an outline of the construction. We start with the
Khovanov--Rozansky link homologies, which are categorifications of the $\glN$
quantum link invariants of Reshetikhin--Turaev \cite{MR1036112}. These link
homologies take the shape of functors
\[
\begin{Bmatrix}
\textrm{framed oriented link embeddings in } \R^3
\\
\textrm{oriented link cobordisms in }
\R^3\times \I \textrm{ up to isotopy rel } \bd
\end{Bmatrix}
\xrightarrow{\KhR}
\begin{Bmatrix}
\textrm{bigraded abelian groups}
\\
\textrm{homogeneous homomorphisms}
\end{Bmatrix}
\]
which were constructed by Ehrig--Tubbenhauer--Wedrich in \cite{MR3877770}
following earlier work on functoriality by Bar-Natan~\cite{MR2174270},
Clark--Morrison--Walker \cite{MR2496052}, and Blanchet \cite{MR2647055}, and
using technology developed by Robert--Wagner \cite{1702.04140} and Rose--Wedrich
\cite{MR3590355} following Mackaay--Sto{\v{s}}i{\'c}--Vaz~\cite{MR2491657},
Lauda--Queffelec--Rose~\cite{MR3426687}, and Queffelec--Rose~\cite{MR3545951}.

It is worth emphasizing that the functors $\KhR$ considered here are defined
combinatorially and normalized to be sensitive to framing changes, in
contrast to earlier incarnations of Khovanov--Rozansky homology. In the
following, all links are oriented and framed and all link
cobordisms are oriented.

The first step in our construction is to show that Khovanov--Rozansky homologies
make sense as \emph{functorial invariants of links in} $S^3$, rather than just
in $\R^3$. From the point of view of link embeddings and link cobordisms, there
is not much difference between these two cases. A generic link embedding will
miss the point $\infty$ if we consider $\R^3=S^3\setminus \{\infty\}$ and a
generic link cobordism embedded in $S^3\times \I$ will miss $\{\infty\}\times
\I$. However, the analogous statement is no longer true for isotopies of link
cobordisms. While link embeddings and their cobordisms can be represented by
link diagrams in $\R^2$ and movies between them, there are additional isotopies
of link cobordisms in $S^3\times \I$, that do not exist in $\R^3 \times \I$. In
addition to the standard Carter--Rieger--Saito movie moves
\cite{MR1238875,MR1445361}, a link homology theory that is functorial in $S^3$
additionally has to satisfy the so-called \textit{sweep-around move}
\eqref{eqn:sweep1}, which encodes a small isotopy of a sheet of link cobordism
through $\infty\times \I$. The central technical result that we prove in \S
\ref{sec:sweep} is the following.
\begin{thm}
\label{thm:sweeparound} The Khovanov--Rozansky link homologies satisfy the sweep-around move, i.e. they 
associate identity maps to link cobordisms represented by movies of the 
form \eqref{eqn:sweep1}.
\end{thm}

This move is significantly more complex than any of the Carter--Saito movie moves
because it lacks any locality after the projection to $\R^2$, and thus has to be
checked for any tangle $T$ with two endpoints. 
We do this in \S \ref{sec:sweep} and thereby also demonstrate how computable cobordism maps 
in Khovanov--Rozansky homology have become.

\subsection*{4-categories}
The main tool in constructing the 4-manifold invariants $\skeinzero$ is a
family of 4-categories with sufficiently well-behaved duals. This is in analogy
with the case of quantum invariants of 3-manifolds, which---in one way or
another---all depend on a suitable 3-category, such as the ribbon category
$\Rep(U_q(\glN))$ of finite-dimensional representations of quantum $\glN$. 

In fact, the 4-categories we construct should be thought of as \emph{categorified
representation categories}\footnote{These are related, but not identical, to categories of higher
representations of \emph{categorified} quantum $\glN$.} of quantum $\glN$. 
They are defined to have unique 0- and 1-morphisms and
\begin{itemize}
\item 2-morphisms are indexed by finite sets of points in a disk,
\item 3-morphisms are indexed by tangles in a ball,
\item 4-morphisms between two tangles $T_1$ and $T_2$ are elements of the Khovanov--Rozansky
homology $\KhR(T_1\sqcup \overline{T_2})$ of the link obtained by reflecting $T_2$ and gluing it
with $T_1$ along their corresponding endpoints.
\end{itemize}
The various ways of composing $k$-morphisms are purely geometric for $k\leq 3$
and use certain cobordism maps between Khovanov--Rozansky homologies to define composition of $4$-morphisms.
We give two constructions of such 4-categories,
following the axioms of a disklike 4-category in \S \ref{sec:tqft}
and of a braided monoidal 2-category in \S \ref{sec:bmcat}. 

We invite the reader to
use Khovanov--Rozansky link homology to build interesting examples of 
4-categories following different axiomatizations, and to explore the appropriate 
incarnations of the sweep-around property in these settings.

\subsection*{The skein invariant}
The construction of the 4-manifold invariant $\skeinzero$ is most
straightforward when using the setting of a disklike 4-category or the related
notion of a \emph{lasagna algebra}, a 4-dimensional analog of a planar algebra
which we introduce in \S \ref{sec:tqft}. Indeed, the bigraded abelian group
$\skeinzero(W; L)$ is constructed as a skein module (inspired by the
3-dimensional analogs of Conway, Przytycki~\cite{MR1194712} and
Turaev~\cite{MR1142906}) spanned by certain decorated surfaces in $W$ bounding
$L$, which we call \emph{lasagna fillings}, modulo skein relations imposed by
the operad structure of the lasagna algebra. 

More generally, there are bigraded abelian groups $\skein_i(W; L)$
for $i\geq 1$ that arise as homology groups of the blob complex defined in
\cite{MR2978449} and can be thought of as higher derived analogs of
$\skeinzero(W; L)$. In fact, the construction of the blob complex was motivated
by the idea of using the $\skein_i(W; L)$ as tools for computing $\skeinzero(W; L)$.
We can think of $\skein_i(W; L)$ as analogous to the $i$-th Hochschild homology
$\mathrm{HH}_i$, where the input algebra of the Hochschild construction has been
replaced by the 4-category derived from $\KhR$ and the implicit circle in the
Hochschild construction has been replaced by the 4-manifold $W$. In particular,
when $W$ is the standard 4-ball, so $L$ is a link in the 3-sphere, then
$\skeinzero(B^4; L)$ is isomorphic to the usual Khovanov--Rozansky homology of
$L$, and for $i > 0$ the abelian groups are zero.

The Khovanov--Rozansky 4-categories can be fed into the general machinery of \cite{kw:tqft,MR2978449}
to produce fully extended $4{+}\epsilon$-dimensional TQFTs.
One consequence of this is that the invariants $\skeinstar$ satisfy a
gluing formula \cite[Theorem 7.2.1]{MR2978449} expressed in terms of a tensor
product over a category associated to the gluing locus. In particular, we expect that
$\skeinstar(B^3\times S^1; \{2n \text{ points}\}\times S^1)$ is related to the
Hochschild homology of the $\glN$ analog of Khovanov's arc algebra. For
applications of the latter to link homology see Rozansky~\cite{1011.1958} and
Willis~\cite{1812.06584} $(N=2)$ and
Gorsky--Hogancamp--Wedrich~\cite{Gorsky2020DerivedTO} ($N=\infty$).

We would like to emphasize that $\skeinstar$ should be thought of as a
categorified analog of the $3{+}\epsilon$-dimensional skein module TQFTs, see
Walker \cite{kw:tqft}, or the 3-dimensional layers of Crane--Yetter--Kauffman
TQFTs \cite{MR1452438} at generic $q$, but not the $2{+}1$-dimensional
Witten--Reshetikhin--Turaev TQFTs \cite{MR990772,MR1091619}.

Precise relationships along these lines and calculations based on the gluing
formula appear in later papers \cite{MN20, MWW2, HRW3}. 

\subsection*{Homotopy coherence}
Current constructions of Khovanov--Rozansky link homologies proceed via a
functorial invariant of tangles and tangle cobordisms up to isotopy, taking
values in the bounded homotopy category of an additive category; see \S
\ref{sec:tech}. Our proof of Theorem~\ref{thm:sweeparound} is stronger than
necessary in the sense that it shows that a certain equivalent reformulation of
the sweep-around move holds on the chain level (i.e. not just up to homotopy)
provided the tangle $T$ is presented as a partial braid closure.

It is an open question whether the Khovanov--Rozansky homologies are truncations
of homotopy-coherent versions with values in chain complexes over the
same additive category. If this is indeed the case, then it is plausible that
our method of proof would be suitable for an analog of
Theorem~\ref{thm:sweeparound} in this setting. Given a fully
homotopy-coherent invariant of links in $S^3$, we could construct a disklike
4-category enriched in chain complexes (rather than abelian groups), and then
use a homotopy colimit construction to extend this invariant to 4-manifolds
\cite[\S 7]{MR2978449}. The result would be a
well-defined-up-to-coherent-homotopy chain complex assigned to a 4-manifold $W$
and a boundary condition $L$. At the end of this process we could take homology
of this chain complex to produce an abelian group. The invariants $\skeinstar(W;
L)$ can be thought of as an approximation to the latter, given by taking homology
(too) early in the construction. One would then expect the two theories to be
related by a spectral sequence.

\subsection*{Genus bounds}
The results here hold for the ordinary Khovanov--Rozansky $\glN$ link
homologies as well as for their $\mathrm{GL(N)}$-equivariant and deformed
versions \cite{MR2173845, MR2232858, MR2253455, MR3392963, MR3877770}. In the
case of links in $S^3=\bd B^4$, the passage from the ordinary to deformed
settings gives rise to spectral sequences that were studied in \cite{0402266,
MR3447099, MR2509322, MR3590355}. Lobb and Wu \cite{MR2554935, MR2509322},
following pioneering work of Rasmussen \cite{MR2729272}, showed that the
associated filtrations for the generically deformed knot homologies in $S^3=\bd
B^4$ contain lower bounds on the slice genus, i.e. the minimal genus of smooth
surfaces in $B^4$ bounding the knot. Using such invariants,
Freedman--Gompf--Morrison--Walker have outlined a strategy for testing
counterexamples to the smooth 4-dimensional Poincar\'e conjecture
\cite{MR2657647}. One motivation for studying 4-manifold invariants from
Khovanov--Rozansky homologies is that analogous spectral sequences might give
rise to lower bounds on the genera of smooth surfaces in 4-manifolds $W^4$
bounding knots in $M^3=\bd W^4$.

\subsection*{Relations to other work} 
There have been several proposed approaches to constructing
homology theories for links in 3-manifolds, or 4-manifold invariants, which either 
intended to categorify $\mathfrak{sl}_2$ or $\glN$ quantum invariants or to directly generalize 
Khovanov--Rozansky homology. These include
\begin{enumerate}
	\item categorifying Witten--Reshetikhin--Turaev invariants at roots of unity, 
	see e.g. Khovanov~\cite{MR3475073}, Qi \cite{MR3164358}, Elias--Qi \cite{MR3519482} and Qi--Sussan \cite{MR3611717},
	\item using 2-representations of categorified quantum groups in the sense of 
	Rouquier \cite{0812.5023} and Khovanov--Lauda~\cite{MR2628852} to construct a 4-category that 
	can be integrated over 4-manifolds, see e.g. Webster \cite{MR3709726} for categorified tensor products,
	\item categorifying skein algebras and 3-manifold skein modules, see Asaeda--Przytycki--Sikora~\cite{MR2113902} 
	Thurston~\cite{MR3263305} and Queffelec--Wedrich~\cite{MR3934851, 1806.03416}, 
	starting from the thickened annulus, see Grigsby--Licata--Wehrli \cite{MR3731256}, 
	Beliakova--Putyra--Wehrli~\cite{MR3910068} and Queffelec--Rose~\cite{MR3729501}, or connect sums of $S^1\times S^2$, 
	see Rozansky~\cite{1011.1958} and Willis~\cite{1812.06584}.
	\item giving a mathematically rigorous construction of the BPS spectra (``relative Gromov--Witten invariants'') proposed by Gukov--Putrov--Vafa~\cite{MR3686727} 
	and	Gukov--Pei--Putrov--Vafa~\cite{2017arXiv170106567G} based on Gukov--Schwarz--Vafa~\cite{MR2193547}, 
	see e.g. Gukov--Manolescu~\cite{2019arXiv190406057G} and Ekholm--Shende~\cite{2019arXiv190108027E},
	\item extending Witten's gauge-theoretic interpretation of Khovanov homology \cite{MR2852941} 
	from $\R^3$ to other 3-manifolds, see also Taubes \cite{2013arXiv1307.6447T, 2018arXiv180502773T}.
\end{enumerate}
Comparing these approaches with the invariants defined here may be an interesting topic 
for further research.
We expect a close relationship with approach (2) already at the level of 4-categories, 
and with approach (3) since it uses the same underlying combinatorics. 
The latter is especially appealing since (3) is, on the one hand, computationally well-developed for thickened surfaces, 
but, on the other hand, poses many open questions about the categorification of skein algebras and related quantum cluster algebras,
onto which our invariants might shed new light.  

While this article was under review, Manolescu--Neithalath~\cite{MN20} have
shown that the values of $\skeinzero$ on 2-handlebodies can be computed from the
Khovanov--Rozansky homology of cables of attaching links. The procedure takes
the form of evaluating the Khovanov--Rozansky homology of the attaching link
colored by a \emph{categorical Kirby color}, a structure developed in the
prototypical case $N=2$ by Hogancamp--Rose--Wedrich in \cite{HRW3}. More
generally, the values of $\skeinzero$ can be computed for general 4-manifolds
from a handle decomposition, see Manolescu--Walker--Wedrich \cite{MWW2}.

\subsection*{Acknowledgements}
The authors would like to thank Ian Agol, Chris Douglas, Mike Freedman, Marco
Mackaay, Anton Mellit, Stephen Morgan, and Hoel Queffelec for helpful
conversations. Scott Morrison was partially supported by Australian Research
Council grants `Low dimensional categories' DP160103479 and `Quantum symmetries'
FT170100019. Paul Wedrich was supported by Australian Research Council grants
`Braid groups and higher representation theory' DP140103821 and `Low dimensional
categories' DP160103479.

\section{Technology}
\label{sec:tech}
The purpose of this section is to survey the technology used in  functorial
Khovanov--Rozansky link homologies and to set up notation.
\subsection{Webs}

The category $\Rep(U_q(\glN))$ of finite-dimensional $U_q(\glN)$-modules is a ribbon category and
thus provides Reshetikhin--Turaev invariants of framed oriented tangles with components labeled by
objects of $\Rep(U_q(\glN))$. A framed oriented link $L$ labeled by the $U_q(\glN)$-module
$V=\C^N(q)$ yields an endomorphism of $\C(q)$, the tensor unit in $\Rep(U_q(\glN))$, which is just
multiplication by the $\glN$ link polynomial of $L$.

While we will focus on invariants of links labeled by $V$, it is convenient to also consider the
fundamental modules $\bigwedge^k(V)$ and their duals. Together, these generate the full monoidal
subcategory $\Fund(U_q(\glN))$, which admits a graphical presentation and which recovers
$\Rep(U_q(\glN))$ upon idempotent completion.

The $\C(q)$-linear pivotal category $\Web$ has objects given by finite sets of
points in an interval $\I$, each labeled by an element of $\{n,n^*|n\in
\Z_{>0}\}$. The morphisms are $\C(q)$-linear combinations of webs: oriented
trivalent graphs, properly embedded in $\I^2$, with edges labeled by a
non-negative integer flow, considered up to isotopy relative to the boundary and
certain local relations, including those shown in \eqref{eq:webrel}. The source
and target of a web are determined by its intersections with $\I\times \{0\}$
and $\I\times \{1\}$, with downward oriented boundary points of label $n$ being
recorded as $n^*$. Composition is given by the bilinear extension of stacking
webs and the tensor product is given on objects by concatenating labeled
intervals and on morphisms by the bilinear extension of placing webs side by
side.

The morphisms in $\Web$ are generated under composition, tensor product, and duality by identity
morphisms and trivalent \emph{merge} and \emph{split} vertices:
\[
\begin{tikzpicture}[anchorbase,scale=.5]
	\draw [thick, dashed, opacity=0.4] (-1,2) to (1,2) (-1,0) to (1,0);
	\draw [very thick, directed=0.55] (0,0) to (0,2);
	\node at (-.4,1.1) {\tiny $a$};
\end{tikzpicture}
\quad,\quad
\begin{tikzpicture}[anchorbase,scale=.5]
	\draw [thick, dashed, opacity=0.4] (-1,2) to (1,2) (-1,0) to (1,0);
	\draw [very thick, directed=0.55] (0,1) to (0,2);
	\draw [very thick, directed=0.5] (.5,0) to [out=90,in=330] (0,1);
	\draw [very thick, directed=0.5] (-.5,0) to [out=90,in=210] (0,1);
	\node at (-.75,1.6) {\tiny $a{+}b$};
	\node at (-.75,.75) {\tiny $a$};
	\node at (.75,.75) {\tiny $b$};
\end{tikzpicture}
\quad,\quad
\begin{tikzpicture}[anchorbase,scale=.5]
	\draw [thick, dashed, opacity=0.4] (-1,2) to (1,2) (-1,0) to (1,0);
	\draw [very thick, directed=0.55] (0,0) to (0,1);
	\draw [very thick, directed=0.55] (0,1) to [out=30,in=270] (.5,2);
	\draw [very thick, directed=0.55] (0,1) to [out=150,in=270] (-.5,2);
	\node at (-.75,0.6) {\tiny $a{+}b$};
	\node at (-.75,1.45) {\tiny $a$};
	\node at (.75,1.45) {\tiny $b$};
\end{tikzpicture}\]
The merge and split vertices encode the natural $U_q(\glN)$-intertwiners \[\textstyle \bigwedge^a(V)\otimes
\bigwedge^b(V) \to \bigwedge^{a+b}(V) \quad \text{and} \quad \bigwedge^{a+b}(V) \to \bigwedge^a(V)\otimes
\bigwedge^b(V),\] respectively. The local relations in $\Web$ include
\begin{gather}\label{eq:webrel}
\begin{tikzpicture}[anchorbase,scale=.4]
	\draw [thick, dashed, opacity=0.4] (-1,4) to (1,4) (-1,0) to (1,0);
\draw [very thick, directed=0.55] (0,0) to (0,1);
\draw [very thick, directed=0.55] (0,3) to (0,4);
	\draw [very thick, directed=0.55] (0,1) to [out=150,in=210] (0,3);
	\draw [very thick, directed=0.55] (0,1) to [out=30,in=330] (0,3);
	\node at (-.5,0.7) {\tiny $a$};
	\node at (-1.3,2.2) {\tiny $a{-}b$};
	\node at (0.9,2.2) {\tiny $b$};
\end{tikzpicture}
= {\textstyle\genfrac{[}{]}{0pt}{}{a}{b}}\!
\begin{tikzpicture}[anchorbase,scale=.4]
	\draw [thick, dashed, opacity=0.4] (-1,4) to (1,4) (-1,0) to (1,0);
	\draw [very thick, directed=0.55] (0,0) to (0,4);
	\node at (-.5,0.7) {\tiny $a$};
\end{tikzpicture}
,\;\;
\begin{tikzpicture}[anchorbase,scale=.4]
	\draw [thick, dashed, opacity=0.4] (-1,4) to (1,4) (-1,0) to (1,0);
	\draw [very thick, directed=0.55] (0,0) to (0,1);
	\draw [very thick, directed=0.55] (0,3) to (0,4);
	\draw [very thick, directed=0.55] (0,1) to [out=150,in=210] (0,3);
	\draw [very thick, rdirected=0.5] (0,1) to [out=30,in=330] (0,3);
	\node at (-.5,0.7) {\tiny $a$};
	\node at (-1.3,2.2) {\tiny $a{+}b$};
	\node at (0.9,2.2) {\tiny $b$};
	\end{tikzpicture}
	= {\textstyle\genfrac{[}{]}{0pt}{}{N-a}{b}}\!\!
	\begin{tikzpicture}[anchorbase,scale=.4]
	\draw [thick, dashed, opacity=0.4] (-1,4) to (1,4) (-1,0) to (1,0);
	\draw [very thick, directed=0.55] (0,0) to (0,4);
	\node at (-.5,0.7) {\tiny $a$};
\end{tikzpicture}
,\;\;
\begin{tikzpicture}[anchorbase,scale=.4]
		\draw [thick, dashed, opacity=0.4] (-1.5,4) to (1.5,4) (-1.5,0) to (1.5,0);
	\draw [very thick, directed=0.55] (0,0) to (0,1);
	\draw [very thick, directed=0.55] (0,1) to [out=150,in=270] (-.75,2);
	\draw [very thick, directed=0.55] (-.75,2) to [out=30,in=270] (0,4);
	\draw [very thick, directed=0.55] (-.75,2) to [out=150,in=270] (-1.3,4);
	\draw [very thick, directed=0.7] (0,1) to [out=30,in=270] (1.3,4);
	\node at (-1.4,2.1) {\tiny $a$};
	\node at (0,2.1) {\tiny $b$};
	\node at (1.4,2.1) {\tiny $c$};
\end{tikzpicture}
=
\begin{tikzpicture}[anchorbase,scale=.4]
	\draw [thick, dashed, opacity=0.4] (-1.5,4) to (1.5,4) (-1.5,0) to (1.5,0);
	\draw [very thick, directed=0.55] (0,0) to (0,1);
	\draw [very thick, directed=0.55] (0,1) to [out=30,in=270] (.75,2);
	\draw [very thick, directed=0.55] (.75,2) to [out=150,in=270] (0,4);
	\draw [very thick, directed=0.55] (.75,2) to [out=30,in=270] (1.3,4);
	\draw [very thick, directed=0.7] (0,1) to [out=150,in=270] (-1.3,4);
	\node at (-1.4,2.1) {\tiny $a$};
	\node at (0,2.1) {\tiny $b$};
	\node at (1.4,2.1) {\tiny $c$};
\end{tikzpicture}
,\;\;
\begin{tikzpicture}[anchorbase,scale=.4]
	\draw [thick, dashed, opacity=0.4] (-1,2) to (1,2) (-1,-2) to (1,-2);
	\draw [very thick, directed=0.55] (1,.5) to [out=120,in=300] (-1,1);
	\draw [very thick, directed=0.55] (-1,-1) to [out=60,in=240] (1,-.5);
	\draw [very thick, directed=0.55] (1,-2) to (1,2);
	\draw [very thick, directed=0.55] (-1,-2) to (-1,2);
	\node at (-.7,-1.2) {\tiny $k$};
	\node at (0,1.2) {\tiny $r$};
	\node at (0,-.3) {\tiny $s$};
	\node at (.7,-1.2) {\tiny $l$};
\end{tikzpicture}
= {\textstyle\sum_t \genfrac{[}{]}{0pt}{}{k-l+r-s}{t}}
\begin{tikzpicture}[anchorbase,scale=.4]
	\draw [thick, dashed, opacity=0.4] (-1,2) to (1,2) (-1,-2) to (1,-2);
	\draw [very thick, directed=0.55] (1,-1) to [out=120,in=300] (-1,-.5);
	\draw [very thick, directed=0.55] (-1,.5) to [out=60,in=240] (1,1);
	\draw [very thick, directed=0.55] (1,-2) to (1,2);
	\draw [very thick, directed=0.55] (-1,-2) to (-1,2);
	\node at (-.7,-1.2) {\tiny $k$};
	\node at (0,1.2) {\tiny $s-t$};
	\node at (0,-.3) {\tiny $r-t$};
	\node at (.7,-1.2) {\tiny $l$};
\end{tikzpicture}
\end{gather}
together with the reflections of these relations in a vertical line. Edges
labeled zero are to be erased and edges labeled by negative integers or
integers greater than $N$ force the morphism to be the zero morphism. The local
relations ensure that $\Web\cong\Fund(U_q(\glN))$ as $\C(q)$-linear
pivotal categories, see \cite{MR3263166,MR3912946,MR3709658}. In the following,
we also consider an integral version of $\Web$, which is defined over $\Z[q^{\pm
1}]$ and also satisfies the relations \eqref{eq:webrel}.

\subsection{Foams}
Foams provide a framework for a combinatorial description of Khovanov--Rozansky link homologies, in
a similar way as webs are useful for the type A Reshetikhin-Turaev invariants. We will use
$\glN$-foams constructed via the combinatorial evaluation formula for closed foams due to Robert--Wagner
\cite{1702.04140}. More precisely, we will organize these $\glN$-foams into a monoidal bicategory
$\Foam$ which categorifies the integral form of $\Web$.

The graded, additive monoidal bicategory $\Foam$ has objects given by finite sets of points in $\I$,
each labeled by an element of $\{n,n^*|n\in \Z_{>0}\}$. The 1-morphisms are (formal direct sums of
formal grading shifts of) webs, properly embedded in $\I^2$ and connecting boundary points of
appropriate labels. Note that webs are not considered up to any relations in $\Foam$. The
2-morphisms are (matrices of degree zero) $\Z$-linear combinations of $\glN$-foams in $\I^3$,
considered up to isotopy relative to the boundary and certain local relations, as defined in
\cite[Section 2]{MR3877770}. The three compositions are given by (the bilinear extension of) stacking these topological objects along the three interval directions.

Foams are the natural notion of cobordisms between webs and the relations between $2$-morphisms in
$\Foam$ are chosen so that the defining \emph{web equalities} in $\Web$ can be
lifted to explicit \emph{web isomorphisms} in $\Foam$. We refer to \cite{MR3877770} for a rigorous
definition of $\glN$-foams, as well as a complete description of the relations between them, and a
survey of various flavors of $\Foam$. Here we only comment on aspects relevant to the rest of this
paper.

\begin{wrapfigure}{R}{0.25\textwidth}
	\vspace{-12pt}
	\sixfoam
	\vspace{-25pt}
  	\caption{\qquad\qquad\qquad\qquad\qquad\qquad\qquad\qquad\qquad\qquad\qquad\qquad}
  	\label{fig:foam}
  	\vspace{-15pt}
\end{wrapfigure}

Foams are represented by 2-dimensional cell complexes, such that every point has a neighborhood
either modelled on $\R^2$, three half-planes meeting in a line, or the cone on the 1-skeleton of a
tetrahedron. Such cone points are called \emph{singular vertices} of the foam. The
points on the line in the second case form a \emph{seam} of the foam, and the connected components
of the set of manifold points are called the \emph{facets} of the foam. An example of a foam with
six singular vertices is shown in Figure~\ref{fig:foam}. The facets are oriented and labeled by
positive integers. If three facets meet along a seam, then two of their labels, say $a$ and $b$, sum
to the third, $a+b$. The orientation of the seam agrees with the orientation induced by the $a$ and
$b$ facets, and disagrees with the $a+b$ facet.

Each facet of a foam in $\Foam$ admits an action of the algebra of symmetric functions $\Lambda$.
This is to say that facets may be decorated by points labeled by symmetric functions, which are
allowed to move freely on facets. A point labeled by a product $f g \in \Lambda$ may be split into
two points labeled $f$ and $g$ respectively, and a foam with a point labeled $f+g\in \Lambda$ may
be split into a sum of foams with points labeled $f$ and $g$ respectively. The $\Lambda$-actions on
adjacent facets are compatible in the sense that $f\in \Lambda$ on an $a+b$ facet may be moved
across a seam, where it distributes into $\Delta(f)\in \Lambda\otimes \Lambda$ acting on the
adjacent $a$ and $b$ facets. The degree of a foam is computed as twice the degree of the symmetric
function decoration, minus a weighted Euler characteristic, depending on facet labels.

$\Foam$ is designed to have finite-dimensional spaces of $2$-morphisms, and in particular, the
$\Lambda$-action on each $a$-facet factors through a finite-dimensional quotient, namely
$H^*(\mathrm{Gr}(\C^a\subset \C^N))$, the cohomology ring of the Grassmannian of $a$-dimensional
subspaces of $\C^N$, which is obtained as quotient of $\Lambda$ by the ideal $\langle
h_{N-a+i}|i>0\rangle$ generated by sufficiently large complete symmetric functions. In the case of a
$1$-labeled facet, the symmetric function $e_1=h_1$ is called the \emph{dot}.

\begin{example}
The algebra of decorations on a $1$-facet in $\Foam$ can be realised as the space of $2$-morphisms
$A_1\deq\Foam(\emptyset, \bigcirc^1)$ between the empty web and a $1$-labeled circle. It is spanned
by foams consisting of disks, decorated by a number $0\leq n\leq N-1$ of dots, for which we write
$X^n$. The multiplication of such foams is realised by gluing two such dotted disks onto the legs of
a pair of pants, giving $m(X^{n_1},X^{n_2})=X^{n_1+n_2}$, subject to the relation $X^{N-1+i}=0$
for $i>0$. In fact, $A_1$ is a commutative Frobenius algebra, with counit given by capping disks
off:
\begin{equation}
\label{eqn:neckcut}
\neckcutexample .
\end{equation}
Thus we have $A_1\cong \Z[X]/\langle X^N\rangle\cong H^*(\C P^{N-1})$ as commutative Frobenius
algebras, and the $1$-labeled part of $\Foam$ is nothing but the quotient of the linearised
2-dimensional oriented cobordism category by the relations in the kernel of the $(1+1)$-dimensional
TQFT corresponding to $H^*(\C P^{N-1})$. More generally, we have $A_k\deq\Foam(\emptyset,
\bigcirc^k)\cong H^*(\mathrm{Gr}(\C^a\subset \C^N)) \cong \bigwedge^k A_1$ and $\Foam$ can be
considered as the universal source for a TQFT-like functor defined on foams, which evaluates to
$A_1$ on $1$-circles and is compatible with induction and restriction between tensor products of
exterior powers of $A_1$.
\end{example}
~
\begin{remark}
\label{rem:equivariant}
There is also an \emph{equivariant} version of $\Foam$, with facet algebras given by the
$\GLN$-equivariant cohomology rings $H_{\GLN}^*(\Gr(\C^a\subset \C^N))$, defined over the base ring
$H_{\GLN}^*(\mathrm{point})$. This version is important due to its role in the proof of
functoriality of Khovanov--Rozansky homology \cite{MR3877770} and as the source of Lee-type
deformation spectral sequences \cite{MR2173845, MR3590355} and Rasmussen-type invariants
\cite{MR2729272}. Everything in this paper works, mutatis mutandis, in the equivariant framework.
\end{remark}

\subsection{Khovanov--Rozansky homology}
\label{sec:KhR}
The construction of Khovanov--Rozansky link homologies now proceeds in two steps. The first step is
a functor that sends link diagrams to chain complexes in $\Foam$ and link cobordisms to chain maps,
which depend only on the isotopy type of the cobordism up to homotopy. The second step evaluates
such a chain complex to a bigraded abelian group through a representable functor and taking
homology.

\begin{defn}
The category $\R^3\Linko$ has objects given by embedded, framed oriented links in $L \subset\R^3$,
such that the projection along the $z$-axis maps $L$ to a blackboard-framed link diagram in
$\R^2\times \{0\}  \subset \R^3$, together with an ordering of the finitely many crossings in the
diagram. The morphisms are oriented link cobordisms in $\R^3\times \I$ up to isotopy rel boundary, together
with formal crossing reordering isomorphisms.
\end{defn}

In one direction, by forgetting the condition on the projection and ignoring the crossing order, this
category is equivalent to the usual category of all embedded, framed oriented links and link cobordisms. In the other
direction, the category $\R^3\Linko$ is equivalent to the category whose objects are link diagrams and whose
morphisms are sequences of Reidemeister moves, Morse moves, planar isotopies, and formal reorderings, considered up to
Carter--Rieger--Saito movie moves \cite{MR1238875,MR1445361}.

We will now describe the construction of a functor $\CC{-}\colon \R^3\Linko \to
\Komh(\Foam)$, with target given by the bounded homotopy category of $\Foam$. In
particular, the functor sends link diagrams to certain bounded chain complexes
of webs and foams. On single, $1$-labeled crossings, it is defined as:
\begin{equation}
\label{eqn:crossings}
\CC{\begin{tikzpicture}[anchorbase,scale=.2]
	\draw [very thick, ->] (1,-2)to(1,-1.7) to [out=90,in=270] (-1,1.7) to (-1,2);
	\draw [white, line width=.15cm] (-1,-2) to (-1,-1.7) to [out=90,in=270] (1,1.7) to (1,2);
	\draw [very thick, ->] (-1,-2) to (-1,-1.7) to [out=90,in=270] (1,1.7) to (1,2);
\end{tikzpicture}
}
 \quad = \quad
q\;
\begin{tikzpicture}[anchorbase,scale=.2]
	\draw [very thick, ->] (1,-2)to(1,-1.7) to [out=90,in=315] (0,-.5) to (0,.5) to [out=45,in=270] (1,1.7) to (1,2);
	\draw[very thick] (0,.5) to (0,-.5);
	\draw [very thick, ->] (-1,-2) to(-1,-1.7) to [out=90,in=225] (0,-.5) to (0,.5) to [out=135,in=270] (-1,1.7) to (-1,2);
	\node at (-.5,0) {\tiny $2$};
\end{tikzpicture}
\to
\uwave{
\begin{tikzpicture}[anchorbase,scale=.2]
	\draw [very thick, ->] (1,-2) to (1,2);
	\draw [very thick, ->] (-1,-2) to (-1,2);
\end{tikzpicture}
}
, \qquad
\CC{\begin{tikzpicture}[anchorbase,scale=.2]
	\draw [very thick, ->] (-1,-2) to (-1,-1.7) to [out=90,in=270] (1,1.7) to (1,2);
	\draw [white, line width=.15cm] (1,-2)to(1,-1.7) to [out=90,in=270] (-1,1.7) to (-1,2);
	\draw [very thick, ->] (1,-2)to(1,-1.7) to [out=90,in=270] (-1,1.7) to (-1,2);
\end{tikzpicture}
}
 \quad = \quad
\uwave{
\begin{tikzpicture}[anchorbase,scale=.2]
	\draw [very thick, ->] (1,-2) to (1,2);
	\draw [very thick, ->] (-1,-2) to (-1,2);
\end{tikzpicture}
}
\to
q^{-1}\;
\begin{tikzpicture}[anchorbase,scale=.2]
	\draw [very thick, ->] (1,-2)to(1,-1.7) to [out=90,in=315] (0,-.5) to (0,.5) to [out=45,in=270] (1,1.7) to (1,2);
	\draw[very thick] (0,.5) to (0,-.5);
	\draw [very thick, ->] (-1,-2) to(-1,-1.7) to [out=90,in=225] (0,-.5) to (0,.5) to [out=135,in=270] (-1,1.7) to (-1,2);
	\node at (-.5,0) {\tiny $2$};
\end{tikzpicture}
 \end{equation}
The \uwave{underlined} term is placed in homological degree zero. We call the
non-identity webs that appear here \emph{thick edges}. The differentials in both
complexes are given by the combinatorially simplest foam between the two shown
webs. We call them \emph{unzip} and \emph{zip foams} respectively. The
reader be warned that the assignments in \eqref{eqn:crossings} differ from the
conventions in \cite[Fig. 46]{MR2391017} by a $q$-grading shift of magnitude
$(N-1)w$, where $w$ denotes the writhe of the diagram; the latter further differ
from the conventions in \cite[Equation (3.2)]{MR3877770} by mirroring.

A link diagram with several crossings (in a specified order) is sent to the chain complex
constructed from the formal tensor product of the crossing complexes \eqref{eqn:crossings} (in that
order) by gluing its resolutions into the link diagram in place of the original crossings.

The chain complexes associated to link diagrams which differ only by Reidemeister moves are homotopy
equivalent, see Sections~\ref{sec:RI}--\ref{sec:RIII}. Similarly, one can define chain maps for
Morse moves. However, a highly non-trivial fact is that there exists a coherent choice for such
chain maps.

\begin{thm}[{\cite{MR3877770}}]
\label{thm:ETWfunctoriality}
The construction $\CC{-}\colon \R^3\Linko \to \Komh(\Foam)$ is functorial.
\end{thm}
In fact, Theorem~\ref{thm:ETWfunctoriality} holds in much greater generality, including colored
links and the equivariant framework mentioned in Remark~\ref{rem:equivariant}. More importantly for
us, the theorem holds locally, i.e. for tangle diagrams and tangle cobordisms.

\begin{defn}\label{def:KhR} The Khovanov--Rozansky $\glN$ link homology
$\KhR\colon \R^3 \Linko \to \mathrm{gr}^{\Z\times \Z} \AbGrp$, with target
given by the category of $\Z\times \Z$-graded abelian groups and homogeneous
homomorphisms, is defined as the composition of $\CC{-}$, the representable
functor $\bigoplus_{k\in \Z} \Foam(q^{-k} \emptyset, -)$, and taking homology.
It is functorial by Theorem~\ref{thm:ETWfunctoriality}.
\end{defn}

\section{The sweep-around move}
\label{sec:sweep}
The purpose of this section is to prove Theorem~\ref{thm:sweeparound}.
\subsection{Reduction to almost braid closures}

Given a braid word $\beta$  for a braid $[\beta] \in \mathrm{Br}_{n+1}$, we can get a 1-1-tangle
diagram by taking the braid closure of the $n$ rightmost strands. We say that such 1-1-tangle
diagrams are in \emph{almost braid closure form}. From a 1-1-tangle diagram $T$, one can
obtain link diagrams $L$ and $L'$ by taking either the left- or right-handed closure of the single
open strand. These diagrams are illustrated at the top left and top right of \eqref{eqn:movie}
respectively.

We note the following straightforward extension of the Alexander theorem.

\begin{lem} Every 1-1-tangle can be isotoped into almost braid closure form.
\end{lem}

\begin{prop}\label{prop:braidreduction}
If the sweep-around map is homotopic to the identity for 1-1-tangles in almost braid
closure form, then the same is true for all 1-1-tangle diagrams.
\end{prop}
\begin{proof}
Consider an isotopy that brings the tangle diagram $T$ into almost braid closure form
$T^\prime$ and denote its image under the Khovanov invariant as $\phi$. Furthermore, let the maps
associated to the sweep-around for $T$ and $T^\prime$ be denoted by $\sw_T$ and $\sw_{T^\prime}$
respectively. Now, note that $\sw_T\simeq \phi^{-1}\circ \sw_{T^\prime} \circ \phi$ because the
underlying link cobordisms are isotopic in $\bbR^3\times \I$.
By assumption $\sw_{T^\prime}\simeq \textrm{id}_{T^\prime}$ and thus also $\sw_{T}\simeq
\textrm{id}_{T}$.
\end{proof}

\subsection{The game plan}
Fix an almost closure $T$ of a braid word $\beta$ for $[\beta] \in\mathrm{Br}_{n+1}$. We call the
right-hand closure $L$ and the left-hand closure $L'$. We consider the following movies of
intermediate diagrams and their associated chain maps between Khovanov--Rozansky complexes. 
\begin{equation}
\label{eqn:movie}
\begin{tikzcd}
\movA{.25}\arrow[r, "R1_{\pm}"]
&
\!\!\!\!\movB{.25} \arrow[r, "R2_{\pm}"]
&
\!\!\!\!\movC{.25}\arrow[r, "R3_{\pm}"]
&
\cdots \arrow[r, "R3_{\pm}"]
&
\movD{.25} \arrow[r, "R2^{-1}_{\pm}"]
&
\movE{.25} \arrow[r, "R1^{-1}_{\pm}"]
&
\movF{.25}\\[-7mm]
& \CC{L^{0}_+} \arrow[r, "R2_+"]& \CC{L^{1}_+} \arrow[r, "R3_+"]&\cdots \arrow[r, "R3_+"]& \CC{L^{x-1}_+}\arrow[r, "R2^{-1}_+"]& \CC{L^{x}_+}\arrow[dr, "R1^{-1}_+"]&  \\[-10mm]
\CC{L} \arrow[ur,"R1_+"]\arrow[swap, dr,"R1_-"]  & & & &&& \CC{L'}\\[-10mm]
 & \CC{L^{0}_-} \arrow[swap, r, "R2_-"]& \CC{L^{1}_-}\arrow[swap, r, "R3_-"]&\cdots \arrow[swap, r, "R3_-"]& \CC{L^{x-1}_-}\arrow[swap, r, "R2_-"]& \CC{L^{x}_-}\arrow[swap, ur, "R1_-^{-1}"]&.
\end{tikzcd}
\end{equation}
In the first row, the $\pm$ signs indicate the two versions of this movie, in
which the horizontal strand passes in front of ($+$) or behind ($-$) $T$. We
denote the composition along the top by $\swf$ and the composition along the bottom
by $\swb$. In either case we first see a Reidemeister I move (denoted by $R1_{\pm}$),
then a composite of Reidemeister II moves (denoted by $R2_{\pm}$), a number of
Reidemeister III moves (each denoted by $R3_{\pm}$), a composite of inverse
Reidemeister II moves ($R2^{-1}_{\pm}$), and finally an inverse Reidemeister I
move ($R1^{-1}_{\pm}$). Our goal is to show that, after making careful use of
the freedom,  described later, to choose up-to-homotopy representatives of the
chain maps for Reidemeister III moves, we have the following:

\begin{thm}\label{thm:front-is-back}
For every almost braid closure diagram $T$, the front sweep $\swf$ and the back sweep $\swb$ chain
maps constructed above are identical (not just merely homotopic).
\end{thm}
Together with Proposition~\ref{prop:braidreduction}, this will imply Theorem~\ref{thm:sweeparound}.

\begin{cor} For every almost braid closure diagram $T$, we have $\sw_T=\id_T$.
\end{cor}
\begin{proof} We have $\sw_T=(\swb)^{-1}\circ\swf=(\swb)^{-1}\circ\swb = \id_T$.
\end{proof}

The proof of Theorem~\ref{thm:front-is-back} will occupy the rest of this section.

We distinguish two types of crossings in the intermediate diagrams $L_{\pm}^{i}$. The crossings of
the moving, horizontal, strand with everything else will be called \emph{external}. The remaining
crossings were already present in $T$ and will be called \emph{internal}.

\begin{defn}
The homological grading on $\CC{L^i_{\pm}}$ splits into the sum of the \emph{internal} and
\emph{external homological gradings}, contributed by resolutions of internal and external crossings
respectively. The internal and external homological degrees of a web $W$ appearing in $\CC{L^i_{\pm}}$
will be denoted by $\gri(W)$ and $\gre(W)$ respectively.
\end{defn}

The braid word $\beta$ determines an ordering of the crossings in $T$, $L$, and $L'$, namely
\emph{from top to bottom}. This ordering also induces an ordering of the internal crossings in all
other diagrams in \eqref{eqn:movie}. The diagrams $L^0_{\pm}$ and $L^x_{\pm}$ have one additional
external crossing. The diagrams $L^{i}_{\pm}$ for $1\leq i \leq x-1$ all have $2n+1$ external
crossings, which are ordered from right to left. We will classify webs $W$ in each of these
complexes according to the resolutions that appear at the crossings. For the following, let $M$
denote the number of crossings in $T$, $L$, and $L'$.

\begin{defn}
The \emph{type} of a web $W$ in any of the complexes in \eqref{eqn:movie} is the element $\tau(W)\in
\{p,t\}^M$ that records in the $j$-th coordinate whether the $j$-th internal crossing in the
respective link diagram is resolved in a parallel way (p), or using the thick edge (t).

The \emph{offset} of a web $W$ in any of the complexes $\CC{L^{i}_{\pm}}$ is the element $o(W) \in
\{p,t\}$ that records the resolution of the leftmost external crossing.

The \emph{state} of a web $W$ in any of the complexes $\CC{L^{i}_{\pm}}$ for $1\leq i \leq x-1$ is
the element $s(W)\in  \{p,t\}^{2n}$, which records the resolutions of the $2n$ rightmost external
crossings (that is, all except the leftmost external crossing). Such a web $W$ is said to be
\emph{palindromic} if $s(W)$ is a palindrome.
\end{defn}
~
\begin{rem}
The webs $W$ in the complexes $\CC{L}$ and $\CC{L'}$ are indexed by their types $\tau(W)$. The webs
$W$ in $\CC{L^{0}_{\pm}}$, and $\CC{L^{x}_{\pm}}$ are indexed by the pairs $(\tau(W), o(W))$. The
webs $W$ in the complexes $\CC{L^{i}_{\pm}}$ for $1\leq i \leq x-1$ are indexed by the triples
$(\tau(W),o(W),s(W))$.
\end{rem}

\begin{defn}
If $i\in \{0,1,\cdots,x\}$, $\epsilon\in\{+,-\}$, $s\in \{p,t\}^{2n}$, $o\in \{p,t\}$ and
$\tau\in\{p,t\}^M$, we will use $W_{\epsilon}^i(\tau,o,s)$ or $W_{\epsilon}^i(\tau,o)$ to denote the
web in $\CC{L_{\epsilon}^i}$ with indexing data $(\tau,o,s)$ or $(\tau,o)$, as appropriate.
Analogously, we write $W(\tau)$ and $W'(\tau)$ for $\tau$-indexed webs in $\CC{L}$ and $\CC{L'}$
respectively. If the indexing data is fixed, we will sometimes omit it from the notation (e.g.
$W^i_{\pm}=W^i_{\pm}(\tau,o,s)$ and $W=W(\tau)$) and say that the webs $W^i_+$ and $W^i_-$
\emph{correspond} to each other.

If $f$ is a chain map and $V$ and $W$ are webs in the source and target complexes, then we write
$f(V,W)$ for the component of $f$ from $V$ to $W$.
\end{defn}
~
\begin{rem}
Suppose $s\in \{p,t\}^{2n}$, $o\in \{p,t\}$ and $\tau\in\{p,t\}^M$. For $1\leq i \leq x-1$ we have
$W^i_+(\tau,o,s)=W^i_-(\tau,o,s)$ as webs, and for $i\in\{0,x\}$ we have
$W^i_+(\tau,o)=W^i_-(\tau,o)$ as webs. Moreover, $\gre(W^i_+(\tau,p,s))=-\gre(W^i_-(\tau,p,s))$.
\end{rem}

\subsection{Reidemeister I moves}
\label{sec:RI}
The Reidemeister I chain maps are the following.

\[
\RIeqn
\]

Here $\mathrm{cap}$ and $\mathrm{cup}$ simply denote the cap and cup foams,
while $\mathrm{dcap}$ and $\mathrm{dcup}$ denote decorated cap and cup foams.
The decoration is by the polynomial $\sum_{a+b=N-1}{X^a Y^b}$ where $X$ denotes
the dot on the strand and $Y$ the dot on the circle; see \eqref{eqn:neckcut}.
We have only assigned notation $R1_{\pm}$ and $R1^{-1}_{\pm}$ to those
Reidemeister I chain maps that are relevant for the sweep-around move.

\begin{lem}\label{lem:R1}
The Reidemeister I chain maps $R1_{\pm}\colon \CC{L} \to \CC{L^{0}_{\pm}}$ and $R1^{-1}_{\pm}\colon
\CC{L^{x}_{\pm}} \to \CC{L'}$ preserve the internal and external homological degrees individually.
Moreover, their only non-zero components are in external homological grading zero.
\end{lem}
\begin{proof}
	This is immediate from the explicit description of these chain maps.
\end{proof}

\begin{lem}\label{lem:R1-2}
In external homological grading zero, we have $R1_- = p\circ R1_+ $ and $R1^{-1}_+ =  R1^{-1}_-\circ
p'$ where $p$ and $p'$ are chain maps of decorated identity foams such that
\[R2^{-1}_-\circ R3_-\circ  \cdots \circ R3_- \circ R2_-\circ p = p' \circ  R2^{-1}_-\circ R3_-\circ  \cdots \circ R3_- \circ R2_- \]
\end{lem}

\begin{proof}
The chain maps $p$ and $p'$ each consist of identity foams decorated by the polynomial $\sum_{a+b=N-1}{X^a Y^b}$.
\begin{minipage}[t]{0.8\textwidth}\vspace{-11mm}
	In $p$, the dots $X$ and $Y$ are placed next to the Reidemeister I crossing, as shown in the
	first picture on the right. These dots are spatially separated from the region in which the $R2$
	and $R3$ moves are taking place, so we can slide them spatially lower in the diagram, and
	timewise past all the $R2$ and $R3$ moves. At that point, shown in the second diagram on the
	right, the dots are in exactly the positions to give $p'$.
    \vspace{-2mm}
\end{minipage}
\begin{minipage}[t]{0.2\textwidth}
$\dotslide$
\vspace{-2mm}
\end{minipage}
\end{proof}

\subsection{Reidemeister II moves}
\label{sec:RII}
We will use Elias--Khovanov's Soergel calculus~\cite{MR3095655} to describe the
chain maps associated to Reidemeister II and III moves. The Soergel calculus of
type $A_{n-1}$ is a graphical incarnation of the 2-category of Soergel
bimodules, which categorifies the Hecke algebra for $S_n$. For any $N\geq 2$, it
admits a 2-functor to the monoidal subcategory of $\Foam$ of webs and foams with
$2n$ boundary components with suitable orientations, see e.g. \cite{MR2671770}.
Instead of describing these $2$-functors formally, we will just use the Soergel
calculus as \emph{shorthand} notation for foams using the following dictionary:

\begin{itemize}
\item In the $A_1$ calculus, we have only a blue object, which we will interpret as the two strand web
\[
	\begin{tikzpicture}[anchorbase,scale=.4]
		\draw[blue, line width=1mm] (0,-1) to (0,1);
	\end{tikzpicture}
	\;\;\mapsto\;\;
	\begin{tikzpicture}[anchorbase,scale=.2]
		\draw [very thick, ->] (1,-2)to(1,-1.7) to [out=90,in=315] (0,-.5) to (0,.5) to [out=45,in=270] (1,1.7) to (1,2);
		\draw[very thick] (0,.5) to (0,-.5);
		\draw [very thick, ->] (-1,-2) to(-1,-1.7) to [out=90,in=225] (0,-.5) to (0,.5) to [out=135,in=270] (-1,1.7) to (-1,2);
		\node at (-.5,0) {\tiny $2$};
	\end{tikzpicture}.
\]
\item In the $A_2$ calculus, we have red and blue objects, interpreted as three strand webs
\[
\begin{tikzpicture}[anchorbase,scale=.4]
\draw[blue, line width=1mm] (0,-1) to (0,1);
\end{tikzpicture}
\;\;\mapsto\;\;
\begin{tikzpicture}[anchorbase,scale=.2]
\draw [very thick, ->] (-3,-2)to(-3,2);
	\draw [very thick, ->] (1,-2)to(1,-1.7) to [out=90,in=315] (0,-.5) to (0,.5) to [out=45,in=270] (1,1.7) to (1,2);
	\draw[very thick] (0,.5) to (0,-.5);
	\draw [very thick, ->] (-1,-2) to(-1,-1.7) to [out=90,in=225] (0,-.5) to (0,.5) to [out=135,in=270] (-1,1.7) to (-1,2);
	\node at (-.5,0) {\tiny $2$};
\end{tikzpicture}
,\qquad
\begin{tikzpicture}[anchorbase,scale=.4]
\draw[red, line width=1mm] (0,-1) to (0,1);
\end{tikzpicture}
\;\;\mapsto\;\;
\begin{tikzpicture}[anchorbase,scale=.2]
\draw [very thick, ->] (3,-2)to(3,2);
	\draw [very thick, ->] (1,-2)to(1,-1.7) to [out=90,in=315] (0,-.5) to (0,.5) to [out=45,in=270] (1,1.7) to (1,2);
	\draw[very thick] (0,.5) to (0,-.5);
	\draw [very thick, ->] (-1,-2) to(-1,-1.7) to [out=90,in=225] (0,-.5) to (0,.5) to [out=135,in=270] (-1,1.7) to (-1,2);
	\node at (-.5,0) {\tiny $2$};
\end{tikzpicture}.
\]
\item Start dots $\Sdot{blue}{.2}{-.2}$ and end dots $\Sdot{blue}{.2}{.2}$ (in any color) 
correspond to zip and unzip foams.
\item The trivalent vertices $\trival{blue}{.2}{-.2}$ and $\trival{blue}{.2}{.2}$ correspond to 
digon creation and annihilation foams respectively.
We also use cups $\SCupc{blue}{.2}{-.2}:=\trival{blue}{.2}{-.2}\circ\Sdot{blue}{.2}{-.2}$
and caps $\SCupc{blue}{.2}{.2}:=\Sdot{blue}{.2}{.2}\circ\trival{blue}{.2}{.2}$.
\item The 6-valent vertex $\BRBv$ corresponds to the foam shown in Figure~\ref{fig:foam}.
\end{itemize}

The Reidemeister II chain maps are the following.

\begin{equation}
\label{eqn:RII}
\RIIeqn
\end{equation}

In both cases we have chosen to order the crossings from the top to the bottom.
Now we can record two observations concerning the composite (inverse)
Reidemeister II chain maps $R2_{\pm}$ and $R2^{-1}_{\pm}$.

\begin{lem}\label{lem:R2}
The chain maps $R2_{\pm}\colon\CC{L^{0}_{\pm}} \to \CC{L^{1}_{\pm}}$ and $R2^{-1}_{\pm}\colon
\CC{L^{x-1}_{\pm}} \to \CC{L^{x}_{\pm}}$ preserve the internal and external homological gradings
individually and their only non-zero components involve palindromic resolutions.
\end{lem}

\begin{lem}\label{lem:R2-2}
Let $W^0_{\pm}=W^0_{\pm}(\tau,o)$ and $W^x_{\pm}=W^x_{\pm}(\tau,o)$ be pairs of corresponding webs
in $\CC{L^0_{\pm}}$ and $\CC{L^x_{\pm}}$ respectively. Further, let $s\in\{p,t\}^{2n}$ be a
palindrome in which $t$ appears $2k$ times, and consider $W^1_{\pm}=W^1_{\pm}(\tau,o,s)$ and
$W^{x-1}_{\pm}=W^{x-1}_{\pm}(\tau,o,s)$ in $\CC{L^1_{\pm}}$ and $\CC{L^{x-1}_{\pm}}$ respectively.
Then 
\begin{align*}R2_-(W^0_-,W^1_-) &= (-1)^k R2_+(W^0_+,W^1_+)\\
R2^{-1}_+(W^{x-1}_+,W^x_+) &= (-1)^k R2^{-1}_-(W^{x-1}_-,W^x_-).
\end{align*}
\end{lem}

\begin{proof}
In a single Reidemeister II move, the identity resolution is always sent to the identity resolution
via the identity. The maps involving the resolution with two thick edges are negatives of each
other, when comparing the two types of Reidemeister II moves with fixed order of crossings as in
\eqref{eqn:RII}.
\end{proof}

\subsection{Reidemeister III moves}
\label{sec:RIII}
In \eqref{eqn:movie} we encounter four types of Reidemeister III moves. Namely, the \emph{moving
strand} can pass in front of or behind a positive or a negative crossing. In the following we show
the front and back versions alongside each other. In every case, the moving strand is the one
connecting the bottom left and top right boundary points.

In each variant of Reidemeister III, we order the crossings in each tangle from top to bottom. The
parts of the complexes with internal homological degree zero---where the internal crossing is
resolved in the parallel fashion---are highlighted in blue. The parts with internal homological
degree $\pm 1$ are highlighted in yellow.

There is a 2-dimensional space of chain maps between the two sides of each Reidemeister
III move \cite{MR2721032}. There is a 1-dimensional affine subspace of these chain maps which, given the previous
choices for Reidemeister I and II maps, provides a functorial link invariant, by Theorem
\ref{thm:ETWfunctoriality}. (Note that their proof does not rely on any particular choice of chain
maps from this subspace; any will do!) This subspace is characterised by the condition that the
component of the chain map between parallel resolutions is the identity (this condition corresponds
to the appearance of a blue highlighted $1$ in each chain map below). In the diagrams below, we
parametrise this subspace by a variable $y$; shortly we shall specialize to $y=0$.

All choices of chain map in this affine subspace are homotopic, so for many purposes this
structure can be ignored. For the present proof, however, it is quite important that we make
the most convenient choice of up-to-homotopy representative.

When the moving strand passes a positive crossing we have:

\begin{equation}
\label{eqn:RIII+}
\RIIIypeqn
\end{equation}

Next, we consider the two ways in which the moving strand may pass a negative crossing:

\begin{equation}
\label{eqn:RIII-}
\RIIIyneqn
\end{equation}

For the remainder of this paper we specialise to the choice $y=0$. (Note in particular that the
statements immediately below are not true for other choices!)

\begin{lem}
The chain maps $R3_{\pm}\colon \CC{L^{i}_{\pm}}\to \CC{L^{i+1}_{\pm}}$ in \eqref{eqn:movie} do not
decrease the external homological grading.
\end{lem}
\begin{proof}
Since chain maps are of homological degree zero, the statement is equivalent to saying that the
Reidemeister III chain maps in \eqref{eqn:movie} never \emph{increase} the \emph{internal}
homological grading. This can be verified by inspecting \eqref{eqn:RIII+} and \eqref{eqn:RIII-}. For
the reader's convenience we have highlighted the components of negative internal homological degree
in green. All other non-zero components are highlighted blue or yellow and have internal homological
degree zero because they map between the yellow and blue layers of the relevant complexes. Thus we
only need to worry about components of the chain map which are \emph{not} highlighted in the
diagrams above. With $y=0$, these components all vanish.
\end{proof}

In other words, the Reidemeister III maps are filtered with respect to the filtration determined by
the internal homological degree, which we shall call the \emph{internal filtration}.

\begin{prop}
\label{prop:R3}
The filtration-preserving components of the chain maps 
\[R3_+\colon \CC{L^{i}_{+}}\to
\CC{L^{i+1}_{+}}\quad \text{and} \quad R3_-\colon \CC{L^{i}_{-}} \to \CC{L^{i+1}_{-}}\] agree if $1\leq i<x-1$. More
precisely,  we have
\[R3_+(W^i_+,W^{i+1}_+) = R3_-(W^i_-,W^{i+1}_-)
\]
for pairs of corresponding webs $W^i_{\pm}$ in $\CC{L^{i}_{\pm}}$ and
$W^{i+1}_{\pm}$ in $\CC{L^{i+1}_{\pm}}$ with
$\gre(W^{i}_{\pm})=\gre(W^{i+1}_{\pm})$.
\end{prop}
\begin{proof} By inspecting \eqref{eqn:RIII+} and \eqref{eqn:RIII-} ---
for each of the 1+9+9+1 components of the $R3_+$ chain map, check that the corresponding component of the $R3_-$ chain map is the same (recalling $y=0$).
\end{proof}

\begin{cor}\label{cor:R3}
The filtration-preserving component of the chain maps \[R3_+\circ\cdots\circ R3_+ \colon
\CC{L^{1}_{+}}\to \CC{L^{x-1}_{+}}\quad \text{and} \quad R3_-\circ\cdots\circ R3_-\colon \CC{L^{1}_{-}} \to
\CC{L^{x-1}_{-}}\] agree. More precisely, we have
\[R3_+(W^1_+,W^{x-1}_+) = R3_-(W^1_-,W^{x-1}_-) \]
for pairs of corresponding webs $W^1_{\pm}$ in
$\CC{L^{1}_{\pm}}$ and $W^{x-1}_{\pm}$ in $\CC{L^{x-1}_{\pm}}$ with
$\gre(W^{1}_{\pm})=\gre(W^{x-1}_{\pm})$.
\end{cor}
~
\begin{rem}
The Reidemeister III chain maps shown in \eqref{eqn:RIII+} and \eqref{eqn:RIII-}, their inverses,
and four additional variations were studied by Elias--Krasner \cite{MR2721032}. Note, however, the
following differences in conventions. Their positive crossings are our negative crossings and the
crossings in their braids are ordered from bottom to top, while we order them from top to bottom.
Finally, they read Soergel diagrams from left to right, while we read them from right to left.
\end{rem}

\subsection{Proof of the sweep-around property} 
\begin{proof}[Proof of Theorem~\ref{thm:front-is-back}]

We need to show that the two chain maps $\swf$ and $\swb$ from \eqref{eqn:movie} are equal. For
this, let $W$ and $W'$ be webs in $\CC{L}$ and $\CC{L'}$ respectively. We shall compare the
components of $\swf$ and $\swb$ between $W$ and $W'$.

By Proposition~\ref{prop:R3}, the $R3_{\pm}$ maps do not decrease the external homological degree,
but by Lemmas~\ref{lem:R1} and \ref{lem:R2}, the $R1^{\pm 1}_{\pm}$ and $R2^{\pm 1}_{\pm}$ maps
preserve the external homological degree. Since $\gre(W)=\gre(W')=0$, the increasing components of
$R3_{\pm}$ do not contribute to $\swf$ or $\swb$. Now suppose that $W^1_{\pm}$ are corresponding
webs in $\CC{L^1_\pm}$ and $W^{x-1}_{\pm}$ are corresponding webs in $\CC{L^{x-1}_\pm}$ with
$\gre(W^1_{\pm})=\gre(W^{x-1}_{\pm})=0$. Then, by Corollary~\ref{cor:R3}, 
\[(R3_{+}\circ \cdots \circ R3_{+})(W^1_{+},W^{x-1}_{+})=(R3_{-}\circ \cdots \circ R3_{-})(W^{1}_{-},W^{x-1}_{-}).\]
Let us also record that if $R3_{\pm}\circ \cdots \circ R3_{\pm}$
has a non-zero component between two webs $W^1$ and $W^{x-1}$, then first $n$ digits of $t(W^1)$
agree with the first $n$ digits of $t(W^{x-1})$. (Recall that the first $n$ digits describe the
rightmost $n$ crossings, which are spatially separated from the region in which Reidemeister III
moves occur.)

Next we consider the pair of corresponding webs $W^0_{\pm}=W^0_{\pm}(s(W),p)$ in
$\CC{L^0_\pm}$, which appear in the image of $W$ under $R1_{\pm}$, and the pair
of corresponding webs $W^{x}_{\pm}=W^{x}_{\pm}(s(W'),p)$ in $\CC{L^{x}_\pm}$,
which have $W'$ as image under $R1^{-1}_{\pm}$. The components of
$R2^{-1}_{\pm}\circ R3_{\pm}\circ  \cdots \circ R3_{\pm} \circ R2_{\pm}$ between
these webs are sums over components through many possible intermediate webs
$W^1_{\pm}$ and $W^{x-1}_{\pm}$. By the previous argument, the Reidemeister III
portions of the $+$- and the $-$-version of the map agree. By
Lemma~\ref{lem:R2-2}, the Reidemeister II portions could at most cause a
sign-discrepancy. However, since the first $n$ digits of $t(W^1_{\pm}),
t(W^2_{\pm}), \ldots, t(W^{x-1}_{\pm})$ all agree, and since Reidemeister II
chain maps are zero on non-palindromic webs by Lemma~\ref{lem:R2}, there is no
sign-discrepancy. Thus, we record:
\[(R2^{-1}_{+}\circ R3_{+}\circ  \cdots \circ R3_{+} \circ R2_{+})(W^0_+,W^x_+)=(R2^{-1}_{-}\circ R3_{-}\circ  \cdots \circ R3_{-} \circ R2_{-})(W^0_-,W^x_-)\]

Finally, we use Lemma~\ref{lem:R1-2} to compute:
\begin{align*}
\swb(W,W')&=(R1^{-1}_{-}\circ R2^{-1}_{-}\circ R3_{-}\circ  \cdots \circ R3_{-} \circ R2_{-}\circ R1_{-})(W,W')\\
&=(R1^{-1}_{-}\circ R2^{-1}_{-}\circ R3_{-}\circ  \cdots \circ R3_{-} \circ R2_{-}\circ p \circ R1_{+})(W,W')\\
&=(R1^{-1}_{-}\circ p'\circ R2^{-1}_{-}\circ R3_{-}\circ  \cdots \circ R3_{-} \circ R2_{-} \circ R1_{+})(W,W')\\
&=(R1^{-1}_{+}\circ R2^{-1}_{-}\circ R3_{-}\circ  \cdots \circ R3_{-} \circ R2_{-} \circ R1_{+})(W,W')\\
&=(R1^{-1}_{+}\circ R2^{-1}_{+}\circ R3_{+}\circ  \cdots \circ R3_{+} \circ R2_{+} \circ R1_{+})(W,W')\\
&=\swf(W,W')\qedhere
\end{align*}
\end{proof}

\section{Khovanov--Rozansky homology in \texorpdfstring{$S^3$}{3-spheres}}
From now on, we will only consider framed oriented links and framed oriented link cobordisms.
Furthermore, all diffeomorphisms are oriented.

\subsection{Link homology in abstract 3-balls}
\label{sec:balls}
The purpose of this section is to define a functorial Khovanov--Rozansky link homology for links in
abstract $3$-manifolds (abstractly) diffeomorphic to $\R^3$, which is functorial under link cobordisms in
abstract $4$-manifolds diffeomorphic to $\R^3\times \I$. The framework set up in this section
could have been developed immediately after the initial construction of functorial link invariants,
but to our knowledge it has not been developed in the literature. We hope that the careful presentation
of this improvement of the invariant will be a helpful warm-up for the following section, where we
employ a very similar strategy to build invariants of links in abstract 3-spheres.

\[
\begin{Bmatrix}
\textrm{link embeddings in oriented } B\cong \R^3
\\
\textrm{link cobordisms in or.}
W \cong \R^3\times \I \textrm{ up to isotopy rel } \bd
\end{Bmatrix}
\xrightarrow{\KhR}
\begin{Bmatrix}
\textrm{bigraded abelian groups}
\\
\textrm{homogeneous homomorphisms}
\end{Bmatrix}
\]
We will call such an invariant a \emph{link homology for links in $3$-balls}.

Throughout this section, $B$ will denote a \emph{3-ball}: an oriented $3$-manifold that is
diffeomorphic to $\R^3$ via some (unspecified!) diffeomorphism.
We say a link embedding $L$ in $\R^3$ is \emph{generic} if it is in generic position with
respect to the projection along the $z$-axis to $\R^2$ and all crossings in the resulting link diagram
have distinct $y$ coordinates.
In this case, we consider the crossings as ordered from smallest to largest $y$ coordinate.
We say a link embedding $L$ in $\R^3$ is \emph{blackboard-framed} if the framing is parallel to $\R^2$.

\begin{lem}
\label{lem:connectdiffeos}
Let $L\subset B$ be a link embedded in a 3-ball. Let $\phi_0$ and $\phi_1$ be two diffeomorphisms
from $B$ to $\R^3$ such that $\phi_0(L)$ and $\phi_1(L)$ are generic. Then we have the
following:
\begin{enumerate}
\item There exists a continuous family of diffeomorphisms $\phi_t$ for $t\in \I$, such that
$\phi_t(L)$ is generic for all but finitely many $t\in \I$, at which a Reidemeister move occurs 
or the crossing height order changes.
\item Given two such families $\phi_{t,0}$ and $\phi_{t,1}$, both interpolating between $\phi_0$ and
$\phi_1$, then there exists a continuous family $\phi_{t,s}$ of diffeomorphisms interpolating
between the families $\phi_{t,0}$ and $\phi_{t,1}$, for which the parameter space $\I\times \I$ is
stratified such that:
\begin{itemize}
\item $\phi_{t,s}(L)$ is generic for $(t,s)$ in any codimension-0 stratum,
\item $\phi_{t,s}(L)$ undergoes a Reidemeister move or the crossing height order changes as $(t,s)$ crosses through a codimension-1 stratum,
\item $\phi_{t,s}(L)$ has a movie move as monodromy if $(t,s)$ loops around a codimension-2 stratum.
\end{itemize}
\end{enumerate}
\end{lem}
\begin{proof}
	These facts follow from \cite{MR1238875,MR1445361}.
\end{proof}

\begin{defn}
Let $B$ be an oriented $3$-manifold diffeomorphic to $\R^3$. We define
\[M(B)\deq\{\textrm{diffeomorphisms } \phi\colon B \to \R^3\}.\]
Given an embedded link $L\subset B$, we define the subspace
\[M(B,L)\deq\{\phi\in M(B)| \phi(L) \text{ is } z\text{-generic and
blackboard-framed} \},\] and consider the bundle $\pi\colon  T(B,L)\to M(B,L)$
of bigraded abelian groups, whose fiber at the point $\phi$ is
$\Khb(\phi(L))$.
\end{defn}

For a path $\phi_t$ in $M(B)$ between points $\phi_0,\phi_1 \in M(B,L)\subset M(B)$, we define the grading-preserving isomorphism
\[\left( \Kh(\phi_t)\colon T(B,L)_{\phi_0} \to T(B,L)_{\phi_1} \right)
\deq \left(\Khb(\phi_t(L)) \colon \Khb(\phi_0(L))\to \Khb(\phi_1(L))\right), \]
where the latter denotes the homomorphism associated to the trace of the link isotopy $\phi_t(L)$ in
$\R^3\times \I$. This is well-defined by Theorem~\ref{thm:ETWfunctoriality}, even though
for some $t$ the embeddings $\phi_t(L)$ can be highly non-generic with respect to projection in the
$z$-coordinate. Also note that while Reidemeister I moves induce $q$-grading shifts on the level of
$\KhR$, any isotopy of framed links features such moves in pairs, leading to a grading-preserving isomorphism.

\begin{lem} The parallel transport isomorphisms $\Kh(\phi_t)$ define a flat connection on $T(B,L)$.
\end{lem}
\begin{proof}
Lemma~\ref{lem:connectdiffeos} (1) implies that we have such parallel transport
maps $\Kh(\phi_t)$ between the fibers over any pair of points $\phi_0$ and
$\phi_1$ in the base. Note that $\pi_1(M(B), \phi_0) \cong \pi_1(SO(3))\cong
\Z/2\Z$ is generated by the class of the loop obtained from $\phi_0$ by rotating
through $360$ degrees around the $z$-axis, to which $\Kh$ assigns the identity
map, see also Section~\ref{sec:duality}. Lemma~\ref{lem:connectdiffeos} (2) and
Theorem~\ref{thm:ETWfunctoriality} thus imply that the parallel transport maps
between the fibers do not depend on the choice of the path $\phi_t$.
\end{proof}

\begin{defn}\label{def:KhRball} Let $L\subset B$ be a link embedded in a 3-ball.
Then we define the Khovanov--Rozansky homology of $L$ in $B$ to be
\[\KhR(B,L)\deq\Gamma_{\textrm{flat}}(T(B,L)),\]
the bigraded abelian group of flat sections of the bundle $T(B,L)$.
\end{defn}
Note that every diffeomorphism $\phi\colon B\to \R^3$ such that $\phi(L)$ is generic and blackboard-framed
induces a grading-preserving isomorphism $\KhR(B,L)\to \Khb(\phi(L))$ by evaluating sections at the point $\phi$.

~
\begin{rem} An alternative way to describe this definition is as follows.
Consider the groupoid with set of objects given by $M(B,L)$ and with morphisms
given by paths $\phi_t$ modulo isotopies $\phi_{t,s}$ as in
Lemma~\ref{lem:connectdiffeos}. Then $\KhR$ restricts to a functor from this
groupoid to bigraded abelian groups and $\KhR(B,L)$ is defined as its (co)limit.
\end{rem}

\begin{defn}
Consider a link cobordism $\Sigma\subset W$ in a 4-manifold $W$ diffeomorphic to $\R^3\times \I$.
Let $\Si\subset \Wi$ and $\So\subset \Wo$ denote the boundary links in the incoming and outgoing
boundary 3-balls of $W$. Then we define
\[\KhR(W,\Sigma) \colon \KhR(\Wi,\Si)\to \KhR(\Wo,\So)\] in two steps. First we pick a
diffeomorphism $\phi\colon W\to \R^3 \times \I$, such that $\phi_{\mathrm{out}}:= \phi|_{\Wo}\colon
\Wo \to \R^3$ and $\phi_{\mathrm{in}}:= \phi|_{\Wi}\colon \Wi \to \R^3$ are such that
$\phi_{\mathrm{in}}(\Si)$ and $\phi_{\mathrm{out}}(\So)$ are both generic and blackboard-framed.
Then we declare
$\KhR(W,\Sigma)(\eta)$, for a flat section $\eta\in \KhR(\Wi,\Si)$, to be the unique flat section of
$\KhR(\Wo,\So)$ with value:
\[\KhR(W,\Sigma)(\eta)(\phi_{\mathrm{out}})=\Khb(\phi(\Sigma))(\eta(\phi_{\mathrm{in}}))\]
\end{defn}

\begin{lem}
\label{lem:welldefball} $\KhR(W,\Sigma)$ is independent of the choices of $\phi_{\mathrm{in}}$, $\phi_{\mathrm{out}}$ and $\phi$, and thus well-defined.
\end{lem}

\begin{proof}
We first show independence of $\phi$, given a fixed choice of $\phi_{\mathrm{in}}$
and $\phi_{\mathrm{out}}$. Suppose that $\phi'\colon W\to \R^3 \times \I$ is another
diffeomorphism restricting to $\phi_{\mathrm{in}}$ and $\phi_{\mathrm{out}}$ on $\Wi$ and $\Wo$
respectively.

Lemma \ref{lem:diffeo2isotopy}, proved below, implies that the link cobordisms $\phi(\Sigma)$ and
$\phi'(\Sigma)$ are isotopic rel boundary in $\R^3\times \I$ and we have
$\Khb(\phi(\Sigma))=\Khb(\phi'(\Sigma))$ by Theorem~\ref{thm:ETWfunctoriality}.

Next we show independence of $\phi_{\mathrm{in}}$, given a fixed choice of $\phi_{\mathrm{out}}$.
Let $\phi'_{\mathrm{in}}\colon \Wi\to \R^3$ be another diffeomorphism such that
$\phi'_{\mathrm{in}}(\Si)$ is generic and blackboard-framed, and $\phi'$ another diffeomorphism $W\to \R^3 \times \I$
restricting to $\phi'_{\mathrm{in}}$ on $\Wi$ but still to $\phi_{\mathrm{out}}$ on $\Wo$. Then, by
Lemma~\ref{lem:connectdiffeos} (1), we can find a family $\phi_{\mathrm{in},t}$ connecting
$\phi_{\mathrm{in}}$ to $\phi'_{\mathrm{in}}$. By definition of parallel transport, we have:
\[\eta(\phi'_{\mathrm{in}}) = \Khb(\phi_{\mathrm{in},t}(\Si))(\eta(\phi_{\mathrm{in}}))\]
Now we obtain a new diffeomorphism $\phi'\circ \phi_{\mathrm{in},t}\colon W\to \R^3 \times \I$
and by the previous independence result and Theorem~\ref{thm:ETWfunctoriality}, we have:
\[
\Khb(\phi'(\Sigma))(\eta(\phi'_{\mathrm{in}}))
=
\Khb(\phi'(\Sigma)\circ \phi_{\mathrm{in},t}(\Si))(\eta(\phi_{\mathrm{in}}))
=
\Khb(\phi(\Sigma))(\eta(\phi_{\mathrm{in}}))
\]
Thus, the definition was independent of the choice of $\phi_{\mathrm{in}}$. An analogous argument
also establishes independence of the choice of $\phi_{\mathrm{out}}$.
\end{proof}

It remains to prove the lemma referenced above.
\begin{lem}
\label{lem:diffeo2isotopy}
Let $\Sigma \subset \R^3\times \I$ be a link cobordism
and let $f: \R^3\times \I \to \R^3\times \I$ be a
diffeomorphism which restricts to the identity in a neighborhood of the boundary $\R^3\times
\{0,1\}$. Then $\Sigma$ is isotopic rel boundary to $f(\Sigma)$.
\end{lem}

\begin{proof}
The proof would be easy if we knew that $f$ were isotopic to the identity, but
$\pi_0(\Diff^+(\R^3\times \I,\R^3\times \{0,1\}))$ is unknown. We can, however, replace $f$ with a
diffeomorphism $f'$ which is isotopic (rel boundary) to $f$, or replace $f$ with a diffeomorphism
$f'$ which coincides with $f$ in a neighborhood of $\Sigma$. In both cases, proving that $f'(\Sigma)$ is
isotopic to $\Sigma$ easily implies that $f(\Sigma)$ is isotopic to $\Sigma$.

Choose a point $p\in \R^3$ such that $p\times \I$ is disjoint from $\Sigma$. There is no obstruction to
modifying (post-composing) $f$ by an isotopy which takes $f(p\times \I)$ to $p\times \I$, so we may
assume that $f$ restricts to the identity on $p\times \I$.
(Note that this modification changes $f(\Sigma)$ as well as $f(p\times \I)$, so there is no issue
of the image of $p\times \I$ getting ``caught" on the image of $\Sigma$.)

Next consider the tangent map of $f$ along $p\times \I$. We would like to deform the tangent map to
the identity, but there is an obstruction living in $\pi_1(SO(3)) \cong \Z/2$. We can modify
(precompose) $f$ in a neighborhood of $p\times \I$ (and away from $\Sigma$) so that this obstruction
vanishes. (Specifically, let $f:[0,1]\to[0,1]$ be a smooth function such that $f(t) = 0$ for $t$
near 0 and $f(t) = 1$ for $t$ near 1. Let $\gamma: [0,1]\to SO(3)$ be a representative of the
nontrivial element of $\pi_1(SO(3))$, with $\gamma(0) = \gamma(1) = \id$. Let $B^3$ be the unit ball
in $\R^3$, and for $p\in B^3$, let $|p|$ denote the distance from $p$ to the origin. Define a
diffeomorphism of $[0,1]\times B^3$ by
\[
	(s, \: p) \mapsto (s, \: \gamma(f(s \cdot (1 - |p|)))(p)) .
\]
This diffeomorphism is the identity near $[0,1]\times \bd B^3$ and it effects a full twist on the tangent space along
$[0,1]\times\{0\}$.)

Once the above obstruction vanishes we can isotope $f$ to a map which is the identity
on a neighborhood $N$ of $p\times \I \cup \R^3\times \{0,1\}$.

Choose a family of diffeomorphisms $g_t : \R^3\times \I \to \R^3\times \I$, with $t\in \I$, such
that $g_0$ is the identity, $g_t$ restricted to $\R^3\times \{0,1\}$ is the identity for all $t$,
and $g_1(\Sigma) \subset N$. The family of surfaces $f(g_t(\Sigma))$ provides an isotopy from $f(\Sigma) =
f(g_0(\Sigma))$ to $f(g_1(\Sigma))$. But $f$ is the identity on $N$ and $g_1(\Sigma) \subset N$, so $f(g_1(\Sigma)) =
g_1(\Sigma)$. The family of surfaces $g_t(\Sigma)$ provides an isotopy from $g_1(\Sigma)$ to $g_0(\Sigma) = \Sigma$.
Composing these two isotopies provides the desired isotopy from $f(\Sigma)$ to $\Sigma$.
\end{proof}

\subsection{Link homology in abstract 3-spheres}
\label{sec:spheres}
\begin{defn} A \emph{link homology for links in $3$-spheres} is a functor
\[
\begin{Bmatrix}
\textrm{link embeddings in oriented } S\cong S^3
\\
\textrm{link cobordisms in oriented }
Y \cong S^3\times \I \textrm{ up to isotopy rel } \bd
\end{Bmatrix}
\xrightarrow{}
\begin{Bmatrix}
\textrm{bigraded abelian groups}
\\
\textrm{homogeneous homomorphisms}
\end{Bmatrix}
\]
\end{defn}

\begin{thm}\label{thm:S3} $\KhR$ extends to a link homology theory for links in $3$-spheres.
\end{thm}
The proof occupies the remainder of this subsection.

~
\begin{rem}
In the proof of Theorem~\ref{thm:S3}, we will show that the sweep-around property from
Theorem~\ref{thm:sweeparound} is sufficient to extend a link homology for links in $3$-balls to
$3$-spheres, without using any special properties of $\KhR$.
\end{rem}

\begin{defn}
Let $S$ be an oriented $3$-manifold diffeomorphic to $S^3$. For any point $p\in S\setminus L$, we
consider $L$ as a link in the $3$-ball $S\setminus \{p\}$ and denote by $\pi \colon T(S,L) \to
S\setminus L$ the bundle of bigraded abelian groups, whose fiber at the point $p\in S\setminus L$
is $\KhR(S\setminus \{p\},L)$ as defined in Definition~\ref{def:KhRball}.
\end{defn}

For any path $p_t$ in $S\setminus L$, we have that $L\times \I \subset S\times \I\setminus
\{(p_t,t)\}_{t\in \I} \cong \R^3\times \I$. By the results of \S \ref{sec:balls}, this
cobordism induces a parallel transport isomorphism \[(\KhR(p_t)\colon T(S,L)_{p_0}\to
T(S,L)_{p_1})\deq \KhR(S\times \I\setminus \{(p_t,t)\}_{t\in \I}, L\times \I)\]

\begin{lem}\label{lem:flatconnsphere} The parallel transport isomorphisms $\Kh(p_t)$ define a flat connection on $T(S,L)$.
\end{lem}
\begin{proof}
We have to show that the parallel transport isomorphisms associated to closed loops $p_t$ in
$S\setminus L$ are identity maps. Suppose first that $p_t$ is a contractible loop. Then the pair
$(S\times \I\setminus \{(p_t,t)\}_{t\in \I}, L\times \I)$ is diffeomorphic to a pair $(\R^3 \times
\I, \Sigma)$ where $\Sigma$ is isotopic to an identity link cobordism, which implies that the
parallel transport isomorphism $\KhR(p_t)$ is the identity. This also implies that the parallel
transport isomorphisms associated to isotopic paths between two points $p_0$ and $p_1$ in
$S\setminus L$ are equal. Now suppose that the loop $p_t$ is a small meridian around a component of
$L$. Then the pair $(S\times \I\setminus \{(p_t,t)\}_{t\in \I}, L\times \I)$ is diffeomorphic to a
pair $(\R^3 \times \I, \Sigma)$ where $\Sigma$ is a sweep-around cobordism as in
\eqref{eqn:sweep1}. By Theorem~\ref{thm:sweeparound}, it follows that the parallel transport
isomorphism $\KhR(p_t)$ is the identity. Since $\pi_1(S\setminus L)$ is generated by such small
meridian loops, it follows that the parallel transport isomorphism for every loop $p_t$ is the
identity.
\end{proof}

\begin{defn}
\label{def:links-in-spheres}
Let $L\subset S$ be a link embedded in a 3-sphere. Then we define the Khovanov--Rozansky homology of $L$ in $S$ to be
\[\KhR(S,L)\deq\Gamma_{\textrm{flat}}(T(S,L)),\]
the bigraded abelian group of flat sections of the bundle $T(S,L)$.
\end{defn}

Note that every point $p\in S\setminus L$ induces a grading-preserving isomorphism \[\KhR(S,L)\to \KhR(S\setminus \{p\},L)\] of evaluating sections at the point $p$.

\begin{defn}
Consider a link cobordism $\Sigma\subset W$ in a 4-manifold $W$ diffeomorphic to $S^3\times \I$.
Let $\Si\subset \Wi$ and $\So\subset \Wo$ denote the boundary links in the incoming and outgoing boundary 3-spheres of $W$.
Now we define
\[\KhR(W,\Sigma) \colon \KhR(\Wi,\Si)\to \KhR(\Wo,\So)\] by first choosing a path $p_t\subset
W\setminus \Sigma$ from $p_0\in \Wi\setminus \Si$ to $p_1\in \Wo\setminus\So$. Then we have
$W\setminus
\{(p_t,t)\}_{t\in \I} \cong \R^3 \times \I$
and we declare $\KhR(W,\Sigma)(\eta)$, for a flat
section $\eta\in \KhR(\Wi,\Si)$, to be the unique flat section of $\KhR(\Wo,\So)$ with value
\[\KhR(W,\Sigma)(\eta)(p_{\mathrm{out}})=\KhR(W\setminus
\{(p_t,t)\}_{t\in \I} ,\Sigma)(\eta(p_{\mathrm{in}}))\]
\end{defn}

\begin{lem} $\KhR(W,\Sigma)$ is independent of the choice of the path $p_t$.
\end{lem}
\begin{proof}
Let us first fix a choice of endpoints $p_0\in \Wi\setminus \Si$ and $p_1\in \Wo\setminus \So$. Then
any two choices of paths $p_t\in $ and $p'_t$ from $p_0$ to $p_1$ can be related by isotopy in
$W\setminus \Sigma$ or splicing in a little loop linking a component of $\Sigma$. As
before, isotopic paths give rise to isotopic surfaces in $\R^3\times \I$, which induce equal maps.
Similarly, in the case of a linking loop, we can choose a standard local model and then notice that
the sweep-around property from Theorem~\ref{thm:sweeparound} implies that the two paths induce the
same map. Finally, the independence from the choice of endpoints $p_0\in \Wi\setminus \Si$ and
$p_1\in \Wo\setminus \So$ follows as in the proof of Lemma~\ref{lem:welldefball}.
\end{proof}

This completes the proof of Theorem~\ref{thm:S3}.

\subsection{Monoidality}
\label{sec:monoidality}
Links in $3$-balls and their cobordisms form a symmetric monoidal category under
boundary connect sum, which is respected by $\KhR$ as we will now see.

\begin{prop} The Khovanov--Rozansky homologies $\KhR$ are lax symmetric
monoidal functors.
\end{prop}
\begin{proof} Let $L_1\in B_1$ and $L_2\in B_2$ and write $L$ for the resulting
split disjoint union in $B \deq B_1 \#_\bd B_2$. We can find a diffeomorphism
$\phi \colon B \to \R^3$ such that not only is $L$ generic and
blackboard-framed, but also the $z$-projections of the $L_1$ and $L_2$
components of $L$ are contained in disjoint disks in $\R^2$. Then, monoidality
on the chain level is manifest in the definition of $\Khb$, and we get
\begin{align*}
	\KhR(B_1,L_1)\otimes \KhR(B_2,L_2) &\cong \Khb(\phi(L_1))\otimes \Khb(\phi(L_2)) \\
&\to \Khb(\phi(L_1\sqcup L_2))  \\
&\cong \KhR(B_1 \#_\bd B_2,L_1\sqcup L_2)
\end{align*}
where the map in the second line comes from the K\"{u}nneth theorem (this
is guaranteed to be an isomorphism when working with field coefficients). The
compatibility on the level of morphisms is verified similarly.
\end{proof}

Given a finite collection of links in $3$-balls $L_i\subset B_i$, we can also define
\[\KhR(\sqcup_i B_i, \sqcup_i L_i)\deq \bigotimes_i\KhR(B_i, L_i).\] Then the proof of
the proposition implies that the boundary connect sum of $3$-balls induces
natural homomorphisms (and even isomorphisms when working with field coefficients)
\[\KhR(\sqcup_i B_i ,\sqcup_i L_i) \to \KhR(\#_\partial B_i ,\sqcup_i L_i).\]
~
\begin{rem} This monoidality property can be interpreted as saying that $\KhR$ categorifies the $\glN$ skein algebra of $\R^2$.
For more on skein algebra categorification we refer to \cite{1806.03416}.
\end{rem}

\section{A TQFT in dimensions \texorpdfstring{$4+\epsilon$}{4+epsilon}}
\label{sec:tqft}
In this and the following section we construct three alternative 4-categorical structures from Khovanov--Rozansky homology.
(The three alternatives are not essentially different; they ought to be different descriptions of the same thing.)
These are:
\begin{itemize}
\item a ``lasagna algebra'', which is a higher dimensional analog of a planar algebra, introduced here,
\item a ``disklike 4-category'', as defined in \cite{MR2978449},
\item a ``braided monoidal 2-category'', in the sense of \cite{MR1402727}.
\end{itemize}
In fact, we use the construction of a lasagna algebra as a shortcut towards building a disklike
4-category. The construction of a braided monoidal 2-category is independent, and can be read
separately. The advantage of the lasagna algebra and disklike 4-category approaches is that
they immediately provide invariants of oriented smooth 4-manifolds, valued in bigraded abelian groups.
We briefly describe
these invariants but do not explore them further.

In \S \ref{sec:bmcat}, we recast Khovanov--Rozansky homology in the more traditional framework of a
braided monoidal 2-category $\bmKhR$ with duals.
We conjecture that the sweep-around property implies that this braided monoidal 2-category
is an $SO(4)$ fixed point in the sense of Lurie \cite{MR2555928}, and consequently
leads to invariants of oriented 4-manifolds using the framework of factorization homology
(see also \cite{MR3847209,1804.07538} for related constructions one dimension down).
We do not pursue this, preferring the more direct approach to oriented 4-manifold invariants
described in this section.

As before, links and link cobordisms are assumed to be oriented and framed, and all diffeomorphisms are oriented in this section.

\subsection{An algebra for the lasagna operad}
Throughout this section we assume familiarity with planar algebras \cite{math.QA/9909027}.

\begin{figure}[ht]
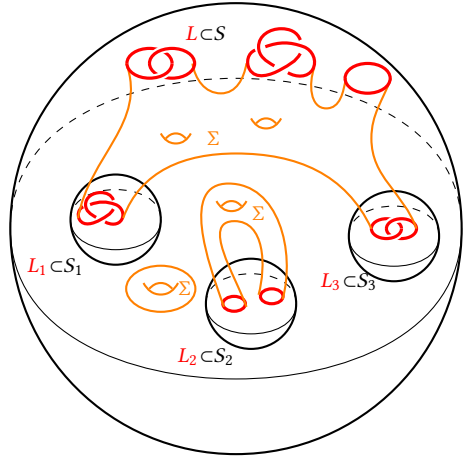

\[\lasagnafigure\]
\caption{A lasagna diagram, projected into 3d}
\label{fig:lasagna}
\end{figure}

\begin{defn}
\label{defn:lasagnaalgebra}
A \emph{lasagna algebra} $\cL$ consists of
\begin{itemize}
	\item for each link $L$ in a 3-sphere $S$, a (bigraded) abelian group $\cL(S,L)$,
	which depends functorially on the pair $(S, L)$,
	\item for each \emph{lasagna diagram} $D$,
	which, by definition, consists of a 4-ball $B$,
	with a finite collection of disjoint 4-balls $B_i$ removed from the interior,
	with boundary components $S$ (on the outside) and $S_i$ (the boundaries of the removed interior balls $B_i$),
	and properly embedded framed oriented surface $\Sigma$ in the complementary region, meeting the boundary spheres	in links
	$L$ and $L_i$ (see Figure~\ref{fig:lasagna}), a (homogeneous) homomorphism
	$$ \cL(D) : \bigotimes_i \cL(S_i,L_i) \to \cL(S,L),$$
	such that
	\item surfaces $\Sigma$ and $\Sigma'$ which are isotopic rel boundary induce identical
	homomorphisms,
	\item if $f:D \to D'$ is a diffeomorphism between lasagna diagrams, then the square
	  $$\begin{tikzcd}
		\bigotimes_i \cL(S_i,L_i) \arrow[r, "\cL(D) "] \arrow[d, "{\bigotimes_i \cL(f|_{S_i})}"]
		& \cL(S,L) \arrow[d, "{\cL(f|_{S})}"] \\
		\bigotimes_i \cL(f(S_i),f(L_i)) \arrow[r, "\cL(D')"]
		& \cL(f(S),f(L))
		\end{tikzcd}$$
	  commutes,
	\item a `radial' surface $L \times \I \subset S \times \I$ induces the identity map
	$\cL(S,L) \to \cL(S,L)$ (or more precisely, mapping cylinders of diffeomorphisms induce the same map specified
	for that diffeomorphism by functoriality),
	\item gluing of a `smaller' lasagna diagram into one of the removed balls of a `larger' lasagna
	diagram (with compatible boundaries) to obtain a single lasagna diagram is compatible with the
	corresponding composition of homomorphisms.
\end{itemize}
\end{defn}

We won't actually spell this out in detail, but one can easily extract from this definition the
notion of the lasagna operad (actually a coloured operad, with colours corresponding to links), and
that a lasagna algebra is an algebra for that operad. One can of course consider lasagna algebras
valued in symmetric monoidal categories other than (bigraded) abelian groups.

The `one input ball' part of a lasagna algebra is essentially equivalent to a
functorial invariant of links in 3-spheres: we have an abelian group for
each such link, and homomorphisms for cobordisms between them, which
compose appropriately. It is not immediately clear that any such functorial
invariant extends to a full lasagna algebra, with well-defined operations for
multiple input balls. The goal in this section is to show that this is the case
for Khovanov--Rozansky homology. In fact, our argument shows that any functorial
invariant of links and cobordisms in 3-spheres which satisfies the monoidality
property and sweep-around move extends to a lasagna algebra.

\begin{thm}
\label{thm:KhR-lasagna}
Khovanov--Rozansky homology affords the structure of a lasagna algebra.
\end{thm}
\begin{proof}
For a lasagna diagram $D$ (as in Figure~\ref{fig:lasagna}) we define a homomorphism
$$ \KhR(D) : \bigotimes_i \KhR(S_i,L_i) \to \KhR(S,L),$$ as follows. We first choose points $q_j\in
S_j\setminus L_j$ and $q\in S\setminus L$ and then a properly embedded 1-complex $T\subset B$, disjoint from
$\Sigma$, such that the underlying graph of $T$ is a tree and the endpoints of the 1-complex are
$\{q_j,q\}$. Choose a small closed tubular neighborhood $N$ of $T$, also disjoint from $Y$. The
complement of $N$ in $B \setminus \sqcup B_i$ is diffeomorphic to $\R^3\times [0,1]$ with some
embedded surface $\Sigma'$. We will view $\Sigma'$ as a bordism between two links in two copies of
$\R^3$. One copy is identified with $S^q:=S\setminus \{q\}$, which contains the link $L$. The other
copy $X$ is the remainder of the boundary, and can be expressed as the boundary connect sum of the
3-balls $S_i^{q_i}:=S_i \setminus \{q_i\}$, connected along the tree $T$. The 3-ball $X$ contains
the split disjoint union of the links $L_i$. Khovanov--Rozansky homology for links in 3-balls gives us a
map
\[
	\KhR(\Sigma') : \Kh(X, \sqcup_i L_i) \to \Kh(S^q, L)
\]
which, together with the monoidality maps from \S \ref{sec:monoidality}, specifies a map
\begin{align*}
	\KhR(D) : \bigotimes_i \KhR(S_i,L_i) & \iso \bigotimes_i \KhR(S_i^{q_i}, L_i) 
	\stackrel{T}{\to} \KhR(X, \sqcup_i L_i) 
	\to \KhR(S^q, L) 
	\iso \Kh(S, L) .
\end{align*}
Here the first and last maps are the `evaluation' isomorphisms discussed below Definition \ref{def:links-in-spheres},
and we highlight that the monoidality map depends on the tree $T$.

We must check that the overall map above does not depend on the choices of $q_i$ and $T$. This is
straightforward, so we merely sketch the argument. Isotoping the points $q_i$ does not change the
map, by the same argument that showed that $\KhR$ is well-defined for links in 3-spheres; see
\S \ref{sec:spheres}. Isotoping $T$ disjointly from $\Sigma$ clearly does not affect the map.
Changing the combinatorics of the underlying tree of $T$ can be done in such a way that $N$ varies
continuously and remains far from $\Sigma$, and so does not affect the map. Isotoping $T$ through
$\Sigma$ does not affect the map, thanks to the sweep-around property (see
Theorem~\ref{thm:sweeparound} above). Thus $\KhR(\Sigma)$ is well-defined.
\medskip

Next we must show compatibility with the operad composition. For this we consider three lasagna diagrams:
\begin{itemize}
	\item $D_1$ with output boundary $(S_1,L_1,q_1)$, with input boundaries $(S_i,L_i,q_i)_{i\in J}$, surface $\Sigma_1$, and tree $T_1$,
	\item $D_2$ with output boundary $(S_2,L_2,q_2)$, with input boundaries $(S_1,L_1,q_1)$
	along with $(S_i,L_i,q_i)_{i\in K}$, surface $\Sigma_2$, and tree $T_2$
	\item $D$, the result of gluing $D_1$ inside $D_2$, with outer boundary $(S_2,L_2,q_2)$, input boundary $(S_i,L_i,q_i)_{i\in J\cup K}$,
	surface $\Sigma=\Sigma_1\cup \Sigma_2$, and tree $T=T_1\cup T_2$.
\end{itemize}

Compatibility with the operad composition now boils down to the claim:
\begin{equation*}
\KhR(D) = \KhR(D_2) \circ (\KhR(D_1)\otimes \id)
\end{equation*}
as maps
\begin{equation*}
\left(\bigotimes_{i\in J} \KhR(S_i,L_i)\right) \otimes\left(\bigotimes_{i\in K} \KhR(S_i,L_i) \right) \to \KhR(S_2,L_2)
\end{equation*}
We compare these two homomorphisms on the level of 3-ball link homologies, that is, with respect to a fixed choice
of basepoints $q_i$ and $q$, and we suppress associators. On this level $\KhR(D)$ is determined by the homomorphism
\begin{equation}
	\label{eq:operadLHS}
\left(\bigotimes_{i\in J} \KhR(S_i^{q_i},L_i) \right) \otimes \left(\bigotimes_{i\in K} \KhR(S_i^{q_i},L_i) \right)
\xrightarrow{T}  \KhR(X, \sqcup_i L_i) 
\to \KhR(S_2^q, L_2)
\end{equation}
where $X$ denotes the boundary connect sum of the 3-balls $S_i^{q_i}$ for $i \in
J\cup K$ along the tree $T$, and the first map is provided by
lax monoidality. On the other hand, the homomorphism $\KhR(D_2) \circ
(\KhR(D_1)\otimes \id)$ is determined by the composite
\begin{align*}
	\left( \bigotimes_{i\in J} \KhR(S_i^{q_i},L_i) \right)
	\otimes \left( \bigotimes_{i\in K} \KhR(S_i^{q_i},L_i) \right)
	\xrightarrow{T_1\otimes \id} & \; \KhR(X_J, \sqcup_{i\in J}L_i) \otimes \left( \bigotimes_{i\in K} \KhR(S_i^{q_i},L_i) \right)  \\
     \xrightarrow{\KhR(\Sigma'_1)\otimes \id} & \; \KhR(S_1^{q_1}, L_1) \otimes \left( \bigotimes_{i\in K} \KhR(S_i^{q_i},L_i) \right) \\
     \xrightarrow{T_2} & \; \KhR(X_K,\sqcup_{i\in \{1\}\cup K} L_i) \\
     \xrightarrow{\KhR(\Sigma'_2)} & \; \KhR(S_2^q, L_2) .
\end{align*}
Here we write $X_J$ for the boundary connect sum of the 3-balls
$S_i^{q_i}:=S_i\setminus \{q_i\}$ for $i\in J$ that is determined by $T_1$, and
$X_K$ for the boundary connect sum of the $S_i^{q_i}$ for $i\in \{1\}\cup K$
along $T_2$. After commuting the map induced by $\Sigma'_1$ past the second
monoidality map, we arrive at
\begin{equation}
	\label{eq:operadRHS}
	\left( \bigotimes_{i\in J} \KhR(S_i^{q_i},L_i) \right)
	\otimes \left( \bigotimes_{i\in K} \KhR(S_i^{q_i},L_i) \right)
	\xrightarrow{T}  \KhR(X, \sqcup_{i\in J\cup K}L_i)  \xrightarrow{\KhR(\Sigma'_2\circ (\Sigma'_1 \cup \id))} \KhR(S_2^q, L_2) .
\end{equation}
Since the link cobordism $\Sigma'_2\circ (\Sigma'_1 \cup \id)$ is isotopic to $\Sigma$,
the functoriality of $\KhR$ implies that the maps in \eqref{eq:operadLHS} and \eqref{eq:operadRHS} are equal. This proves the claim.
\end{proof}

\subsection{Skein theory for lasagna algebras}
\label{sec:skeinzero}
In this section, we use the lasagna algebra described above to construct an invariant $\skeinzero(W; L)$
of smooth oriented 4-manifolds $W$, possibly with a link $L$ in the boundary, valued in bigraded abelian groups.
It is akin to the skein modules of 3-manifolds, which can be defined from any ribbon category,
except that everything happens one dimension higher.
The relationship between this invariant and what we are eventually after is analogous to that between $HH_0$ and $HH_*$ of an algebra.

\begin{defn}\label{def:lasagnaf} Let $W$ be a smooth oriented 4-manifold and $L\subset \bd W$ a link.
	A \lasagnaf{} $F$ of $W$ with boundary $L$ consists of the following data
	\begin{itemize}
	\item a finite collection of `small' 4-balls $B_i$ embedded in the interior of $W$;
	\item a framed oriented surface $\Sigma$ properly embedded in $X \setminus \sqcup_i B_i$, 
	meeting $\bd W$ in $L$ and
	meeting each $\bd B_i$ in a link $L_i$; and
	\item for each $i$, a homogeneous label $v_i \in \KhR(\bd B_i, L_i)$.
	\end{itemize}
\end{defn}

The bidegree of $F$ is $\deg(F):=\sum_i \deg(v_i)+(0,(1-N)\chi(\Sigma))$.
We will also consider linear combinations of \lasagnafs{} and
impose the relation that \lasagnafs{} are multilinear in the input labels $v_i$.
Thus, \lasagnafs{} of $W$ with boundary $L$ form a bigraded abelian group.

For a 4-ball $W$ with a link $L\in \bd W$, a \lasagnaf{} $F$ is equivalent to
the data of a lasagna diagram $D$ together with input labels $v_i\in\KhR(S_i,L_i)$.
In particular, we can compute the \emph{evaluation} $\KhR(F)=\KhR(D)(v_i)\in \KhR(\bd W, L)$.

\begin{defn}
	\label{def:skeinzero} Let $W$ be a smooth oriented 4-manifold and $L\subset \bd W$ a link.
	Then we define the bigraded abelian group
	\[
	\skeinzero(W; L) \deq \Z\{\text{\lasagnafs{} } F \text{ of } W \text{ with boundary } L\}/\sim
	\]
	where $\sim$ is the transitive and linear closure of the relation on \lasagnafs{} for which
	$F_1\sim F_2$ if $F_1$ has an input ball $B_1$ with label $v_1$, and $F_2$ can be obtained from $F_1$
	by replacing $B_1$ with a third \lasagnaf{} $F_3$ of a 4-ball such that $v_1=\KhR(F_3)$,
	followed by an isotopy rel boundary.
	This is illustrated in Figure~\ref{fig:lasagnafillingequiv}.
\end{defn}

\begin{figure}[ht]
	\[
	\lasagnafillingfigure{.15}
	\]
	\caption{}
	\label{fig:lasagnafillingequiv}
\end{figure}

The relation $\sim$ is homogeneous and, thus, $\skeinzero(W; L)$ is a bigraded abelian group
since the bidegree of a cobordism map $\KhR(\Sigma)$ is $(0,(1-N)\chi(\Sigma))$
and the Euler characteristic of surfaces is additive under gluing along links.

\begin{remark}
If $L$ represents a non-zero class in $H_1(W)$, then we have $\skeinzero(W;
L)=\emptyset$ since there are no compatible \lasagnafs. It would be interesting
to relate $\skeinzero$ to $\glN$ versions of the invariants of links in $S^2
\times S^1$ or $(S^2\times S^1)^{\# g}$ constructed by Rozansky \cite{1011.1958}
and Willis \cite{1812.06584} respectively, which are trivial for
homologically non-zero links.
\end{remark}

\begin{example}
\label{ex:lasagna-filling-ball}
If $W$ is a standard 4-ball with $L\subset \bd W = S^3$, then the evaluation of \lasagnafs{} induces
an isomorphism $\ev\colon \skeinzero(W;L)\cong\KhR(S^3,L)$.
In other words, the above complicated quotient yields the usual $\KhR(S^3,L)$ in this case.
\end{example}
\begin{proof}
It follows from Theorem~\ref{thm:KhR-lasagna} that equivalent \lasagnafs{} of $W$ have equal evaluation.
Thus, we get a well-defined homomorphism $\ev\colon \skeinzero(W;L)\to \KhR(S^3,L)$, which is surjective since
any homogeneous $v\in \KhR(S^3,L)$ appears as the image of a radial \lasagnaf{} $F_v$.
Similarly, if two \lasagnafs{} $F_1$ and $F_2$ have equal evaluation $v\in \KhR(S^3,L)$, then we observe $F_1\sim F_v\sim F_2$,
and so $\ev$ is injective.
\end{proof}

Having defined the skein module $\skeinzero(W; L)$, we now proceed to constructing a disklike $4$-category.
This will also lead to a more refined invariant, taking the form of a chain complex with 0-th homology $\skeinzero(W; L)$.

\subsection{A disklike 4-category}

We very briefly recall the key points of the definition of a disklike 4-category, from \S 6 of \cite{MR2978449}.
A disklike $n$-category $\cC$ consists of:
\begin{itemize}
  \item for each $0 \leq k \leq n$, a functor
  $$\cC^k : \{ \text{$k$-balls and diffeomorphisms} \} \to \Set $$
  (and we interpret $\cC^k(X)$ as the set of $k$-morphisms with shape $X$),
  \item for each $k-1$-ball $Y$ in the boundary of a $k$-ball $X$, a \emph{restriction}
  map $\cC^k(X) \to \cC^{k-1}(Y)$ (to be more careful, these restriction
  maps only need to be defined on sufficiently large subsets of $\cC^k(X)$, for example to allow for transversality issues),
  \item for each $k$-ball $X$ presented as the gluing of two $k$-balls $X_1$ and $X_2$ along a common
  $k-1$-ball $Y$ in their boundaries, a gluing map
  $$\cC^k(X_1) \times_{\cC^{k-1}(Y)} \cC^k(X_2) \to \cC^k(X),$$
  \item such that these gluing operations are compatible with the action of diffeomorphisms,
  and associative on the nose,
  \item and that two diffeomorphisms of $n$-balls which are isotopic rel boundary act identically,
  \item along with some data and axioms concerning identities which we omit here.
\end{itemize}
(As a reminder, the surprising feature of this definition is that while gluing is required to be
strictly associative, this definition actually models fully weak $n$-categories. The key point is
that we do not choose canonical models for the shape of a $k$-morphism, and it is up to `the end
user' to pick reparametrisations of glued balls back to any standard model balls that they prefer.
It is these reparametrisations that are responsible for introducing all the difficult structural
isomorphisms of most definitions. This is analogous to the idea of a Moore loop space, which has a
strictly associative composition, versus an ordinary loop space, which has a complicated higher
associator structure described by Stasheff polyhedra.)

As explained in \cite{MR2978449}, one of the primary examples of a disklike $n$-category is string
diagrams for a pivotal traditional $n$-category. This string diagram construction works just as well
for a lasagna algebra (which is essentially a pivotal 4-category with trivial 0- and 1-morphisms).
Specifically, starting from the lasagna algebra $\KhR$, we define a disklike $4$-category $\KhRcat$ as follows:
\begin{itemize}
\item For $X$ a 0-ball, we define $\KhRcat^0(X)$ to be a single-element set.
\item For $X$ a 1-ball, we define $\KhRcat^1(X)$ to be a single-element set.
\item For $X$ a 2-ball, we define $\KhRcat^2(X)$ to be the set of all configurations of finitely many framed
oriented points in $X$.
\item For $X$ a 3-ball, we define $\KhRcat^3(X)$ to be the set of all framed oriented
tangles ({\it not} up to isotopy) properly embedded in $X$.
If $c$ is a finite configuration of oriented points in $\bd X$, we define $\KhRcat^3(X; c)$ to be the set of all
oriented tangles which restrict to $c$ on $\bd X$.
\item For $X$ a 4-ball, and $L$ a link in $\bd X$, we define $\KhRcat^4(X; L)$ to be the bigraded abelian group
$\skeinzero(X; L)$ defined above, that is, all
\lasagnafs{} of $X$ which restrict to $L$ on the boundary, modulo relations described above.
Recall that by Example \ref{ex:lasagna-filling-ball} we know $\KhRcat^4(X; L) \cong \KhR(\bd X, L)$.
\end{itemize}

We define $\KhRcat^4(X; L)$ to be \lasagnafs{} modulo relations rather than simply defining it to be $\KhR(\bd X,L)$
in order to make it easier to define composition below.

We will henceforth drop superscripts and write $\KhRcat(X)$ instead of $\KhRcat^k(X)$.

\medskip

In dimensions 0 through 3, it is clear that $\KhRcat(X)$ is functorial with respect to diffeomorphisms.
In dimension 4, it is clear the diffeomorphisms act on \lasagnafs; what remains is to show
that the relations we impose are compatible with the action of diffeomorphisms. Specifically, for a
diffeomorphism $f$ we must show that if $\KhR(F) = \KhR(F')$ then $\KhR(f(F)) = \KhR(f(F'))$. This
follows from the fact that any diffeomorphism of a 4-ball (rel boundary) is isotopic to the identity
away from a small 4-ball in the interior. We can arrange that this small 4-ball is disjoint from
$\Sigma$ and the $B_i$. The argument is similar to (but simpler than) the argument given in Lemma
\ref{lem:diffeo2isotopy}. 

\medskip

We must now define gluing (composition) of morphisms.
In dimensions 0 through 3 the morphisms are purely geometric and the gluing is defined to be
the obvious geometric gluing of submanifolds.
In dimension 4, there is again an obvious geometric gluing map of \lasagnafs. 
We must show that this gluing map is compatible with the relations we impose on fillings.
This follows from the operad composition property proved in the previous section.

Finally, the (omitted above) axioms about identities require that we check that 4-ball diffeomorphisms
which are supported away from the surface $\Sigma$ act trivially.
The diffeomorphism action on lasagna fillings is just moving
submanifolds around (and, if the internal balls move, applying the $S^3$-functoriality action from the first piece of
data for a lasagna algebra to the labels),
so a diffeomorphism supported away from the surface and the internal balls does not change a lasagna filling.

\medskip

\subsection{Blob homology}

Having built a disklike 4-category we immediately obtain
an alternative description of the skein module $\skeinzero(W; L)$ for a link in the boundary of
any oriented smooth 4-manifold $W$, as first introduced in \S \ref{sec:skeinzero}.

This is the construction from \cite[\S 6.3]{MR2978449}, which describes
$\skeinzero(W; L)$ as a colimit, taken over all ways of decomposing a 4-manifold
$W$ into a gluing of closed balls (with some regularity conditions on the ways
these balls meet). For any such decomposition, we draw compatible links in the
boundaries of each of the balls (i.e. if two balls meet along some 3-manifold,
the intersections of the two links with that 3-manifold are tangles, and
identical, and a similar condition holds for any ball meeting $\bd W$). Then the
bigraded abelian group at such a decomposition is the direct sum, over the
choices of link labels, of the tensor products of the Khovanov--Rozansky
homologies of each link. The arrows in the colimit diagram are ways of
coarsening the decomposition by gluing several balls together into a single
ball. The gluing maps for a disklike 4-category provide morphisms of bigraded
abelian groups. Finally, the skein module invariant $\skeinzero(W; L)$
associated to $W$ is just the colimit of this diagram.

We will leave it as an exercise to the interested reader to verify that these two constructions actually
give the same result!

Our motivation for introducing the disklike 4-category is that the construction
of \cite[\S 6.3]{MR2978449} actually gives much more. Associated to any link $L$
in the boundary of a 4-manifold $W$, we obtain the blob complex (with
coefficients in the disklike 4-category $\KhRcat$), which we write as
$\cB_*(\KhRcat)(W; L)$. (One approach to the definition of this complex is by
replacing the colimit described above with an appropriate homotopy colimit, see
\cite[\S 7]{MR2978449}.) This has a new homological grading, unrelated to the
internal homological grading from Khovanov-Rozansky homology. The 0-th homology
of this complex recovers the bigraded abelian group $\skeinzero(W; L)$, but the
higher blob homology groups, denoted by $\skein_i(W; L)$ for $i>0$, potentially
carry further information.

Attempting any calculations of this invariant, or of its 0-th homology in either
formulation, remains beyond the scope of this paper, and developing
appropriate computational tools is an open problem for future work (e.g.
\cite{MN20}). One such tool should come from a categorification of the $\glN$
skein relation, namely the skein exact triangle for Khovanov--Rozansky chain
complexes in $\R^3$, which induces a long exact sequence on homology groups. For
a skein triple of links in the boundary of some interesting 4-manifold $W$ we
have every reason to expect that the corresponding sequence on the level of the
skein module $\skeinzero$ is no longer exact. We do, however, obtain long exact
sequences on the level of the blob complex, which give rise to a spectral
sequence that relates the skein modules $\skeinzero$ for the three links. In
fact, the study of these spectral sequences was the original motivation for the
blob complex (however ahistorical this might seem, given the publication dates).

\section{A pivotal braided monoidal 2-category}
\label{sec:bmcat}
In this section, we define a semistrict braided monoidal 2-category $\bmKhR$ in
the sense of \cite{MR1278735} (and in fact, in the stricter sense of
\cite{MR1402727}) from Khovanov--Rozansky homology. The spaces of 2-morphisms
form bigraded abelian groups.

Recall that one expects that braided monoidal 2-categories should be the same as 4-categories which are
`boring at the bottom two levels', so there is a shift by two in the dimensions of the morphisms relative
to the previous section.

The available definition of a braided monoidal 2-category has already been strictified quite a bit,
and this necessitates jumping through some hoops to even get started. Rather than defining the morphisms of
the 2-category (which would be the 3-morphisms of the corresponding 4-category) to simply be arbitrary
embedded tangles, we will need to introduce a particular combinatorial model of a tangle diagram.

\begin{defn}
The category $\TD$ of oriented tangle diagrams has objects given by finite words in the alphabet $\{\uparrow,\downarrow\}$,
including the empty word. The morphisms are \emph{admissible} words in the alphabet
$\{\textrm{cup}_i,\textrm{cap}_i,\textrm{crossing}_i,\textrm{crossing}^{-1}_i\}_{i\geq 0}$ of generating morphisms.

The \emph{realization} $r(t)$ of a morphism $t\colon A\to B$ is a tangle diagram drawn in
the square $\I \times \I$ by first placing the words $A$ and $B$ as collections of oriented tangle
endpoints on $\I\times \{0\}$ and $\I\times \{1\}$ respectively, and then constructing an oriented
tangle diagram starting from the bottom $A$ by attaching cups, caps, crossings or inverse crossings
with $i$ parallel strands to the left, as specified by the $t$. The word $t$ is defined to be
\emph{admissible} if this procedure succeeds in generating an oriented tangle diagram. We will
consider these diagrams up to individually rescaling the $x$- and $y$-coordinates in $\I \times \I$
by orientation-preserving diffeomorphisms of $\I$. As such, every morphism in $\TD$ has a unique
realization, and we say that the morphism is the \emph{Morse data} of the oriented tangle diagram.

The composition of morphisms in $\TD$ is given by concatenating lists of generating morphisms.
\end{defn}

The remainder of this section contains the definition of the semistrict braided monoidal 2-category $\bmKhR$.
We will first define this as a $2$-category and subsequently add a semistrict monoidal structure and a braiding.

\subsection{A strict 2-category}
The strict 2-category $\bmKhR$ consists of the following data:
\begin{itemize}

\item The objects are given by finite words in the alphabet $\{\uparrow,\downarrow\}$, including the empty word.

\item The 1-morphisms are \emph{admissible} words in the alphabet
\[\{\textrm{cup}_i,\textrm{cap}_i,\textrm{crossing}_i,\textrm{crossing}^{-1}_i\}_{i\geq 0}\]
of generating morphisms.

The horizontal composition of 1-morphisms is given by concatenation of words, which is strictly associative.

\item
Given a pair of $1$-morphisms $f,g\colon A\to B$, the bigraded abelian
group of $2$-morphisms from $f$ to $g$ is defined to be
\begin{equation*}
	\bmKhR(f,g):= \KhR(\Tr(r(f),r(g))):= \KhR\left( \Trfg{f}{g} \right).
\end{equation*}
Here the link diagram $\Tr(r(f),r(g))$ is constructed from the realizations
$r(f)$ and $r(g)$ by reflecting $r(g)$ in a horizontal line, reversing its
orientations, composing with $r(f)$ and closing off as shown in the figure\footnote{We omit the realizations
$r(-)$ in this and all following figures.}. Note that this is well-defined because
the Khovanov--Rozansky invariants of two link diagrams, which are planar-isotopic through link
diagrams with identical Morse data, are canonically isomorphic.

For 1-morphisms $f,g\colon A \to B$ and $k,l \colon B \to C$,
the horizontal composition of 2-morphisms $\bmKhR(f,g)\otimes \bmKhR(k,l)\to \bmKhR(fk,gl)$
is defined as the homogeneous homomorphism computed as follows:
\[
\KhR\left( \Trfg{f}{g}\right)\otimes\KhR\left( \Trfg{k}{l}\right)\to
\KhR\left( \begin{tikzpicture}[anchorbase,scale=.25]
\Trr{k}{l}{-4}{0}
\Trr{f}{g}{0}{0}
\end{tikzpicture}\right)
\to
\KhR\left( \hcomp \right)
\to
\KhR\left( \hcomptwo \right)
\]

Here we have used the monoidality map between the tensor product of Khovanov--Rozansky
homologies of two link diagrams and the homology of the split disjoint union of the diagrams, and
then cobordism maps induced by a collection of saddles and a particular type of planar isotopy.
Using functoriality of $\KhR$, it is easy to check that the horizontal composition is
associative.

Now, for 1-morphisms $f,g,h \colon A \to B$, the vertical composition of 2-morphisms
$\bmKhR(f,g)\otimes \bmKhR(g,h)\to \bmKhR(f,h)$ is defined as the homogeneous homomorphism computed
as follows:

\[
\KhR\left( \Trfg{f}{g}\right)\otimes\KhR\left( \Trfg{g}{h}\right)\to
\KhR\left( \begin{tikzpicture}[anchorbase,scale=.25]
\Trr{g}{h}{0}{6}
\Trr{f}{g}{0}{0}
\end{tikzpicture}\right)
\to
\KhR\left( \Trfg{f}{h} \right)
\]

Here, the interesting map is induced by a link cobordism which is cylindrical over the top and
bottom quarters of the link diagrams, and which can be constructed as $r(g)\times
\mathrm{half circle}$ in the middle. More explicitly, it consists of a composition of elementary
cobordisms which cancel cups with caps and positive with negative crossings in $r(g)$ and
its reflection.

The identity 2-morphism $\id_f$ on a 1-morphism $f\colon A \to B$ is defined to be the image of the unit under the homomorphism
\[\Z = \KhR(\emptyset) \to \KhR\left( \Trfg{\id_A}{\id_A} \right) \to \KhR\left( \Trfg{f}{f} \right) \]
which is induced by the link cobordism which first creates a collection of concentric circles as specified by $A$,
and then pairs of cups and caps, crossings and inverse crossings, to form $r(f)$ composed with its reflection.
It is a consequence of functoriality that the vertical composition of 2-morphisms is strictly associative and that $\id_f$
is indeed an identity 2-morphism.
Finally, a similar check establishes the interchange law that specifies the compatibility of the horizontal and vertical composition
of 2-morphisms.
\end{itemize}

\subsection{A semistrict monoidal 2-category}
Next, we show that the 2-category $\bmKhR$ admits a semistrict monoidal structure.
Following \cite[Lemma 4]{MR1402727} and \cite{MR1626844}, this consists of the following data:

\begin{enumerate}
\item The object $I=\emptyset$.
\item For any two objects $A$ and $B$, another object $A\otimes B$, which we define as the concatenation of the words $A$ and $B$.
\item For any 1-morphism $f\colon A\to A'$ and any object $B$, a 1-morphism $f\otimes B\colon A\otimes B \to A' \otimes B$,
which we define as being represented by the same word of generating morphisms as $f$.
(This has the effect of placing an identity tangle diagram to the right of $f$.)
\item For any 1-morphism $g\colon B\to B'$ and any object $A$, a 1-morphism $A\otimes g\colon A\otimes B \to A \otimes B'$,
which we define as being represented by the same word of generating morphisms as $g$,
except that all subscripts are increased by the length of the word $A$.
(This has the effect of placing an identity tangle diagram on $A$ to the left of $g$.)
\item For any object $B$ and each 2-morphism $\alpha\colon f\to f'$, a 2-morphism $\alpha\otimes B\colon f\otimes B \to f' \otimes B$,
defined as the image of $\alpha$ under the homomorphism
\[
\KhR\left( \Trfg{f}{f'}\right)
\to
\KhR\left( \begin{tikzpicture}[anchorbase,scale=.25]
\Trrr{f}{f'}{0}{0}{3.4}{1.2}
\Trrr{\id_B}{\id_B}{1.8}{1}{0}{-.8}
\end{tikzpicture} \right)
\]
which is induced by the link cobordism that is cylindrical,
except for the a collection of disks that create a collection of nested circles.
\item For any object $A$ and each 2-morphism $\beta\colon g\to g'$,
a 2-morphism $A\otimes \beta\colon A\otimes g \to A\otimes g'$,
defined as the image of $\beta$ under the homomorphism
\[
\KhR\left( \Trfg{g}{g'}\right)
\to
\KhR\left( \begin{tikzpicture}[anchorbase,scale=.25]
\Trrr{\id_A}{\id_A}{0}{0}{3.4}{1.2}
\Trrr{g}{g'}{1.8}{1}{0}{-.8}
\end{tikzpicture} \right)
\]
which is again induced by the link cobordism that is cylindrical,
except for the a collection of disks that create a collection of nested circles.
\item For any two 1-morphisms $f\colon A \to A'$, $g \colon B \to B'$, a 2-isomorphism
\[\textstyle\bigotimes_{f,g}\colon (A\otimes g)(f \otimes B') \to (f \otimes B)(A' \otimes g)\]
which we define as the image of the identity 2-morphism on $(A\otimes g)(f \otimes B')$ under the isotopy-induced homomorphism:
\begin{equation}
\label{eqn:interchanger}
\KhR\left( \begin{tikzpicture}[anchorbase,scale=.25]
\Trrrr{f}{f}{0}{0}{3.4}{2.5}{-1}{2.5}
\Trrrr{g}{g}{1.8}{1}{0}{0}{2}{0}
\end{tikzpicture} \right)
\to
\KhR\left( \begin{tikzpicture}[anchorbase,scale=.25]
\Trrrr{f}{f}{0}{0}{3.4}{2.5}{.5}{1}
\Trrrr{g}{g}{1.8}{1}{0}{0}{.5}{1.5}
\end{tikzpicture} \right)
\end{equation}
\end{enumerate}
It is straightforward to verify that with this data, $\bmKhR$ satisfies the axioms (i)-(viii)
of a semistrict monoidal 2-category as presented in \cite[Lemma 4]{MR1402727}. 
In fact, each axiom expresses equalities of 2-morphisms that are computed via Khovanov--Rozansky cobordisms maps, and
their images are equal since the relevant link cobordisms are isotopic.

\begin{remark} The definitions of the 2-morphism spaces of $\bmKhR$ and the
	composition operations are motivated by the isomorphisms \[\bmKhR(f,g)
	\cong \HCh(\Foam)(\CC{f},\CC{g}),\] where the latter denotes the cohomology
	category of the dg category of bounded chain complexes over $\Foam$ (compare
	with \cite[Proposition 3.1]{MR2496052}). Under these isomorphisms, the
	horizontal composition corresponds to stacking tangles, the tensor product
	corresponds to placing tangles side by side, and the vertical composition
	corresponds to composing homotopy classes of chain maps. In the following,
	we will take this space saving point of view when describing 2-morphisms.
	For example, we will say that the 2-morphism in \eqref{eqn:interchanger} is
	induced by the movie of tangle diagrams:

\[\begin{tikzpicture}[anchorbase,scale=.2]
\parlines{0}{-.25}{1.75}
\parlines{0}{3}{.25}
\parlines{2}{-.25}{0.25}
\parlines{2}{1.5}{1.75}
\draw (0,1.5) rectangle (1.5,3) ;
\draw (2,0) rectangle (3.5,1.5) ;
\node at (2.75,0.75) {\tiny $g$};
\node at (.75,2.25) {\tiny $f$};
\end{tikzpicture}
\to
\begin{tikzpicture}[anchorbase,scale=.2]
\parlines{0}{-.25}{0}
\parlines{0}{1.5}{1.75}
\parlines{2}{-.25}{1.75}
\parlines{2}{3}{.25}
\draw (0,0) rectangle (1.5,1.5) ;
\draw (2,1.5) rectangle (3.5,3) ;
\node at (2.75,2.25) {\tiny $g$};
\node at (.75,0.75) {\tiny $f$};
\end{tikzpicture}\]

\end{remark}

\subsection{A braided monoidal 2-category}
Finally, we equip $\bmKhR$ with the structure of a braided monoidal 2-category. This consists of the following data:
\begin{enumerate}
\item The semistrict monoidal 2-category $(\bmKhR, \otimes, I)$.
\item A pseudonatural equivalence $R\colon \otimes \to \otimes^{\textrm{op}}$, which assigns to
pairs of objects $A$ and $B$ the 1-morphism $R_{A,B}\colon A\otimes B \to B \otimes A$ given by the Morse
datum of an (oriented) braid diagram of the form
\[R_{A,B}\deq
\begin{tikzpicture}[anchorbase,scale=.2]
\fatbraid{0}{0}
\end{tikzpicture}
\]
In this intentionally asymmetric braid diagram, we see boundary points $A\otimes B$ at the bottom
and $B\otimes A$ at the top. Additionally, for a pair of 1-morphisms $f\colon A \to A'$ and $g \colon B
\to B'$, it assigns the 2-isomorphism induced by the isotopy:
\[\braidfg\]
\item
Additionally there is an invertible modification $\tilde{R}_{-|-,-}$, which associates to
triples $A,B$ and $C$ of objects the 2-isomorphisms \[\tilde{R}_{(A|B,C)}\colon (R_{A,B}\otimes C)(B\otimes
R_{A,C})\to R_{A, B\otimes C}\] which are induced by isotopies of the following type
\[
\Rtilde
\]
Similarly, the definition of a braided monoidal 2-category calls for the existence of an invertible
modification $\tilde{R}_{-,-|-}$, which, however, in the case of $\bmKhR$ is simply the identity
modification.
\end{enumerate}
Using the functoriality of Khovanov--Rozansky homology it is straightforward to check that these
data satisfy the axioms of a braided monoidal 2-category as in \cite[Definition 6]{MR1402727}.

\subsection{Duality}
\label{sec:duality}
The braided monoidal 2-category $\bmKhR$ has duals in the sense of \cite{Barrett2012evi}. This is a
slight modification of the duality proposed by \cite{MR2020556} and used by \cite{MR1686421}.
Following \cite{Barrett2012evi}, instead of three dualities we only consider two dualities $\#$ and $*$ which correspond to rotations by 
$\pi$ in two different axes.

For an object $A$, the dual object $\hdual{A}$ is obtained by reversing the word $A$ and then
exchanging orientations $\uparrow\leftrightarrow \downarrow$. On identity 1-morphisms, this
corresponds to the result of a $\pi$ rotation in a vertical line, followed by a change of
orientation. There are unit and counit 1-morphisms $i_A\colon I \to A \otimes \hdual{A}$ and
$e_A\colon \hdual{A} \otimes A \to I$ given by nested collections of cups and caps, as well as a
triangulator 2-isomorphism $T_A\colon (i_A\otimes A)(A\otimes e_A)\to A$ represented by the obvious
string-straightening isotopy. It is clear that $\hddual{A}=A$.

Every 1-morphism $f\colon A \to B$ in $\bmKhR$ has a simultaneous left and right adjoint
$\vdual{f}\colon B \to A$ which is given by the Morse data of the result of reflecting
$r(f)$ by $\pi$ in a horizontal axis and then reversing orientations (previously we have suggestively
drawn this as a reflected $f$ in figures). Further, there are unit and counit 2-morphisms $i_f\colon
\id_A \to f \vdual{f}$ and $e_f \colon \vdual{f} f \to \id_B$, which satisfy the expected identities
$(i_f  f)(f  e_f)=\id_f$ and $(\vdual{f} i_f)(e_f \vdual{f})=\id_{\vdual{f}}$. It is clear that
$\vddual{f}=f$.  

For any 2-morphism $\alpha \colon f \to g$, we denote by $\vdual{\alpha} \colon \vdual{g}\to
\vdual{f}$ the 2-morphism obtained as the image under the isomorphism
\[
\KhR\left( \Trfg{f}{g} \right)
\to
\KhR\left( \Trfg{\scalebox{1}[-1]{$g$}}{\scalebox{1}[-1]{$f$}} \right)
\]
induced by a planar anticlockwise $\pi$-rotation of the shown link diagrams. The dualities $*$ and
$\#$ satisfy a host of unsurprising compatibility relations with the tensor product and the
horizontal and vertical composition, which are consequences of the functoriality of $\KhR$. The
only non-trivial relation is that for $\alpha \in \bmKhR(f,g)$ we have $\alpha^{**}=\alpha$, which
is implicit in Definition~\ref{def:KhR}, using the fact that foams in $\Foam$ are considered up to
isotopy relative to the boundary.

\subsection{Pivotality}
In \cite{MR1686421} Mackaay introduces the notion of sphericality for monoidal 2-categories with
suitable duals. This boils down to the extra structure providing natural 2-isomorphisms between
right- and left-traces of 1-endomorphisms.
\[\Trf{f} \to \scalebox{-1}[1]{\Trf{\scalebox{-1}[1]{$f$}}}\]
For a braided monoidal 2-category with duals, such as $\bmKhR$, which is \emph{categorified ribbon}
in the sense that it admits 2-isomorphisms that provide a vertical categorification of the framed
Reidemeister I move\footnote{In contrast, the property of being \emph{spatial} in
\cite{Barrett2012evi} is a horizontal categorification of the framed Reidemeister I move.}, such
isomorphisms always exist. In fact there are two natural choices, corresponding to sliding the
closure arcs over or under the diagram for $f$:
\[
\Trf{f}
\to
\begin{tikzpicture}[anchorbase,scale=.2]
\capp{0}{9.5}{0}
\fatbraid{0}{5.5}
\parlines{2}{0}{1.5}
\draw (0,0) rectangle (1.5,1.5) ;
\node at (0.75,0.75) {\tiny $f$};
\begin{scope}[yscale=-1,xscale=1]
\capp{0}{0}{0}
\fatbraid{0}{-5.5}
\end{scope}
\end{tikzpicture}
\to
\begin{tikzpicture}[anchorbase,scale=.2]
\capp{0}{5.5}{0}
\fatbraid{0}{1.5}
\parlines{0}{0}{1.5}
\draw (2,0) rectangle (3.5,1.5) ;
\node at (2.75,0.75) {\tiny $f$};
\begin{scope}[yscale=-1,xscale=1]
\capp{0}{4}{0}
\fatbraid{0}{0}
\end{scope}
\end{tikzpicture}
\to
\scalebox{-1}[1]{\Trf{\scalebox{-1}[1]{$f$}}}
\quad, \quad
\Trf{f}
\to
\begin{tikzpicture}[anchorbase,scale=.2]
\capp{0}{9.5}{0}
\fatbraid{0}{1.5}
\parlines{2}{0}{1.5}
\draw (0,0) rectangle (1.5,1.5) ;
\node at (0.75,0.75) {\tiny $f$};
\begin{scope}[yscale=-1,xscale=1]
\capp{0}{0}{0}
\fatbraid{0}{-9.5}
\end{scope}
\end{tikzpicture}
\to
\begin{tikzpicture}[anchorbase,scale=.2]
\capp{0}{5.5}{0}
\fatbraid{0}{-4}
\parlines{0}{0}{1.5}
\draw (2,0) rectangle (3.5,1.5) ;
\node at (2.75,0.75) {\tiny $f$};
\begin{scope}[yscale=-1,xscale=1]
\capp{0}{4}{0}
\fatbraid{0}{-5.5}
\end{scope}
\end{tikzpicture}
\to
\scalebox{-1}[1]{\Trf{\scalebox{-1}[1]{$f$}}}
\]

The sweep-around property implies that these two choices produce equal 2-isomorphisms in $\bmKhR$.
(Compare \cite[Prop A.4]{MR3578212}, for an apparently analogous situation one dimension down.)

Motivated by the equivalent fact that $\KhR$ carries a well-defined action of $\Diff^+(S^3)$, we propose that
$\bmKhR$ should be a prototypical example of some future definition of a
\emph{$SO(4)$-pivotal} braided monoidal 2-category, and suggest the possibility that these are
the $SO(4)$ fixed points in the braided monoidal 2-categories with duals.

\begin{remark}
An analogous trigraded semistrict braided monoidal 2-category $\bmKhRi$ can be constructed from the
triply-graded Khovanov--Rozansky homology, which categorifies the HOMFLY-PT polynomial. This uses
the functoriality of Rouquier complexes in the homotopy categories of type A Soergel bimodules under
braid cobordisms, which has been proven by \cite{MR2721032}. The 2-category $\bmKhRi$ admits
vertical duals $*$, but it has no duality $\#$ with respect to its monoidal structure. It is an open
problem to find a categorification of the HOMFLY-PT polynomial that allows the construction of a
version of $\bmKhRi$ that admit duals, and beyond that an $SO(4)$-pivotal structure.
\end{remark}


\newcommand{\noopsort}[1]{}\def\cprime{$'$} \def\cprime{$'$} \def\cprime{$'$}

\end{document}